\documentclass{amsart}

\usepackage{tikz}
\usetikzlibrary{calc, shapes,arrows,positioning,cd}
\tikzset{>=stealth',
  head/.style = {fill = white, text=black},
  plaque/.style = {draw, rectangle, minimum size = 10mm},
  pil/.style={->,thick},
  junct/.style = {draw,circle,inner sep=0.5pt,outer sep=0pt, fill=black}
  }
 \usetikzlibrary{arrows,automata}
\usetikzlibrary{patterns,snakes}
\usetikzlibrary{shapes.geometric,calc}
\tikzstyle{arrow}=[thick, ->, >=stealth]
\tikzset{ edge/.style={->,> = latex'} }

\usepackage{mathtools}
\usepackage[margin=1.25in]{geometry}
\usepackage[enableskew]{youngtab}
\usepackage{verbatim}
\usepackage{color,amsfonts,amssymb}
\usepackage[colorinlistoftodos]{todonotes}

\usepackage{tikz}
\usetikzlibrary{cd}
\usepackage{bbm}
\usepackage{adjustbox}
\usepackage{multirow}
\usepackage{multicol}
\usepackage{ytableau}
\usepackage{float, comment,todonotes}
\usetikzlibrary{decorations.markings}

\definecolor{darkred}{rgb}{0.7,0,0}
\definecolor{darkgreen}{rgb}{0, .6, 0}
\definecolor{orange}{rgb}{1,0.4,0}
\newcommand{\defncolor}{\color{darkred}}
\newcommand{\defn}[1]{{\defncolor\emph{#1}}} % emphasis of a definition

\usepackage[pdftex,colorlinks,citecolor=blue,linkcolor=blue]{hyperref}

\theoremstyle{plain}
\newtheorem{theorem}{Theorem}[section]
\newtheorem{lemma}[theorem]{Lemma}
\newtheorem{proposition}[theorem]{Proposition}
\newtheorem{corollary}[theorem]{Corollary}

\numberwithin{equation}{section}

\theoremstyle{definition}
\newtheorem{definition}[theorem]{Definition}
\newtheorem{example}[theorem]{Example}
\newtheorem{remark}[theorem]{Remark}

\newtheorem{question}[theorem]{Question}

\newcommand{\Spl}{\mathsf{Spl}}
\newcommand{\rank}{\mathsf{rank}}

\begin{document}
\title[Dimensions of splines of degree two]{Dimension of splines on graphs in the case of degree two and smoothness one in two variables}
\author[Nazir]{Shaheen Nazir}
\address[S. Nazir]{Department of Mathematics,
Lahore University of Management Sciences, DHA, Lahore Cantt. 54792, Lahore, Pakistan}
\email{shaheen.nazir@lums.edu.pk}

\author[Schilling]{Anne Schilling}
\address[A. Schilling]{Department of Mathematics, University of California, One Shields
Avenue, Davis, CA 95616-8633, U.S.A.}
\email{anne@math.ucdavis.edu}
\urladdr{http://www.math.ucdavis.edu/\~{}anne}

\author[Tymoczko]{Julianna Tymoczko}
\address[J. Tymoczko]{Department of Mathematics and Statistics, Smith College, Northampton, MA, U.S.A}
\email{jtymoczko@smith.edu}


\makeatletter
\@namedef{subjclassname@2020}{%
  \textup{2020} Mathematics Subject Classification}
\makeatother


\begin{abstract}
Continuous spline functions are defined as piecewise polynomials on the faces of a polyhedral
complex that agree on the intersections of two faces. Splines are used in approximation
theory and numerical analysis, with applications in data interpolation, to create
smooth curves in computer graphics and to find numerical solutions to partial differential
equations. Gilbert, Tymoczko, and Viel generalized the classical splines combinatorially and 
algebraically:  a generalized spline is a vertex labeling of a graph $G$ by elements
of the ring so that the difference between the labels of any two adjacent vertices lies in the ideal generated by the 
corresponding edge label. We study the generalized splines on the planar graphs whose 
edges are labeled by two-variable polynomials of the form $(ax+by+c)^2$ and whose vertices are 
labeled by polynomials of degree at most two.  In this paper we address the upper-bound conjecture
for the dimension of degree-2 splines of smoothness 1. The dimension is
expressed in terms of the rank of the extended cycle basis matrix. We also provide a combinatorial
algorithm on graphs to compute the rank by contracting certain subgraphs.
\end{abstract}

\maketitle

%%%%%%%%%%%%%%%%%%%%%%%%%%%%%%%%%%%%%%%%%%%%%%%%%%%%%%%%%%%%%%%
\section{Introduction}

Splines refer to a wide class of functions that are used in computer graphics and design applications requiring data interpolation and smoothing. 
The term originated from flexible devices used for shipbuilding to draw smooth shapes that have certain key parameters.
In mathematics, splines are defined as piecewise polynomials on a combinatorial partition of a geometric object that agree up to some specified 
differentiability on the intersection of the top-dimensional pieces of the partition.

One major open question in the field comes from Strang, and is sometimes referred to as the \defn{upper-bound conjecture}.  
Strang asked for a general formula for the dimension of splines of degree 
at most $d$ and smoothness $r$ on triangulated regions in the plane~\cite{Strang.1973}, later conjecturing a formula in the generic 
case~\cite{Strang.1974}.  Schumaker proved a lower bound for this dimension when $r \leqslant d$ ~\cite{Schumaker.1979}, and 
produced early work on upper bounds~\cite{Schumaker.1984}.  In Schumaker's work, and also more generally, sharp upper bounds can give rise to exact formulae and are in some cases conjectured to coincide with the lower-bound formula. When the smoothness is 
$r=1$, Strang's dimension formula is
\begin{itemize}
    \item proven for degree $d \geqslant 4$~\cite{AS.1990, Hong.1991},
    \item conjectured to coincide with Schumaker's lower bound for degree $d=3$~\cite{Strang.1973,Strang.1974} and proven for 
    {\it generic triangulations} in~\cite{Billera.1988} though open for general triangulations, and
    \item proven for \textit{generic triangulations} in~\cite{Billera.1988,Wh.1991} when $d=2$.
\end{itemize}
For further details, see Lai and Schumaker's survey~\cite[Chapter 9]{LaiSchumaker.2007} and the recent paper~\cite{SSY.2020}.

This paper addresses the last, in our view most mysterious part of this open question: to develop formulas for the dimension of degree-$2$ splines 
of smoothness $1$. Our algorithm to compute the dimension of splines of degree two, thereby proving the upper-bound conjecture 
for $d = 2$ for graphs with edge labels not in special position, is summarized in Section~\ref{section.algorithm}.
We use the more general setting of \defn{generalized splines} introduced by Gilbert, Tymoczko, and Viel~\cite{GTV.2016}. Fix a commutative ring $R$ with 
identity~$1$ and a finite graph $G=(V,E)$ with vertex set $V$ and edge set $E$. If there is an edge $e\in E$ between vertex $u$ and 
vertex $v$ in $V$, we also denote this edge by $e=uv$.
An \defn{edge-labeling} for $G$ is a function $\ell \colon E \rightarrow R$.
We denote the edge-labeled graph by $(G,\ell)$.

\begin{definition}
\label{definition.splines}
A \defn{spline} on the edge-labeled graph $(G,\ell)$ is a map $p\colon V \to R$ so that for each edge $e=uv \in E$ we have 
\[
	p(u)-p(v) \textup{ is a multiple of } \ell(uv).
\]
We may also identify the map $p$ as a tuple $p \in R^{|V|}$, where the entry in position $u\in V$ is $p(u)$.
\end{definition}

The collection of all splines $\mathsf{Spl}(G,\ell)$ on $(G,\ell)$ forms a ring and an $R$-module with vertex-wise addition, multiplication, and scaling, 
see~\cite[Proposition 2.4]{GTV.2016}.  

\begin{example}\label{example.graph}
Consider the edge-labeled graph in the center of Figure~\ref{figure.spline example} with edges labeled by polynomials in $R=\mathbb{C}[x,y]$ given 
in red on the left.  One possible way to label the vertices so that $p(u)-p(v)$ is a multiple of $\ell(uv)$ is given in blue on the right.

\begin{figure}[!htb]
\begin{minipage}{0.3\textwidth}
\begin{tikzpicture}[scale=0.5]
\draw (0,0) rectangle (2,2);
\draw (2,0) rectangle (4,2);
\draw (4,2) to (6,2);
\node at (-0.6,1) {\textcolor{red}{$x^2$}};
\node at (1.2,2.4){\textcolor{red}{$y^2$}};
\node at (1.2,-0.5){\textcolor{red}{$x$}};
\node at (2.5,1){\textcolor{red}{$y$}};
\node at (3.2,2.3){\textcolor{red}{$xy$}};
\node at (3.2,-0.5){\textcolor{red}{$y^2$}};
\node at (4.5,1){\textcolor{red}{$xy$}};
\node at (5.2,2.4){\textcolor{red}{$x^2+y^2$}};
\node at (0,)[anchor=south east]{};
\node at (0,0){$\bullet$};
\node at (0,2){$\bullet$};
\node at (2,0){$\bullet$};
\node at (4,0){$\bullet$};
\node at (2,2){$\bullet$};
\node at (4,2){$\bullet$};
\node at (6,2){$\bullet$};
\end{tikzpicture}   
\end{minipage}
\begin{minipage}{0.3\textwidth}
\begin{tikzpicture}[scale=0.5]
\draw (0,0) rectangle (2,2);
\draw (2,0) rectangle (4,2);
\draw (4,2) to (6,2);
\node at (-0.6,1) {$e_1$};
\node at (1.2,2.3){$e_2$};
\node at (1.2,-0.5){$e_3$};
\node at (2.5,1){$e_4$};
\node at (3.2,2.3){$e_5$};
\node at (3.2,-0.5){$e_6$};
\node at (4.5,1){$e_7$};
\node at (5.2,2.3){$e_8$};
\node at (0,)[anchor=south east]{};
\node at (0,0){$\bullet$};
\node at (0,2){$\bullet$};
\node at (2,0){$\bullet$};
\node at (4,0){$\bullet$};
\node at (2,2){$\bullet$};
\node at (4,2){$\bullet$};
\node at (6,2){$\bullet$};
\end{tikzpicture}   
\end{minipage}
\begin{minipage}{0.3\textwidth}
\begin{tikzpicture}[scale=0.5]
\draw (0,0) rectangle (2,2);
\draw (2,0) rectangle (4,2);
\draw (4,2) to (6,2);
\node at (0,-0.5) {\textcolor{blue}{$0$}};
\node at (-0.2,2.2){\textcolor{blue}{$x^2$}};
\node at (1.7,-0.5){\textcolor{blue}{$x^2$}};
\node at (6.8,2){\textcolor{blue}{$2xy$}};
\node at (1.9,2.5){\textcolor{blue}{$x^2 \hspace{-0.35em} +\hspace{-0.1em}y^2$}};
\node at (4.5,2.5){\textcolor{blue}{$(x\hspace{-.25em}+\hspace{-0.25em}y)^2$}};
\node at (4.5,-.5){\textcolor{blue}{$x^2+y^2$}};
\node at (0,)[anchor=south east]{};
\node at (0,0){$\bullet$};
\node at (0,2){$\bullet$};
\node at (2,0){$\bullet$};
\node at (4,0){$\bullet$};
\node at (2,2){$\bullet$};
\node at (4,2){$\bullet$};
\node at (6,2){$\bullet$};
\end{tikzpicture}   
\end{minipage}
\caption{Example of spline given in Example~\ref{example.graph}. Center: Graph $G=(V,E)$
with edges labeled $e_1,\ldots,e_8$. Left: Edge-labeled graph $(G,\ell)$ with edge labels indicated
in red. Right: A spline with the function $p$ on the vertices indicated in blue.
\label{figure.spline example}}
\end{figure}

Note that all vertex-labelings that satisfy Definition~\ref{definition.splines} also have the property that the difference of any polynomials added around 
a cycle must be zero.  For instance, going over edges $e_1, e_2, e_4, e_3$ adds $x^2$, then $y^2$, then $-y(y)$, then $-x(x)$ respectively, for a net 
change of $0$ from the entry in the bottom-left corner.
\end{example}

The theory of generalized splines unifies questions across a wide range of fields.  When the coefficient ring $R$ is the integers, generalized splines extend 
classical questions in number theory like the Chinese Remainder Theorem.  When $R$ is a complex polynomial ring, generalized splines appear in 
torus-equivariant cohomology. Goresky, Kottwitz, and MacPherson~\cite{GKM.1998} identified topological conditions to encode the $T$-action on certain 
varieties $X$ by a labeled graph $G_X$ whose vertices are $T$-fixed points, edges are one-dimensional $T$-orbits, and edge-labels $\ell_X$ are $T$-weights,
and for which
\[
	H^*_T(X) \cong \textup{ splines on } (G_X,\ell_X). 
\]

Most important for us is the fact that Definition~\ref{definition.splines} of splines is dual (in an algebraic and combinatorial sense) to the classical definition of 
splines~\cite[Theorem 2.4]{Billera.1991}, for a broad family of triangulations in the plane.  We use the coefficient ring $R = \mathbb{C}[x,y]$ and consider the
graph $G$ which is dual as a planar graph to the triangulation. If an edge of the original triangulation is on the line $ax+by+c=0$, then its dual edge in $G$ is 
labeled $(ax+by+c)^{r+1}$, where $r$ is the desired smoothness of the splines.  In our case, we further reduce to the case that each edge label has the form 
$\ell(uv) = (x+ay)^2$. By translating the original triangulation,  we can assume this about any individual face without loss of generality. This is shown
in Section~\ref{section.homogeneous} and uses the notion of ``face-translatability'' (see Definition~\ref{definition.face translate}). Other than these assumptions 
on the edge labels, we do not use specific properties of the graph as a dual graph to a triangulation. For example, we do not use that the interior
vertices of $G$ are trivalent. Our main constraints on $G$ are that $G$ is a finite, planar graph without holes and two faces share at most one edge. 

In this paper, we provide a combinatorial algorithm and formula to compute the dimension of splines for degree $d=2$ and smoothness $r=1$.
We transform this problem into an equivalent linear algebra problem of analyzing the rank of a matrix $M^{\mathsf{ext}}$ built from the \defn{face cycle basis} 
of $G$ (see Definition~\ref{definition.extended cycle basis matrix} and Theorem~\ref{theorem.spline rank}). We then focus on the question of when this 
matrix is full rank. We show that this question has three related but separate parts: a \textbf{combinatorial} one asking  whether the underlying graph has three 
distinct Euler cycles (a condition that can be slightly weakened); and an \textbf{algebraic} one asking which 
symmetries a particular determinant has, related to the determinant being symmetric in a subset of variables. 
The combinatorial question concerns the existence of \defn{edge-injective function} on the graph (see Section~\ref{section.edge injective} for the definition
and Section~\ref{section.existence} for a condition of the existence of edge-injective functions).
The algebraic question gives rise to the definition of \defn{generic edge labels} (see Section~\ref{section.generic} for the definition and
Section~\ref{section.existence generic} for conditions of existence of generic edge labels). The third ingredient is of \textbf{topological} nature.
We identify components in the edge-labeled graph which can be contracted without changing the dimension 
of the spline space even in the non-generic case.

Our algorithm to determine the dimension of the degree two splines is based on a condition called \defn{contractibility} (see 
Section~\ref{section.contractible}) under which the problem reduces to smaller cases. 
More concretely, our algorithm contracts \defn{minimal contractible subgraphs}, which can be identified combinatorially using edge-injective functions.
Contracting minimal contractible subgraphs does not change the dimension of the splines for generic edge labels.
Hence graphs without contractible subgraphs form the ``fundamental building blocks" for which the dimension 
of the spline space is the number of edges minus three times 
the number of faces (see Section~\ref{section.algorithm} for more details). 

We emphasize that this method can be used to compute the dimension even for what we refer to as \defn{non-generic edge-labeled graphs}, 
which are graphs for which the determinant of the extended cycle basis matrix is identically zero as a polynomial in the edge label variables. Indeed, if a 
graph has non-generic edge labels, then our contraction algorithm successively reduces to a smaller generic edge labeled graph without 
changing the dimension of the spline space.

Moreover, in Theorem~\ref{theorem: special position formula}, we give an explicit presentation of the polynomials defining the special locus of edge labels.  
In other words, we describe a set of polynomials in the edge labels that vanish if and only if the dimension of the space of splines is larger than expected.  
Given an edge-labeled graph, these polynomials can be used to test whether or not the graph is generic.

We list some related open problems and potentially-fruitful extensions of this work. Much of our work assumes that the edge labels are not in special 
position. It would be interesting to see whether this condition can be relaxed. Furthermore, the splines we consider are over the polynomial ring 
$\mathbb{C}[x,y]$ and adapt easily to $\mathbb{R}[x,y]$, but we expect interesting results with other coefficient rings, especially rings with nonzero 
characteristic or polynomial rings in more variables.  It would also be interesting to extend the analysis in this paper to degree 3 splines.

%%%%%%%%%%%%%%%%%%%%%%%%%%%%%%%%%%%%%%%%%%%%%%%%%%%%%%%%%%%%%%%
\subsection*{Planar triangulations}
The case when the edge labeled graph $G_{\Delta}$ is dual to a planar triangulation $\Delta$ is of considerable interest.  For this reason, we conclude this introduction with a summary of our main results specialized to this case.  We create the dual graph $G_{\Delta}$ by considering $\Delta$ as a planar graph, that is, with one vertex of $G_{\Delta}$ for each bounded face of $\Delta$ and an edge between vertices in $G_{\Delta}$ if the corresponding faces in $\Delta$ share a (necessarily interior) edge. If an (interior) edge in $\Delta$ lies on the line given by the equation $ax+by+c=0$, then label the corresponding edge in $G_{\Delta}$ with $(ax+by+c)^2$.  Figures~\ref{figure: homogeneous and particular example} and~\ref{figure.3squares}  show examples of dual graphs. With this labeling, we have the following relationship between classical splines and the construction we use:
\[\Spl(G_{\Delta}, \ell) \cong S^1(\Delta).\] 
More generally, increasing the exponent on the edge label increases the differentiability constraint $r$ in the spline space $S^r(\Delta)$. (See also Remark~\ref{remark: spline notation for analysts}.)  We restrict to splines of degree at most two, and further disregard constant splines, so our dimension formulas correspond to classical formulas via
\[
	\dim \; \Spl_2(G_{\Delta}, \ell) = \dim \; S^1_2(\Delta)-6.
\]
In Section~\ref{section.homogeneous}, we homogenize each edge label by simply omitting the constant, creating a new labeling $\ell_0$.  In other words, we replace $(ax+by+c)^2$ by $(ax+by)^2$ on each edge label, in essence translating each interior vertex of $\Delta$ to the origin simultaneously.  This is different from most kinds of homogenization in the literature, and we know of no natural geometric interpretation (see Remark~\ref{remark: homogenization contrast}).  At the same time, the main goal of Section~\ref{section.homogeneous} is to show that 
\[\Spl_2(G_{\Delta}, \ell) \cong \Spl_2(G_{\Delta}, \ell_0)\]
for all graphs $G_{\Delta}$ arising as duals to planar triangulations (as well as many other graphs).  

In Section~\ref{section.contractible}, we define and analyze a natural  topological condition called contractibility.  In the case of graphs $G_{\Delta}$, 
Corollary~\ref{corollary: contractible triangulation means subdivided triangle} proves that a contractible subgraph of $G_{\Delta}$ is dual precisely 
to a subdivided triangle in $\Delta$, which is the special case Whiteley uses~\cite[Section 3]{Wh.unpublished}.  Finally, 
Section~\ref{section.existence generic} uses purely graph-theoretic arguments to compute the dimension of splines for $\Spl_2(G_{\Delta}, \ell_0)$ 
with no contractible subgraphs in the generic case, giving Whiteley's result \cite{Wh.unpublished} as a special case. Moreover, the polynomials we 
give in Theorem~\ref{theorem: special position formula} explicitly define conditions to tell when $\Delta$ is non-generic.

%%%%%%%%%%%%%%%%%%%%%%%%%%%%%%%%%%%%%%%%%%%%%%%%%%%%%%%%%%%%%%%
\section*{Acknowledgements}
We thank Susanna Fishel, Pamela Harris, Rosa Orellana, and Stephanie van Willigenburg for organizing the 
Research Community in Algebraic Combinatorics at ICERM in 2021-2022. This project started as part of this program by Zoom.
We would also like to thank the Simons Laufer Mathematical Sciences Institute (SLMath, formerly MSRI) for supporting us to attend
the 2023 Summer Research in Mathematics (SRiM), where this paper was completed in person!
This material is based upon work supported by the National Science Foundation under Grant No. DMS-1928930, while the authors were
in residence at the Mathematical Sciences Research Institute in Berkeley, California.

The authors would like to thank Peter Alfeld, Michael DiPasquale, Tanya Sorokina, Nelly Villamizar, and Walter Whiteley for helpful discussions.

AS was partially supported by NSF grants DMS--1760329 and DMS--2053350, and Simons Foundation grant MPS-TSM-00007191.
JT was partially supported by NSF grants DMS--2349088, DMS--2054513, and DMS--1800773, and by an AWM-MERP fellowship.

%%%%%%%%%%%%%%%%%%%%%%%%%%%%%%%%%%%%%%%%%%%%%%%%%%%%%%%%%%%%%%%
\section{Planar geometry and the cycle basis matrix for splines}
\label{section.geometry}

In this section, we provide a formal definition of splines of degree $d$ and use the cycle basis of a planar graph to construct an extended cycle 
basis matrix that captures all conditions defining splines of degree $2$ on the graph.  

%%%%%%%%%%%%%%%%%%%%%%%%%%%%%%%%%%%%%%%%%%%%%%%%%%%%%%%%%%%%%%%
\subsection{Splines of degree $2$}

Splines over a polynomial ring inherit a notion of degree.  We define the degree of a spline here, as well as several other notions related to this 
particular case.

\begin{definition}
If $R$ is a polynomial ring, then the \defn{degree} of the spline $p \in \Spl(G,\ell)$ is defined as
\[
	\deg(p) = \max_{v \in V(G)} \deg(p(v)).
\]
We denote the set of splines of degree at most $d$ by $\Spl_{\leqslant d}(G,\ell)$ and splines of degree exactly $d$ by $\Spl_d(G,\ell)$.
\end{definition}

Note that if $R$ is a polynomial ring over $\mathbb{C}$, $\mathbb{R}$, or any other field of characteristic zero, then the degree of a sum of 
splines is at most the maximum of the degree of each spline.  This means that the collection of splines of degree at most $d$ forms a vector space 
over the base field (and an $R$-submodule).  However, we lose the ring structure of splines when we restrict to a specific degree.

\begin{remark} \label{remark: spline notation for analysts}
Notational conventions for analysts differ from those for algebraists.  In classical splines, the notation $S^r_d(\Delta)$ denotes splines of 
degree at most $d$, which corresponds to what we call $\Spl_{\leqslant d}(G,\ell)$ here.  In algebra, the ring $R$ is \defn{graded} if it decomposes 
$R = \bigoplus_d R_d$ into $R$-submodules $R_d$ such that $r_i \in R_i, r_j \in R_j$ means that $r_ir_j \in R_{i+j}$.  In this case, the $R_d$ are 
called the \defn{homogeneous} graded submodules. Our notation $\Spl_d(G,\ell)$ agrees with the algebraic conventions.
\end{remark}

Under reasonable assumptions, the splines $\Spl(G,\ell)$ form a graded ring. For a graph $G=(V,E)$, we also denote by $V(G)=V$ its vertex 
set and by $E(G)=E$ its edge set if we want to emphasize the dependence on the graph.
\begin{definition}
 A \defn{homogeneous} polynomial is one whose (nonzero) monomial terms all have the same degree. A \defn{homogeneous spline of degree $d$} 
is one for which $p(v)$ is either zero or a homogeneous polynomial of degree $d$ for every $v\in V(G)$.  An edge-labeling $\ell \colon E(G) \rightarrow R$ 
is \defn{homogeneous} if $\ell(e)$ is a homogeneous polynomial for all $e \in E(G)$.
\end{definition}

The following is proven in~\cite{AMT.2021}.

\begin{proposition}
Suppose that $(G,\ell)$ is an edge-labeled graph. If $\ell(e)$ is a homogeneous polynomial for each edge $e \in E(G)$, then $\Spl(G,\ell)$ is a
graded ring with homogeneous graded parts $\Spl_d(G,\ell)$.
\end{proposition}

For every edge-labeled graph, the splines $\Spl(G,\ell)$ have a trivial and freely-generated submodule consisting of \defn{constant splines}, 
namely the splines $p \in \Spl(G,\ell)$ with $p(u)=p(v)$ for all vertices $u, v \in V(G)$.  Our formulas count the dimension of the complement 
of the constant splines.  In other words, we study the following space.

\begin{definition}
Let $R$ be a polynomial ring and denote the homogeneous polynomials of degree $d$ by $R_d$.  Suppose that $\Spl(G,\ell)$ is graded by degree 
of the splines, that $1 \in \Spl(G,\ell)$ denotes the (multiplicative) identity spline, and that $v_0$ is a fixed vertex in $G$.  We write 
\[
	\Spl_d(G,\ell) = R_d 1 \oplus \Spl_d(G,\ell; v_0),
\]
where $\Spl_d(G,\ell; v_0)$ is the set of splines with a fixed value at vertex $v_0$.
\end{definition}

In this paper, we determine the dimension of $\Spl_2(G,\ell; v_0)$, which differs by one from the dimension of $\Spl_2(G,\ell)$. Note also that if
one requires $p(v_0)=0$, then this is equivalent to restricting to non-constant splines.

For most partitions $\Delta \subseteq \mathbb{R}^n$ of interest in applications, Billera and Rose~\cite{billera1991dimension} constructed a related partition 
$\hat{\Delta} \subseteq \mathbb{R}^{n+1}$ over $\Delta$ that satisfies
\[
	S^r_d(\hat{\Delta}) \cong S^r_d(\Delta).
\]
(This partition $\hat{\Delta}$ is the cone over $\Delta$ and is defined by adding a variable $x_{n+1}$ and homogenizing the edge-labels as in 
algebraic geometry, namely multiplying each term by the appropriate power of $x_{n+1}$ to make each edge-label a homogeneous polynomial.)

We will homogenize $\Spl(G_{\Delta},\ell_{\Delta})$ differently as shown in Section~\ref{section.homogeneous}.

%%%%%%%%%%%%%%%%%%%%%%%%%%%%%%%%%%%%%%%%%%%%%%%%%%%%%%%%%%%%%%%
\subsection{Constructing the cycle basis matrix}

Consider a finite, connected, planar graph $G=(V,E)$. A \defn{cycle} in $G$ is a sequence of vertices $v_1 v_2 \ldots v_k v_1$ such that
$v_i\in V$, $v_i v_{i+1} \in E$ for $1\leqslant i <k$, and $v_k v_1 \in E$.

If $G$ is embedded into the plane, each face of the embedding is bounded by a cycle of edges. MacLane~\cite{MacLane.1937} showed that a 
graph is planar if and only if the graph contains a complete set of cycles such that no edge appears in more than two cycles.

When a planar graph is drawn without any crossing edges, it divides the plane into a set of regions, called \defn{faces}. Each face is bounded 
by a cycle called the boundary of the face or \defn{face cycle}. By convention, we count the unbounded area outside the whole graph as one face. 
It follows from Euler's formula for planar graphs that the set of face cycles is linearly independent. Hence the set of face cycles in a planar
graph forms a \defn{cycle basis} (see~\cite{Diestel.2010, Diestel.2011, Diestel.2012}).

Assume that the edges $E$ are directed.  Choose an order on the set of edges $E=\{e_1,e_2,\ldots,e_{e_G}\}$ and an order on the set of bounded 
faces $\mathcal{F} = \{F_1,F_2,\ldots,F_{f_G}\}$ of $G$. Define the \defn{cycle basis matrix} $M$ to be the matrix with $f_G$ rows and $e_G$ columns 
such that the entry in position $(r,c)$, that is in row $r$ and column $c$, is 
\begin{itemize}
    \item $0$ if edge $e_c$ is not on the face cycle of $F_r$, 
    \item $1$ if edge $e_c$ is an edge on the face cycle of $F_r$ and points clockwise around $F_r$,
    \item $-1$ if edge $e_c$ is an edge on the face cycle of  $F_r$ and points counterclockwise around $F_r$. 
\end{itemize}

\begin{remark}
    MacLane showed that each column of $M$ contains at most two nonzero entries~\cite{MacLane.1937}. 
\end{remark}

\begin{example}
\label{example.M}
Consider the planar graph on the left below.
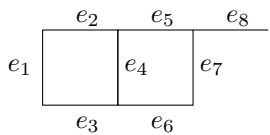
\begin{center}
\begin{tikzpicture}[scale=0.5]
\draw[thick] (0,0) rectangle (2,2);
\draw[thick] (2,0) rectangle (4,2);
\draw[thick] (4,2) to (6,2);
\node at (-0.6,1) {$e_1$};
\node at (1.2,2.3){$e_2$};
\node at (1.2,-0.5){$e_3$};
\node at (2.5,1){$e_4$};
\node at (3.2,2.3){$e_5$};
\node at (3.2,-0.5){$e_6$};
\node at (4.5,1){$e_7$};
\node at (5.2,2.3){$e_8$};
\node at (0,)[anchor=south east]{};
\node at (0,0){$\bullet$};
\node at (0,2){$\bullet$};
\node at (2,0){$\bullet$};
\node at (4,0){$\bullet$};
\node at (2,2){$\bullet$};
\node at (4,2){$\bullet$};
\node at (6,2){$\bullet$};
\end{tikzpicture} \hspace{1in} \raisebox{0.35in}{$M=
	\begin{pmatrix}
	1&1&1&1&0&0&0&0\\
	0&0&0&-1&1&1&1&0
	\end{pmatrix}$}
\end{center}
This graph has two faces $F_1$ and $F_2$ with face cycles $e_1 e_2 e_4 e_3$ and $e_4 e_5 e_7 e_6$, respectively. Directing
the edges $e_1,e_2,e_4,e_3$ clockwise around the face $F_1$ and the edges $e_5,e_7,e_6$ clockwise around $F_2$, the cycle basis matrix is given by $M$ above.
\end{example}

\begin{lemma}
\label{lemma.full rank}
Let $G=(V,E)$ be a finite, planar graph with directed edges. Let $M$ be the corresponding cycle basis matrix. Then the rows of $M$ are linearly 
independent, that is, $\mathsf{rank}(M) = f_G$, where $f_G$ is the number of faces in $G$.
\end{lemma}

\begin{proof}
If $G$ does not have any faces, the cycle basis matrix $M$ is empty and indeed $\mathsf{rank}(M) = f_G=0$. If $G$ has faces, let us inductively
add the faces one by one. We can do so independently for each connected component of $G$. Adding a new face, adds at least one new edge to 
the graph that has not been in any other previous face. This adds a new column to $M$ (corresponding to the new edge) with a pivot in the newly added 
row (the places with the $\pm 1$'s in this new row correspond to the edges in the new face). This proves the rank goes up by one. Inductively, this proves 
$\mathsf{rank}(M) = f_G$.
\end{proof}

%%%%%%%%%%%%%%%%%%%%%%%%%%%%%%%%%%%%%%%%%%%%%%%%%%%%%%%%%%%%%%%
\subsection{Constructing the extended cycle basis matrix}

For a finite, planar, directed, edge-labelled graph $(G,\ell)$, we are now going to define an \defn{extended cycle basis matrix}.

 We are interested in degree 2 splines.
As shown in Section~\ref{section.homogeneous}, we may assume without loss of generality that each edge label $\ell(uv)$ has the form 
$\ell(uv) = (x+a_{uv} y)^2$, also called a homogeneous edge-labeling. The difference of the vertex labels
of two vertices $u$ and $v$ connected by an edge $uv$ has to be a multiple of the edge label $\ell(uv)$, so each clockwise cycle $v_1 v_2 \ldots v_k v_1$ 
in the directed, edge-labeled graph $(G, \ell)$ gives rise to a linear equation on the collection of splines, as follows.  
Suppose the label on edge $v_iv_{i+1}$ is $\ell_i$ and $\varepsilon_i=1$ if the edge is directed from $v_i$ to $v_{i+1}$ and $-1$ otherwise, 
where the index $i+1$ is taken mod $k$.  Each degree 2 spline on these vertices can be written
\[
	(z, z+z_1 \varepsilon_1 \ell_1, z+z_1 \varepsilon_1 \ell_1 + z_2 \varepsilon_2 \ell_2, \ldots, z+z_1 \varepsilon_1 \ell_1 + z_2 \varepsilon_2 \ell_2 
	+ \cdots + z_{k-1} \varepsilon_{k-1} \ell_{k-1})
\]
for some choice of scalars $z, z_i \in  R = \mathbb{C}[x,y]$ so long as there is a scalar $z_k$ satisfying
\begin{equation} 
\label{equation.cycle}
	z_1 \varepsilon_1 \ell_1 + z_2 \varepsilon_2 \ell_2 + \cdots + z_k \varepsilon_k \ell_k = 0.
\end{equation}

Assuming that $\ell_i = (x+a_iy)^2$, the condition~\eqref{equation.cycle} gives rise to three linear equations for the variables $z_1,z_2,\ldots,z_k$
by considering the coefficients of the monomials $x^2,xy,y^2$. This is captured in the $3 \times k$ matrix
\begin{equation}
\label{equation.extended}
	\left(\begin{array}{ccccc} 
	\varepsilon_1 & \varepsilon_2 & \varepsilon_3 & \ldots & \varepsilon_k \\
	\varepsilon_1 a_1 & \varepsilon_2 a_2 & \varepsilon_3 a_3 & \ldots & \varepsilon_k a_k \\
	\varepsilon_1 a_1^2 & \varepsilon_2 a_2^2 & \varepsilon_3 a_3^2 & \ldots & \varepsilon_k a_k^2 \\
	\end{array} \right).
\end{equation}

\begin{definition}
\label{definition.extended cycle basis matrix}
The \defn{extended cycle basis matrix} $M^{\mathsf{ext}}$ of a labelled, finite, directed, planar graph $G$ is obtained from the cycle basis matrix 
$M$ by replacing the support of each row with three rows as in~\eqref{equation.extended} with the same support. Hence $M^{\mathsf{ext}}$ 
is a $3f_G \times e_G$ matrix, where $f_G$ is the number of bounded faces in $G$ and $e_G$ is the number of edges in $G$.
\end{definition}

\begin{example}
\label{example.Mext}
The extended cycle basis matrix for Example~\ref{example.M} is
\[
	M^{\mathsf{ext}}=
	\begin{pmatrix}
	1&1&1&1&0&0&0&0\\
	a_1&a_2& a_3 & a_4 &0&0&0&0\\
	a_1^2&a_2^2& a_3^2 & a_4^2 &0&0&0&0\\
	0&0&0&-1&1&1&1&0\\
	0&0&0&-a_4&a_5&a_6&a_7&0\\
	0&0&0&-a_4^2&a_5^2&a_6^2&a_7^2&0
	\end{pmatrix}.
\]
\end{example}

\begin{theorem}
\label{theorem.spline rank}
For a labelled, directed, finite, planar graph $(G,\ell)$, we have
\[
	\dim \mathsf{Spl}_2(G,\ell;v_0) = e_G - \rank(M^{\mathsf{ext}}).
\]
\end{theorem}

\begin{proof}
Start with a spanning tree $T$ of the graph $G$ with the same labels and directions on the edges as in $G$. The dimension of the degree 2 splines associated 
to $T$ which are zero at $v_0$ is $e_T$. Inductively, add an edge $e\in E(G) \setminus E(T)$ to $T$ in such a way that a new bounded face of $G$ is added 
by adding $e$. The dimension of the spline changes by 1 through the new edge $e$ minus the number of linearly independent conditions of the 
form~\eqref{equation.extended}. Since the 
face cycle basis is complete and independent, the set of linear conditions is captured by $M^{\mathsf{ext}}$. This proves the claim of the proposition.
\end{proof}

\begin{example}
The rank of $M^{\mathsf{ext}}$ in Example~\ref{example.Mext} is 6 if all $a_i$ are nonzero and distinct. Hence $\dim \mathsf{Spl}_2(G,\ell; v_0)=8-6=2$.
\end{example}

\begin{remark}
If edge $e_c$ is not on any bounded face $F \in \mathcal{F}$ then column $c$ of the cycle basis matrix $M$ contains only zero entries. So the rank of the extended cycle basis matrix $M^{\mathsf{ext}}$ only
depends on edges that are part of a cycle in the graph. Hence we assume from now on that the graph $G$ does not have any leaves (that is,
vertices with at most one edge as the vertex as an endpoint of the edge).
\end{remark}

Theorem~\ref{theorem.spline rank} reduces the computation of the dimension of the splines to computing the rank of the extended cycle basis matrix
$M^{\mathsf{ext}}$. Determining splines as the kernel of the extended cycle basis matrix corresponds to the method of smoothing cofactors
introduced by Wang~\cite{CW.1983, Wang.1975} and used by Schumaker~\cite{Schumaker.1979}, Whiteley~\cite{Wh.1991,Wh.1991a},
and Alfeld, Schumaker and Whiteley~\cite{ASW.1993}.

%%%%%%%%%%%%%%%%%%%%%%%%%%%%%%%%%%%%%%%%%%%%%%%%%%%%%%%%%%%%%%%
\section{Homogenizing splines of degree 2 by ignoring constant terms}
\label{section.homogeneous}

In this section, we describe a process of homogenizing splines that greatly simplifies rank calculations: namely, for each edge-label, we omit all terms 
of degree less than two. This is \emph{different} from the typical homogenization in algebraic geometry and commutative algebra, as detailed in 
Remark~\ref{remark: homogenization contrast}. Our goal in this section is to show that in the cases that interest analysts, the rank of the homogenized 
spline system $\Spl_2(G,\ell_0)$ is the same as the rank of the original spline system $\Spl_2(G,\ell)$. In particular, our main theorem generalizes the 
following.

\begin{theorem}
Suppose that $\Delta$ is a planar triangulation with no horizontal edges.  Write the associated smoothness-two edge-labeled dual graph by 
$(G_{\Delta},\ell_{\Delta})$, so $\ell_{\Delta}(uv) = (x+a_{uv}y+b_{uv})^2$ for each edge $uv$.  Let $\ell_0$ be the homogeneous edge-labeling 
defined by $\ell_0(uv)=(x+a_{uv}y)^2$ for each edge $uv$.  
Then
\[
	\dim \Spl_2(G_{\Delta},\ell_{\Delta}) = \dim \Spl_2(G_{\Delta},\ell_0) = \dim \mathcal{S}^1_2(\Delta)-6.
\]
\end{theorem}

In the rest of this section, we 1) define homogeneous splines, 2) construct a vector space isomorphic to homogenized splines $\Spl_2(G,\ell_0)$, 3) 
show that the original system $\Spl_2(G,\ell)$ injects into this vector space, then 4) give increasingly general conditions under which $\Spl_2(G,\ell)$ 
actually is isomorphic to $\Spl_2(G,\ell_0)$. Our results extend earlier work by Whiteley~\cite{Wh.1991,Wh.1991a}, not to mention undergraduate linear 
algebra.

\begin{definition}
Let $(G,\ell_0)$ be an edge-labeled, directed graph for which $\ell_0(uv) = (x+a_{uv}y)^2$ for each edge $uv$.  Then $\ell_0$ is called a 
\defn{homogeneous} edge-labeling. Let $b \colon E \rightarrow \mathbb{C}$ be a function whose image at each edge $uv$ is denoted $b_{uv}$ 
and suppose $\ell'$ is another edge-labeling of $G$ with
\[
	\ell'(uv)=(x+a_{uv}y+b_{uv})^2.
\]
Then we call $\ell'$ a \defn{non-homogeneous} labeling.  We refer to the process of passing from $\ell'$ to $\ell_0$ as \defn{homogenizing} the labeling.  
Moreover, we say that $\Spl_2(G,\ell')$ are a \defn{particular system of splines}, that $\Spl_2(G,\ell_0)$ are the \defn{homogeneous splines} on $G$ 
associated to $\ell'$.
\end{definition}

\begin{remark} \label{remark: homogenization contrast}
This is analogous to homogeneous versus particular linear systems of equations from linear algebra.  Another kind of homogenizing creates polynomials whose terms all have the same degree by adding a new variable.  This is common in algebraic geometry and can be found in other work on splines (e.g. \cite{Billera.1988, AMT.2021} but is different from what we do here.  For instance, an edge labeled $(x+y+3)^2$ is homogenized in our sense (and the interpretation from linear algebra) to $(x+y)^2$ but would be homogenized to $(x+y+3z)^2$ in the interpretation from algebraic geometry.  

When an edge-labeled graph is dual to a planar triangulation, its edge labels are equations for the lines each edge lies on, and our homogenization translates all edges simultaneously to go through the origin.  We do not have a meaningful way to interpret this as a geometric operation on the original planar triangulation; see also Section~\ref{section: open questions}.
\end{remark}

\begin{example} \label{example: starting homogeneous and particular}
Figure~\ref{figure: homogeneous and particular example} shows a graph $G$ in red with edges labeled by the linear forms $x+a_{uv}y+b_{uv}$ for the non-homogenous edge-labeling $\ell'(uv) = (x+a_{uv}y+b_{uv})^2$.  

On the right, a black version of the graph is labeled by the linear forms $x+a_{uv}y$ for the associated homogeneous edge-labeling with $\ell_0(uv)=(x+a_{uv}y)^2$.  (The red graph $G$ is dual as a planar graph to the triangulation $\Delta$ drawn around it in black, whose vertices are $(0,1), (0,-1),(1,0)$ and $(2,3), (-2,0), (2,-3)$. The non-homogeneous edge-labeling gives the equations of the lines that span the interior edges of $\Delta$.) 
\begin{figure}[h]
\scalebox{0.5}{\begin{tikzpicture}
    \draw[thick] (0,-1) -- (0,1);
    \draw[thick] (0,1) -- (1,0);
    \draw[thick] (1,0) -- (0,-1);
    \draw[thick] (-2,0) -- (2,3);
    \draw[thick] (-2,0) -- (0,1);
    \draw[thick] (-2,0) -- (0,-1);
    \draw[thick] (-2,0) -- (2,-3);
    \draw[thick] (2,3) -- (0,1);
    \draw[thick] (2,3) -- (1,0);
    \draw[thick] (2,-3) -- (0,-1);
    \draw[thick] (2,-3) -- (1,0);

    \filldraw[red] (0.4,0) circle (3pt);
    \filldraw[red] (-1,0) circle (3pt);
    \filldraw[red] (1.3,1.8) circle (3pt);
    \filldraw[red] (1.3,-1.8) circle (3pt);
    \filldraw[red] (-0.1,1.2) circle (3pt);
    \filldraw[red] (-0.1,-1.2) circle (3pt);
    \draw[thick, red] (0.4,0) -- (1.3,1.8);
    \draw[thick, red] (0.4,0) -- (1.3,-1.8);
    \draw[thick, red] (0.4,0) -- (-1,0);
    \draw[thick, red] (-0.1,-1.2) -- (1.3,-1.8);
    \draw[thick, red] (-0.1,-1.2) -- (-1,0);
  \draw[thick, red] (-0.1,1.2) -- (1.3,1.8);
    \draw[thick, red] (-0.1,1.2) -- (-1,0);

    \draw[red] (1.4,1) node[anchor=west] {\scalebox{2}{$x+y-1$}};
    \draw[red] (1.4,-1) node[anchor=west] {\scalebox{2}{$x-y-1$}};
    \draw[red] (-1,1.8) node[anchor=south] {\scalebox{2}{$x-y+1$}};
    \draw[red] (-1,-1.8) node[anchor=north] {\scalebox{2}{$x+y+1$}};
    \draw[red] (-1,1) node[anchor=east] {\scalebox{2}{$x-2y+2$}};
    \draw[red] (-1,-1) node[anchor=east] {\scalebox{2}{$x+2y+2$}};
    \draw[red] (-.3,0) node[anchor=south] {\scalebox{2}{$x$}};
\end{tikzpicture}} \hspace{0.5in}
\scalebox{0.5}{\begin{tikzpicture}
        \filldraw (0.4,0) circle (3pt);
    \filldraw (-1,0) circle (3pt);
    \filldraw (1.3,1.8) circle (3pt);
    \filldraw (1.3,-1.8) circle (3pt);
    \filldraw (-0.1,1.2) circle (3pt);
    \filldraw (-0.1,-1.2) circle (3pt);
    \draw[thick] (0.4,0) -- (1.3,1.8);
    \draw[thick] (0.4,0) -- (1.3,-1.8);
    \draw[thick] (0.4,0) -- (-1,0);
    \draw[thick] (-0.1,-1.2) -- (1.3,-1.8);
    \draw[thick] (-0.1,-1.2) -- (-1,0);
  \draw[thick] (-0.1,1.2) -- (1.3,1.8);
    \draw[thick] (-0.1,1.2) -- (-1,0);

    \draw (1,0.8) node[anchor=west] {\scalebox{2}{$x+y$}};
    \draw (1,-0.8) node[anchor=west] {\scalebox{2}{$x-y$}};
    \draw (0,1.7) node[anchor=south] {\scalebox{2}{$x-y$}};
    \draw (0,-1.7) node[anchor=north] {\scalebox{2}{$x+y$}};
    \draw (-0.7,0.7) node[anchor=east] {\scalebox{2}{$x-2y$}};
    \draw (-0.7,-0.7) node[anchor=east] {\scalebox{2}{$x+2y$}};
    \draw (-.3,0) node[anchor=south] {\scalebox{2}{$x$}};
\end{tikzpicture}}
\caption{A triangulation describing a particular system of splines, the corresponding (non-homogeneous) edge-labeling of its dual graph, and the associated homogeneous edge-labeled graph. \label{figure: homogeneous and particular example}}
\end{figure}
\end{example}

The next definition is critical to our argument.  It identifies a vector space that is isomorphic to the homogeneous splines, as well as to the kernel of $M^\mathsf{ext}$, and that in general contains any particular set of splines with that same homogenization.  It relies on the choice of a direction for each edge of the graph $G$.

\begin{definition} \label{definition: map splines to vector space of edge-adjustments}
Suppose that $(G, \ell)$ is a directed, edge-labeled graph. Define a map $\varphi_{\ell} \colon \Spl(G, \ell) \longrightarrow \mathbb{C}[x,y]^{|E(G)|}$ 
as follows.  For each spline $p \in \Spl(G,\ell)$ and each directed edge $u \mapsto v$ let
\[
	\varphi_{\ell}(p)(u \mapsto v) = c_{uv} \hspace{2em} \textup{ if and only if } \hspace{2em} p(v)-p(u) = c_{uv} \ell_{uv}.
\]
\end{definition}

Note that if we consider splines as a $\mathbb{C}[x,y]$-module, then $\varphi_{\ell}$ is a $\mathbb{C}[x,y]$-module map.  

\begin{remark}\label{remark.restricted}
When we restrict the map $\varphi_{\ell}$ to $\Spl_2(G,\ell)$ and $\ell$ is an edge-labeling by degree-two polynomials, then in fact $c_{uv} \in \mathbb{C}$ 
for all edges $uv$ and 
\[
	\varphi_{\ell} \colon \Spl_2(G,\ell) \rightarrow \mathbb{C}^{|E(G)|}
\]
is a linear map of $\mathbb{C}$-vector spaces. We often write $\vec{c}$ for a vector in the image of $\varphi_{\ell}$.
\end{remark}

\begin{example}\label{example: continuing homogeneous and nonhom example}
Continuing Example~\ref{example: starting homogeneous and particular}, Figure~\ref{figure: the constants in homogeneous and particular example} shows a spline in the particular system of splines associated to the non-homogeneous labeling $\ell'$ as well as a spline in the homogeneous system associated to $\ell_0$.  In the center of Figure~\ref{figure: the constants in homogeneous and particular example} are the scalars associated to each spline by the maps $\varphi_{\ell'}$ and $\varphi_{\ell_0}$ respectively.  The main goal of this section is Lemma~\ref{lemma: facetranslateable means nonhom splines are iso to hom}, which gives a condition under which non-homogeneous splines are isomorphic to homogeneous splines.  This condition is satisfied by every edge-labeled graph that is dual to a triangulation (and indeed to more general plane partitions).
\begin{figure}[h]
\scalebox{0.5}{\begin{tikzpicture}
    \filldraw[red] (0.4,0) circle (2pt) node[anchor=west, blue]{\scalebox{1.5}{$-3x^2$}};
    \filldraw[red] (-1,0) circle (2pt) node[anchor=east, blue]{\scalebox{1.5}{$0$}};
    \filldraw[red] (1.3,1.8) circle (2pt) node[anchor=south west, blue]{\scalebox{1.5}{$-2x^2+2xy+y^2$}};
        \draw[blue] (1.7,1.2) node[anchor=south west, blue]{\scalebox{1.5}{$-2x-2y+1$}};
    \filldraw[red] (1.3,-1.8) circle (2pt) node[anchor=north west, blue]{\scalebox{1.5}{$-2x+2y+1$}};
         \draw[blue] (1.3,-1.2) node[anchor=north west, blue]{\scalebox{1.5}{$-2x^2-2xy+y^2$}};
    \filldraw[red] (-0.1,1.2) circle (2pt) node[anchor=south east, blue]{\scalebox{1.5}{$x^2-4xy+4y^2$}};
    \draw[blue] (-0.5,0.6) node[anchor=south east, blue]{\scalebox{1.5}{$+4x-8y+4$}};
    \filldraw[red] (-0.1,-1.2) circle (2pt) node[anchor=north east, blue]{\scalebox{1.5}{$+4x+8y+4$}};
        \draw[blue] (-0.5,-0.6) node[anchor=north east, blue]{\scalebox{1.5}{$x^2+4xy+4y^2$}};

    \draw[thick, red] (0.4,0) -- (1.3,1.8);
    \draw[thick, red] (0.4,0) -- (1.3,-1.8);
    \draw[thick, red] (0.4,0) -- (-1,0);
    \draw[thick, red] (-0.1,-1.2) -- (1.3,-1.8);
    \draw[thick, red] (-0.1,-1.2) -- (-1,0);
  \draw[thick, red] (-0.1,1.2) -- (1.3,1.8);
    \draw[thick, red] (-0.1,1.2) -- (-1,0);

    \draw[red, <-, very thick] (0.7,0.6) -- (1.3,1.8);
\draw[red, <-, very thick] (0.7,-0.6) -- (1.3,-1.8);
\draw[red, ->, very thick] (-1,0) -- (0.1,0);
    \draw[red, ->, very thick] (-0.1,-1.2) -- (0.6,-1.5);
    \draw[red, <-, very thick] (-0.4,-0.8) -- (-1,0);
        \draw[red, ->, very thick] (-0.1,1.2) -- (0.6,1.5);
    \draw[red, <-, very thick] (-0.4,0.8) -- (-1,0);

  %  \draw[red] (1,0.8) node[anchor=west] {\scalebox{1.5}{$x+y-1$}};
  %  \draw[red] (1,-0.8) node[anchor=west] {\scalebox{1.5}{$x-y-1$}};
  %  \draw[red] (0,1.7) node[anchor=south] {\scalebox{1.5}{$x-y+1$}};
  %  \draw[red] (0,-1.7) node[anchor=north] {\scalebox{1.5}{$x+y+1$}};
  %  \draw[red] (-0.7,0.7) node[anchor=east] {\scalebox{1.5}{$x-2y+2$}};
  %  \draw[red] (-0.7,-0.7) node[anchor=east] {\scalebox{1.5}{$x+2y+2$}};
  %  \draw[red] (-.3,0) node[anchor=south] {\scalebox{1.5}{$x$}};
\end{tikzpicture}}  \hspace{0.5in}
\scalebox{0.5}{\begin{tikzpicture}
        \filldraw (0.4,0) circle (2pt);
    \filldraw (-1,0) circle (2pt);
    \filldraw (1.3,1.8) circle (2pt);
    \filldraw (1.3,-1.8) circle (2pt);
    \filldraw (-0.1,1.2) circle (2pt);
    \filldraw (-0.1,-1.2) circle (2pt);
    \draw[thick] (0.4,0) -- (1.3,1.8);
    \draw[thick] (0.4,0) -- (1.3,-1.8);
    \draw[thick] (0.4,0) -- (-1,0);
    \draw[thick] (-0.1,-1.2) -- (1.3,-1.8);
    \draw[thick] (-0.1,-1.2) -- (-1,0);
  \draw[thick] (-0.1,1.2) -- (1.3,1.8);
    \draw[thick] (-0.1,1.2) -- (-1,0);

\draw[<-, very thick] (0.7,0.6) -- (1.3,1.8);
\draw[<-, very thick] (0.7,-0.6) -- (1.3,-1.8);
\draw[->, very thick] (-1,0) -- (0.1,0);
    \draw[->, very thick] (-0.1,-1.2) -- (0.6,-1.5);
    \draw[<-, very thick] (-0.4,-0.8) -- (-1,0);
        \draw[->, very thick] (-0.1,1.2) -- (0.6,1.5);
    \draw[<-, very thick] (-0.4,0.8) -- (-1,0);

    \draw (1,0.8) node[anchor=west] {\scalebox{1.5}{$-1$}};
    \draw (1,-0.8) node[anchor=west] {\scalebox{1.5}{$-1$}};
    \draw (0,1.7) node[anchor=south] {\scalebox{1.5}{$-3$}};
    \draw (0,-1.7) node[anchor=north] {\scalebox{1.5}{$-3$}};
    \draw (-0.7,0.7) node[anchor=east] {\scalebox{1.5}{$1$}};
    \draw (-0.7,-0.7) node[anchor=east] {\scalebox{1.5}{$1$}};
    \draw (-.3,0) node[anchor=south] {\scalebox{1.5}{$1$}};
\end{tikzpicture}} \hspace{0.35in}
\scalebox{0.5}{\begin{tikzpicture}
        \filldraw (0.4,0) circle (2pt) node[anchor=west, blue]{\scalebox{1.5}{$-3x^2$}};
    \filldraw (-1,0) circle (2pt) node[anchor=east, blue]{\scalebox{1.5}{$0$}};
    \filldraw (1.3,1.8) circle (2pt) node[anchor=south west, blue]{\scalebox{1.5}{$-2x^2+2xy+y^2$}};
    \filldraw (1.3,-1.8) circle (2pt) node[anchor=north west, blue]{\scalebox{1.5}{$-2x^2-2xy+y^2$}};
    \filldraw (-0.1,1.2) circle (2pt) node[anchor=south east, blue]{\scalebox{1.5}{$x^2-4xy+4y^2$}};
    \filldraw (-0.1,-1.2) circle (2pt) node[anchor=north east, blue]{\scalebox{1.5}{$x^2+4xy+4y^2$}};
    \draw[thick] (0.4,0) -- (1.3,1.8);
    \draw[thick] (0.4,0) -- (1.3,-1.8);
    \draw[thick] (0.4,0) -- (-1,0);
    \draw[thick] (-0.1,-1.2) -- (1.3,-1.8);
    \draw[thick] (-0.1,-1.2) -- (-1,0);
  \draw[thick] (-0.1,1.2) -- (1.3,1.8);
    \draw[thick] (-0.1,1.2) -- (-1,0);

\draw[<-, very thick] (0.7,0.6) -- (1.3,1.8);
\draw[<-, very thick] (0.7,-0.6) -- (1.3,-1.8);
\draw[->, very thick] (-1,0) -- (0.1,0);
    \draw[->, very thick] (-0.1,-1.2) -- (0.6,-1.5);
    \draw[<-, very thick] (-0.4,-0.8) -- (-1,0);
        \draw[->, very thick] (-0.1,1.2) -- (0.6,1.5);
    \draw[<-, very thick] (-0.4,0.8) -- (-1,0);

%    \draw (1,0.8) node[anchor=west] {\scalebox{1.5}{$x+y$}};
%    \draw (1,-0.8) node[anchor=west] {\scalebox{1.5}{$x-y$}};
%    \draw (0,1.7) node[anchor=south] {\scalebox{1.5}{$x-y$}};
%    \draw (0,-1.7) node[anchor=north] {\scalebox{1.5}{$x+y$}};
%    \draw (-0.7,0.7) node[anchor=east] {\scalebox{1.5}{$x-2y$}};
%    \draw (-0.7,-0.7) node[anchor=east] {\scalebox{1.5}{$x+2y$}};
%    \draw (-.3,0) node[anchor=south] {\scalebox{1.5}{$x$}};
\end{tikzpicture}}
\caption{The maps $\varphi_{\ell'}$ and $\varphi_{\ell_0}$ send the splines on the left and right, respectively, to the scalars labeling edges in the center
as referred to in Example~\ref{example: continuing homogeneous and nonhom example.}
\label{figure: the constants in homogeneous and particular example}
}
\end{figure}
\end{example}

\begin{lemma} 
\label{lemma: varphi for trees}
Suppose that $T$ is a tree rooted at $v_0$ whose edges are all directed and that $\alpha$ is any edge-labeling by complex polynomials. Then the kernel 
of $\varphi_{\alpha}$ consists of the constant splines and the image is $\mathbb{C}[x,y]^{n-1}$, where $n$ is the number of vertices of $T$.

In particular, if $\ell$ and $\ell'$ are any edge-labelings of $T$ by polynomials of degree two, then the map
\[ 
	\varphi_{\ell}^{-1} \circ \varphi_{\ell'} \colon \Spl(T, \ell') \rightarrow \Spl(T,\ell)
\]
is an isomorphism of splines that are zero at $v_0$.  Moreover, this restricts to an isomorphism on the vector spaces 
$\Spl_2(T,\ell') \rightarrow \Spl_2(T,\ell)$ of splines of degree $2$, both of which are isomorphic under $\varphi_{\ell}$ to $\mathbb{C}^{n-1}$.
\end{lemma}

\begin{proof}
This follows directly from the explicit basis for splines on trees given in \cite{GTV.2016}.  Indeed, that result constructed a basis $\mathcal{B}_{v_0}$ 
for an edge-labeled tree $(T,\alpha)$ rooted at $v_0$ as follows.  Since $T$ is a tree, there is a unique path between the root $v_0$ and each vertex 
$u \neq v_0$ in $T$. For each vertex $v$, let $\mathcal{S}_v$ be the set of vertices whose path to $v_0$ contains $v$, namely all the descendants 
of $v$.  The set $\mathcal{S}_v$ is nonempty since the vertex $v$ is contained in $\mathcal{S}_v$.  Suppose that $vv'$ is the first edge in the path from 
$v$ to $v_0$, namely $v'$ is the parent of $v$. Then the basis spline $\mathbf{b}_v$ is defined to be 
\[
	\mathbf{b}_v(u) = \left\{ \begin{array}{ll} 0 & \textup{ if } u \not \in \mathcal{S}_v, \\
	\alpha(vv') & \textup{ if } u \in \mathcal{S}_v. \end{array} \right.
\]
The basis spline $\mathbf{b}_{v_0}$ corresponding to 
the root $v_0$ is the identity spline, namely $1$ at all vertices.

Note that $\varphi_{\alpha}(\mathbf{b}_{v_0})$ is zero since $\mathbf{b}_{v_0}$ is a constant spline.  Moreover if $v \neq v_0$ then 
$\varphi_{\alpha}(\mathbf{b}_v)(u) = \pm \delta_{uv}$ since $\mathbf{b}_v$ is constant on every edge except the unique edge $uv$ on the path 
from $v_0$ to $v$.  In that case $\varphi_{\alpha}(\mathbf{b}_v)(u)$ is $1$ if $u \mapsto v$ and $-1$ otherwise.  Thus the map $\varphi_{\alpha}$ 
is a $\mathbb{C}$-linear map that sends the subset of basis elements $\mathcal{B}_{v_0} \backslash \{\mathbf{b}_{v_0}\}$ bijectively onto a basis of 
$\mathbb{C}[x,y]^{n-1}$ and sends all (polynomial) multiples of $\mathbf{b}_{v_0}$ to zero.  This proves the first part of the claim.  

If $\ell$ is an edge-labeling by degree-two polynomials in $\mathbb{C}[x,y]$ of a tree, then the elements of $\Spl(T,\ell)$ have the form 
\[
	p_0\mathbf{b}_{v_0} + \sum_{v \neq v_0} p_v \mathbf{b}_v
\]
for polynomials $p_0, p_v \in \mathbb{C}[x,y]$. Of these terms, only $\mathbf{b}_{v_0}$ is nonzero at $v_0$.  So we have the claimed isomorphism.  
Restricting to degree two means assuming that $p_0$ has degree at most two and all the $c_{uv}$ are elements of $\mathbb{C}$, since any higher-degree 
coefficient would force at least one polynomial in the spline to have degree greater than two.  Thus the image of both $\varphi_{\ell}$ and $\varphi_{\ell'}$ is 
$\mathbb{C}^{n-1}$.
\end{proof}

\begin{remark}   
The vector space $\mathbb{C}^{n-1}$ in Lemma~\ref{lemma: varphi for trees} arises because we consider splines over the coefficients $\mathbb{C}[x,y]$.  
This result can be extended in several ways. First, if we use any other field $\mathbb{K}$ of characteristic zero then we would replace the vector space 
$\mathbb{C}^{n-1}$ with $\mathbb{K}^{n-1}$, but the proof would otherwise be the same.  Most applications of splines use real polynomials, giving 
an isomorphism to the vector space $\mathbb{R}^{n-1}$. 

Moreover, the proof 
applies to any edge-labeling $\alpha$ for which the image $\alpha(e)$ for an edge $e$ is always a principal ideal.  In fact, the result in~\cite{GTV.2016} 
applies to a larger family of coefficient rings than just polynomial algebras so this lemma could be generalized further.
\end{remark}

\begin{lemma} 
\label{lemma: homogeneous splines surject onto kernel}
The kernel of the map $\varphi_{\ell}$ restricted to to $\Spl_2(G,\ell)$ as in Remark~\ref{remark.restricted} is precisely the collection of constant 
splines on $(G,\ell)$.  The image of the map $\varphi_{\ell} $ is contained 
in the kernel of the matrix $M^{\mathsf{ext}}$ regardless of whether $\ell$ is homogeneous or not.  Moreover, if $\ell$ is homogeneous then 
$\varphi_{\ell} $ is an isomorphism between the homogeneous splines of degree $2$ on $(G,\ell)$ and $\ker M^{\mathsf{ext}}$.
\end{lemma}

\begin{proof}
Suppose that $\vec{c}$ is a vector in the image of $\varphi_{\ell}$ and that $u_0u_1u_2u_3\cdots u_{k-1}u_0$ is any face cycle in $G$ with the vertices 
ordered clockwise around the boundary of the face. 

Then $p \in \Spl_2(G,\ell)$ if and only if
\begin{equation}
\label{equation: edge-label for ordered cycles} 
	p(u_{i+1}) - p(u_i)=d_{u_iu_{i+1}}\ell_{u_iu_{i+1}}
\end{equation}
for some scalar $d_{u_iu_{i+1}}$ on each edge of each face cycle (with indices taken modulo $k$).  In particular, we may take $d_{u_iu_{i+1}}$ 
to be $c_{u_iu_{i+1}}$ if the edge $u_i \mapsto u_{i+1}$ is directed clockwise around its face cycle and $d_{u_iu_{i+1}}=-c_{u_iu_{i+1}}$ if not.

Repeated substitution of Equation~\eqref{equation: edge-label for ordered cycles} shows that $p \in \Spl_2(G,\ell)$ if and only if we can write $p(u_0)$ as
\[
	\begin{array}{ll} p(u_0) &= p(u_{k-1}) + d_{u_{k-1}u_0} \ell_{u_{k-1}u_0} = p(u_{k-2}) + d_{u_{k-2}u_{k-1}} \ell_{u_{k-2}u_{k-1}} + d_{u_{k-1}u_0} 
	\ell_{u_{k-1}u_0} = \cdots \\ & = p(u_0) + d_{u_0u_1} \ell_{u_0u_1} + d_{u_1u_2} \ell_{u_1u_2} + \cdots + d_{u_{k-2}u_{k-1}} \ell_{u_{k-2}u_{k-1}} 
	+ d_{u_{k-1}u_0} \ell_{u_{k-1}u_0}  \end{array}
\]	
for each face cycle in $G$. Subtracting $p(u_0)$ from both sides gives
\[
	d_{u_0u_1} \ell_{u_0u_1} + \cdots + d_{u_{k-2}u_{k-1}} \ell_{u_{k-2}u_{k-1}} + d_{u_{k-1}u_0} \ell_{u_{k-1}u_0} =0.
\]
Examining the coefficients of $x^2$, $xy$, and $y^2$ in this equation gives the triple of equations
\[\left\{\begin{array}{lllll}
	d_{u_0u_1} &+ \cdots &+ d_{u_{k-2}u_{k-1}} &+ d_{u_{k-1}u_0}  &=0, \\
	d_{u_0u_1} a_{u_0u_1} &+ \cdots &+ d_{u_{k-2}u_{k-1}} a_{u_{k-2}u_{k-1}} &+ d_{u_{k-1}u_0} a_{u_{k-1}u_0} & =0, \\
	d_{u_0u_1} a_{u_0u_1}^2 &+ \cdots &+ d_{u_{k-2}u_{k-1}} a_{u_{k-2}u_{k-1}}^2 &+ d_{u_{k-1}u_0} a_{u_{k-1}u_0}^2 & =0.
\end{array}\right.
\]
Note that these are the coefficients of $x^2$, $xy$, and $y^2$ regardless of whether $\ell$ is a homogeneous or non-homogeneous edge-labeling.  
If $\ell$ is a homogeneous labeling, then $p \in \Spl_2(G,\ell)$ if and only if this triple of equations holds for each face cycle in $G$.  If $\ell$ is 
non-homogeneous, then $p \in \Spl_2(G,\ell)$ implies that this triple of equations holds for each face cycle in $G$ but additional equations (for the 
coefficients of $x$, $y$, and the constant term) must also hold.

Since $c_{u_iu_{i+1}}= \pm d_{u_iu_{i+1}}$ depending on whether $u_i \mapsto u_{i+1}$ is directed clockwise around the face cycle or not, we conclude 
that $\vec{c}$ is in the kernel of $M^{\mathsf{ext}}$ if $p \in \Spl_2(G,\ell)$ regardless of whether $\ell$ is homogeneous or not.

Moreover, when $\ell$ is homogeneous we have $\vec{c} \in \ker M^{\mathsf{ext}}$ if and only if $p \in \Spl_2(G,\ell)$.  In other words, for 
homogeneous $\ell$ the map $\varphi_{\ell} $ surjects onto $\ker M^{\mathsf{ext}}$.  

For both homogeneous and non-homogeneous edge-labelings, the kernel of $\varphi_{\ell}$ consists of precisely those splines $p \in \Spl_2(G,\ell)$ 
for which $p(u)=p(v)$ for all vertices $u, v$ in $G$.   This proves the claim.
\end{proof}

\begin{definition}
\label{definition.face translate}
Let $G$ be a planar graph.  We say that a (not necessarily homogeneous) edge-labeling $\ell'$ is \defn{face-translatable} if for each face cycle 
$C = e_1e_2 \cdots e_ke_1$ in $G$ there is a point $(x_0,y_0)$ that simultaneously solves the equation $\ell'(e_i)=0$ for each edge-label on $C$.
\end{definition}

For instance, a homogeneous edge-labeling is face-translatable because the point $(0,0)$ simultaneously solves every equation of the form $ax+by=0$.  

For a non-homogeneous example, consider the edge-labeled graph $G$ that is dual to the planar triangulation $\Delta$.  Each face $F_0$ in $G$ 
corresponds to an interior vertex $(x_0,y_0)$ of the triangulation $\Delta$ by construction of dual graphs.  Also by the construction of dual {\it edge-labeled} 
graphs, all of the edge-labels $ax+by+c$ on the face cycle bounding $F_0$ correspond to line segments that go through this interior vertex $(x_0,y_0)$
in $\Delta$.  Thus the point $(x_0,y_0)$ satisfies the equation $ax+by+c=0$ whenever there is an edge labeled $ax+by+c$ bounding $F_0$.
A triangulation and its dual graph are depicted in Figure~\ref{figure.triangulation}. This example appeared in~\cite{MS.1975} and shows that the
lower bound on the dimension of splines in Schumaker's landmark paper~\cite{Schumaker.1979} can fail to give the dimension of splines
for degree two and smoothness one. We analyze this example further in Example~\ref{example: Morgan-Scott} below.

\begin{center}
\begin{figure}

\scalebox{0.8}{
\begin{tikzpicture}
\draw[red,thick] (0,2) to (0,0);
\draw[red,thick] (0,0) to (1.5,-1.32);
\draw[red,thick] (0,0) to (-1.5,-1.32);
\draw[red,thick] (-1.5,0.68) to (-1.5,-1.32);
\draw[red,thick] (1.5,0.68) to (1.5,-1.32);
\draw[red,thick] (0,2) to (1.5,0.68);
\draw[red,thick] (0,2) to (-1.5,0.68);
\draw[red,thick] (0,-2.64) to (1.5,-1.32);
\draw[red,thick] (0,-2.64) to (-1.5,-1.32);
\node at (0,)[anchor=south east]{};

\draw[thick] (-0.8,0.68) to (0.8,0.68);
\draw[thick] (-0.8,0.68) to (0,-1.4);
\draw[thick] (0.8,0.68) to (0,-1.4);
\draw[thick] (0,-1.4) to (3.8,-3);
\draw[thick] (0.8,0.68) to (3.8,-3);
\draw[thick] (0,-1.4) to (-3.8,-3);
\draw[thick] (-0.8,0.68) to (-3.8,-3);
\draw[thick] (-3.8,-3) to (3.8,-3);
\draw[thick] (-0.8,0.68) to (0,3.8);
\draw[thick] (0.8,0.68) to (0,3.8);
\draw[thick] (-3.8,-3) to (0,3.8);
\draw[thick] (3.8,-3) to (0,3.8);
\filldraw[red] (0,2) circle (3pt);
\filldraw[red] (0,0) circle (3pt);
\filldraw[red] (1.5,-1.32) circle (3pt);
\filldraw[red] (-1.5,-1.32) circle (3pt);
\filldraw[red] (1.5,0.68) circle (3pt);
\filldraw[red] (-1.5,0.68) circle (3pt);
\filldraw[red] (0,-2.64) circle (3pt);
\end{tikzpicture}}
\caption{Triangulation and its dual graph which appeared in~\cite{MS.1975}.\\
This triangulation will be further analyzed in Example~\ref{example: Morgan-Scott}.
\label{figure.triangulation}}
\end{figure}
\end{center}
\begin{lemma}
Suppose that $C$ is a directed cycle with homogeneous edge-labeling $\ell$ and face-translatable edge-labeling $\ell'$ giving rise to a particular 
system of splines associated to $\ell$.  Then 
\[
	\varphi_{\ell'} \colon \Spl_2(C,\ell') \longrightarrow \ker M^{\mathsf{ext}}
\] 
is a surjection. 
\end{lemma}

\begin{proof}
From Lemma~\ref{lemma: homogeneous splines surject onto kernel}, we know that the image of $\varphi_{\ell'}$ is contained in $\ker M^{\mathsf{ext}}$.  
We now prove that the two sets are in fact equal.

If $\ell'$ is face-translatable, then by definition there is a point $(x_0,y_0)$ associated to $C$ so that each edge that bounds $C$ is labeled by an 
expression of the form
\[
	(x+a_iy+b_i)^2 \hspace{0.5in} \textup{ where } b_i = a_i y_0-x_0.
\]

We want to show that if $\vec{c}$ is in the kernel of this matrix then $\vec{c} \in \varphi_{\ell'}(\Spl_2(C,\ell'))$ namely that
\[ 
	d_1(x+a_1y + b_1)^2 + d_2(x+a_2y+b_2)^2 + \cdots + d_k (x+a_ky+b_k)^2 = 0,
\]
where the $d_i = \pm c_i$ as in the proof of Lemma~\ref{lemma: homogeneous splines surject onto kernel}.
Associating, distributing, and collecting terms, we can expand the previous equation as
\[
\begin{array}{rl}
 	\displaystyle d_1((x+a_1y) + b_1)^2 + d_2((x+a_2y)+b_2)^2 + \cdots + d_k ((x+a_ky)+b_k)^2 & \\ 
 	= \displaystyle \sum_{i=1}^k \left( d_i(x+a_iy)^2\right)+2\sum_{i=1}^k \left(d_ib_i(x+a_iy)\right) + \sum_{i=1}^kd_ib_i^2. &
\end{array}
\]
The first sum in this expression can be rewritten as $\sum_{i=1}^k d_i \ell_i$ using the homogeneous edge-labeling $\ell$.  This is the defining equation 
for the splines $\Spl_2(C,\ell)$ so by Lemma~\ref{lemma: homogeneous splines surject onto kernel} we know that $\vec{c} \in \ker M^{\mathsf{ext}}$ 
implies $\sum_{i=1}^k d_i \ell_i = 0$.

Now substitute $b_i = a_i y_0 - x_0$ in the other two sums to get the expression
\[
	\begin{array}{rl} \displaystyle 2\sum_{i=1}^k & \displaystyle \left( d_i(a_iy_0 - x_0)(x+a_iy)\right) + \sum_{i=1}^kd_i(a_iy_0-x_0)^2 \\ 
	&= \displaystyle  2\sum_{i=1}^k \left(d_ia_i^2y_0y +d_ia_iy_0x -d_ix_0x -d_ix_0a_iy \right) + \sum_{i=1}^kd_i\left(a_i^2y_0^2-2a_iy_0x_0+x_0^2\right).	
	\end{array}
\] 
None of $x, y, x_0, y_0$ depend on the index of summation $i$, so we can distribute once more to obtain
\[
	(2y_0y+y_0^2)\sum_{i=1}^k d_ia_i^2 + 2(y_0x-x_0y-x_0y_0)\sum_{i=1}^k d_ia_i + (-2x_0x+x_0^2)\sum_{i=1}^k d_i.
\]
These sums are the rows of $M^{\mathsf{ext}}\vec{c}$ and so are each zero by our assumption that $\vec{c}$ is in the kernel of $M^{\mathsf{ext}}$.  
It follows that $\sum_{i=1}^k d_i \ell_i'=0$ and so $\varphi_{\ell'}$ surjects onto the kernel as desired.
\end{proof}

\begin{lemma}\label{lemma: facetranslateable means nonhom splines are iso to hom}
Suppose that $G$ is a directed planar graph, that $\ell$ is a homogeneous edge-labeling of $G$, and that $\ell'$ is a {\it face-translatable} 
non-homogeneous edge-labeling of $G$ whose associated homogeneous edge-labeling is $\ell$.  

Then $\varphi_{\ell'} \colon \Spl_2(G,\ell') \longrightarrow \ker M^{\mathsf{ext}}$ is a surjection.  In particular 
\[
	\varphi_{\ell}^{-1} \circ \varphi_{\ell'}: \Spl_2(G,\ell') \longrightarrow \Spl_2(G,\ell)
\]
is an isomorphism of (real or complex) vector spaces when restricted to degree-2 splines.  In particular, homogenizing the spline system associated to 
$\ell'$ does not change the dimension of the space of degree-2 splines as long as $\ell'$ is face-translatable.
\end{lemma}

\begin{proof}
We prove this by induction on the number of faces in $G$.  Lemma~\ref{lemma: varphi for trees} proved the claim when $G$ has zero faces, 
namely is a tree.  Assume the claim holds for all directed planar graphs with at most $k-1$ faces and assume that $G$ has $k$ faces.  If so, 
then $G$ has a face that shares an edge with the unbounded face.  Let $e_1$ be one such edge and let $F_1$ be the (unique) bounded face 
with $e_1$ on its boundary.  Write $M^{\mathsf{ext}}$ for the matrix associated to $G$ and $\varphi_{\ell'}$ for the map associated to $G$.

Let $G'$ be the directed graph with $k-1$ faces obtained by removing $e_1$ from $G$.  Let ${M^{\mathsf{ext}}}'$ be the matrix associated 
to $G'$ and $\varphi_{\ell'}'$ for the map associated to the graph $(G',\ell'_{E(G')})$. We obtain ${M^{\mathsf{ext}}}'$ from 
$M^{\mathsf{ext}}$ by erasing the  column for $e_1$ as well as the three rows for $F_1$.  By induction we know  
\[
	\varphi_{\ell'}' \colon \Spl_2(G',\ell') \longrightarrow \ker {M^{\mathsf{ext}}}'
\]
is surjective.  

Now let $C$ be the face cycle that bounds $F_1$ and write $\ell'_C$ for the restriction of $\ell'$ to $C$.  Write  ${M^{\mathsf{ext}}_{C}}$ for the 
matrix associated to $C$ and $\varphi_{\ell'_C}$ for the map associated to the graph $(C,\ell'_{C})$. We obtain ${M^{\mathsf{ext}}_{C}}$ from $M^{\mathsf{ext}}$ by erasing all rows except the three rows for $F_1$ and all columns for edges not on $C$.  By induction we know  
\[
	\varphi_{\ell'_C} \colon \Spl_2(C,\ell'_C) \longrightarrow \ker {M^{\mathsf{ext}}_{C}}
\]
is also surjective.  

Finally, note that $\varphi_{\ell'}'$ is simply the map $\varphi_{\ell'}$ restricted to the edges $E(G) - \{e_1\}$ while $\varphi_{\ell'_C}$ is the map 
$\varphi_{\ell'}$ restricted to the edges in $C$.  This means the map 
\[
	\varphi_{\ell'} \colon \Spl_2(G,\ell') \longrightarrow \ker {M^{\mathsf{ext}}}
\]
factors through the two maps $\varphi_{\ell'_C}$ and $\varphi_{\ell'}'$.  So by induction it, too, is surjective.  So we conclude that when $\ell'$ 
is face-translatable, the dimension of $\Spl_2(G,\ell')$ is the same as the dimension of the homogenized system $\Spl_2(G,\ell)$, as desired.
\end{proof}

We make one other simplifying assumption in our labeling: that the coefficient of $x$ can be taken to be $1$.  This comes from the geometry 
of the planar triangulation. In general, we could have edge-labels of the form $(ay+b)^{r+1}$ with no $x$ term.  When the labeled graph is dual to a 
planar triangulation, we can without loss of generality rotate slightly through a vertex in the planar triangulation until no line segment in the 
triangulation is horizontal.  This is possible by the pigeonhole principal, since there is a (continuously) infinite number of angles we can rotate 
through but only a finite number of edges that can be horizontal.  Moreover, rotating the planar triangulation induces an isomorphism of splines 
since it induces a change of variables on the underlying coordinates $x$ and $y$.  If no edge in the triangulation is horizontal, then no 
edge-label in the dual graph has the form $(ay+b)^{r+1}$. 

We obtain the same result for edge-labelings that are not dual to a planar triangulation by changing variables according to the (invertible) operation
\[(x,y) \mapsto (x_0x-y_0y, y_0x+x_0y)\]
for some choice of scalars $x_0, y_0$.  The image of edge labels under this map is:
\[(ax+by+c)^2 \mapsto \left( (ax_0+by_0)x+ (bx_0-ay_0)y + c\right)^2.\]
We assume that at least one of $a, b$ is nonzero for each edge label, namely the edge label is not degenerate.  Thus by the same pigeonhole-principle argument as above, we can choose $x_0, y_0$ so that $ax_0+by_0 \neq 0$ simultaneously for all edge labels' coefficients $a, b$.  Since changing variables induces an isomorphism on splines, it preserves dimensions.

Finally, note that any nonzero multiple of an edge-label produces the same spline condition on that edge, since
\[
	p(u)-p(v)  = c \ell_{uv} \textup{ if and only if } p(u)-p(v) = \frac{c}{d} \left(d\ell_{uv}\right).
\]
(This is why it can be convenient to think of the edge-labels as ideals rather than elements of a ring.) In particular, if $(ax+by+c)^2$ is 
an edge-label with $a$ a nonzero scalar, we obtain the same collection of splines by replacing $(ax+by+c)^2$ with the edge-label 
$\left(x+\frac{b}{a}y+\frac{c}{a}\right)^2$.

We collect the results from this section in the following theorem.

\begin{theorem} \label{theorem: simplifying labeling}
Suppose that $G$ is a directed planar graph and that $\ell'$ is a face-translatable non-homogeneous edge-labeling of $G$ by polynomials of the form $(ax+by+c)^2$ with associated homogeneous labeling $\ell$.  Then 
\[\Spl_2(G,\ell') \cong \Spl_2(G,\ell)\]
as (real or complex) vector spaces. Moreover, via change of variables if necessary, we can assume that for each edge $uv$ the homogeneous edge label $\ell(uv)$ has the form $(x+by)^2$ for some scalar $b$.  

In particular, if $\Delta$ is a planar triangulation and $(G_{\Delta}, \ell')$ is the associated dual graph with its edge-labeling, we have an isomorphism of vector spaces
\[\Spl_2(G_{\Delta},\ell') \cong \Spl_2(G_{\Delta},\ell)\]
to the homogeneous splines obtained by ignoring all constant terms.
\end{theorem}

%%%%%%%%%%%%%%%%%%%%%%%%%%%%%%%%%%%%%%%%%%%%%%%%%%%%%%%%%%%%%%%
\section{Contracting and reducing to the minimal contractible case}
\label{section.contractible}

In this section, we examine certain numerical conditions on subgraphs of $G$ that are related to block-triangularity within the matrix $M^{\mathsf{ext}}$. We will then reinterpret the diagonal blocks of $M^{\mathsf{ext}}$ graph-theoretically as pairs of a subgraph and a complementary contracted graph, setting up a recursive algorithm to decompose splines on arbitrary graphs (which we detail in Section~\ref{section.algorithm}).  The graphs that serve as the fundamental building blocks of this algorithm are called minimal contractible.  We analyze them combinatorially in this section and describe the matrices $M^{\mathsf{ext}}$ corresponding to minimal contractible graphs in Section~\ref{section.all edge injective}.

Throughout this section, we assume that $G$ is a finite, directed, planar graph without leaves.

Recall the following definitions from graph theory.

\begin{remark}
\label{remark.contracted}
Suppose $G$ contains the edge $e=uv \in E(G)$. The \defn{contraction} of $G$ along the edge $e=uv$ is the graph $G \backslash \{uv\}$ with vertex set $V(G)/(u \sim v)$ in which $u$ and $v$ are identified, and edge set $E(G)-\{uv\}$.  If $G'$ is a subgraph of $G$ consisting of a subset of faces of $G$ together with the vertices and edges bounding those faces, then the \defn{contraction of $G$ along the subgraph $G'$} is the graph 
$G \backslash G'$ with vertex set 
\[
	V(G \backslash G') = \left(V(G) - V(G')\right) \cup \bigcup C_i,
\]
where the union is taken over connected components $C_i$ of $G'$, that is, the vertices of $G'$ are replaced by a new vertex for each connected component of $G'$. The edge set is
\[E(G \backslash G') = E(G)-E(G').\]  
In other words, the contraction $G \backslash G'$ is obtained by successively, for each edge $e=uv \in E(G')$, contracting  the vertices $u$ and $v$ to one vertex.
\end{remark}

%%%%%%%%%%%%%%%%%%%%%%%%%%%%%%%%%%%%%%%%%%%%%%%%%%%%%%%%%%%%%%%%%%%%%%
\subsection{Contractibility and minimal contractibility}
\label{section: minimal contractibility}

We now define \emph{contractibility}, which is a numeric condition on a set of faces $S$ together with their bounding vertices and edges.   
We then prove a set of startling results that together show contractibility sharply restrict what graphs can look like---especially so if the graph has no 
proper contractible subset of faces or if the graph is dual to a triangulation. In the next section, we collect these results to give an explicit description 
of the polynomials defining the algebraic variety for edge labels in special position.

\begin{definition}
\label{definition.contractible}
Suppose $S$ is a subset of the faces of the graph $G$. 
\begin{enumerate}
\item
Let $G_S$ be the corresponding subgraph of $G$ whose vertex set consists of all vertices on any face in $S$ and whose edge set consists 
of all edges on any face in $S$. 
\item
Let $e_S$ be the number of edges in $G_S$ and let $f_S$ be the number of faces in $G_S$.  Let $G \backslash G_S$ be the graph obtained from $G$ by 
contracting $G_S$.
\item
If $e_S - 3f_S \leqslant 0$, the subset $S$ is called a \defn{contractible subset of faces.}  
\item
If $e_S - 3f_S \leqslant 0$, but no proper subset $S' \subsetneq S$ satisfies $e_{S'} - 3f_{S'} \leqslant 0$, the subset $S$ is called a 
\defn{minimal contractible subset of faces.}  
\end{enumerate}
\end{definition}

\begin{remark} \label{remark.not-induced-subgraph}
    Note that $G_S$ is not necessarily an induced subgraph of $G$ in the graph-theoretic sense because if two vertices $u, v \in V(G_S)$ define an edge in $G$, 
but that edge is not on a face in $S$ then they are not neighbors in $G_S$.
\end{remark}

Contractibility actually constrains a graph considerably, as the next lemma describes.  The main point is that if a graph has a contractible subset of faces then we can find a contractible subset $S$ for which $e_S-3f_S \in \{0,-1,-2\}$ and if $S$ is minimal contractible then we can refine this even further.

\begin{lemma}
\label{lemma.minimal contractible}
  Suppose that $G$  is a finite, planar graph with a contractible subset of faces. Then  $G$ must have a contractible set $S$ of faces such that 
  \[e_S-3f_S \in \{0,-1,-2\}.\]
  
  Moreover suppose $S$ is a minimal contractible subset of faces. Then one of two cases holds:
  \begin{itemize} 
   \item If $G$ has at least two faces: Every face in $G$ has at most one edge on the boundary (that is, an edge not shared with another face) and 
   \[e_G - 3f_G \in \{0,-1\}.\]
 \item If $G$ has just one face, then that face is one of the following:
  \begin{itemize}
      \item a loop consisting of one vertex and an edge from the vertex to itself;
      \item a $2$-cycle consisting of two vertices with two edges between them;
      \item or a $3$-cycle consisting of three vertices with all possible edges between them.
  \end{itemize}
  \end{itemize}
\end{lemma}

\begin{proof}
We first find a contractible set of faces $S$ with $e_S - 3f_S \in \{0,-1,-2\}$.  

\smallskip

\noindent
{\bf Case 1:} If $S$ consists of a single face and $e_S - 3f_S \leqslant 0$, then $e_S \leqslant 3$ so the graph $G_S$ must be one of: 
\begin{itemize}
    \item a single vertex and edge from the vertex to itself, with $e_S - 3f_S = -2$;
    \item two vertices with two edges between them, with $e_S - 3f_S = -1$;
    \item three vertices in a $3$-cycle, with $e_S - 3f_S = 0$.
\end{itemize}

So assume that $S$ is a contractible set with at least two faces. If
$e_S-3f_S=0$,  then we are done. We may assume that $e_S-3f_S <0$ and furthermore that $e_T-3f_T>0$  
for all proper subsets $T\subset S$. Otherwise one may replace $S$ by $T$ in the following argument.  Thus we are in the next case.

\smallskip

\noindent
{\bf Case 2:} Suppose that $S$ contains at least two faces. 
We prove by contradiction that there is no face in $S$ with two boundary edges in $G_S$, namely that every face in $S$ has at most one edge 
that is not shared with another face in $S$.  Indeed, suppose there is a face $F$ in $S$ with $2$ boundary edges and let $T$ be the subset 
$S \setminus \{F\}$.  We have 
\[
	e_T-3f_T=(e_S-2)-3(f_S-1)=e_S-3f_S+1.
\]  
But $e_S-3f_S<0$ and so $e_T-3f_T\leqslant 0$.  This contradicts the minimality of $S$.  So every face in $S$ has at most one edge on the boundary 
of $G_S$.  Moreover, let $T$ be a boundary face of $G_S$.  Then the previous equation becomes
\[
	e_T-3f_T=(e_S-1)-3(f_S-1)=e_S-3f_S+2
\]  
which by minimality of $S$ must be positive.  So we have
\[
	e_S - 3f_S < 0 \textup{   and   } e_S - 3f_S + 2 > 0.
\]
In other words $e_S-3f_S = -1$ as desired.  This proves the claim.
\end{proof}

\begin{figure}[h]
\begin{center}
\scalebox{0.75}{\begin{tikzpicture}
\draw[thick] (-2,0) rectangle (3,3);
\draw[thick] (-1,1) rectangle (2,2);
\draw[thick] (4,0) rectangle (9,3);
\draw[thick] (5,1) rectangle (8,2);
\draw[thick] (-2,0) to (-1,1);
\draw[thick] (3,0) to (2,1);
\draw[thick] (3,3) to (2,2);
\draw[thick] (-2,3) to (-1,2);
\draw[thick] (4,0) to (5,1);
\draw[thick, color=black] (4,3) to (5,2);
\draw[thick, color=red] (4,3) to (5,2);
\draw[thick, color=red] (4,3) to (9,3);
\draw[thick, color=red] (4,3) to (4,0);
\draw[thick] (9,0) to (8,1);
\draw[thick] (9,3) to (8,2);
\node at (-1.9,0){\circle*{5}};
\node at (3.1,0){\circle*{5}};
\node at (-1.9,3){\circle*{5}};
\node at (3.1,3){\circle*{5}};
\node at (-0.9,1){\circle*{5}};
\node at (2.1,1){\circle*{5}};
\node at (-0.9,2){\circle*{5}};
\node at (2.1,2){\circle*{5}};
\node at (0,2){\circle*{5}};
\node at (1,2){\circle*{5}};
\node at (4.1,0){\circle*{5}};
\node at (4.1,3){\circle*{5}};
\node at (9.1,3){\circle*{5}};
\node at (5.1,1){\circle*{5}};
\node at (5.1,2){\circle*{5}};
\node at (8.1,1){\circle*{5}};
\node at (8.1,2){\circle*{5}};
\node at (9.1,0){\circle*{5}};
\node at (5.5,3){\circle*{5}};
\node at (7,3){\circle*{5}};

\end{tikzpicture}
}
 \caption{Two different planar embeddings of the same graph.\label{fig:1}}
\end{center}
\end{figure}

Minimal contractibility relies deeply on the choice of embedding in the plane, as demonstrated by the following example.

\begin{example}
Consider the graph in Figure~\ref{fig:1} with two different embeddings into the plane. The set of faces $S$ defined by the left hand embedding is minimal contractible. However, the set of faces $S$ in the right hand embedding is not minimal contractible as it has a subset $T$ of faces with red edges on part of their boundary satisfy $e_T-3f_T=0$. 
\end{example}

Nonetheless, with a slight additional condition, we can show that if the set of faces of a graph $G$ is minimal contractible then either $e_G=f_G$ or $G$ consists of a single face.  The additional condition---that any two (bounded) faces of $G$ share at most one edge---is satisfied trivially when the graph $G$ is dual to a planar graph like a triangulated region.  The next lemma is the key step and the subsequent corollary states the main takeaway.

\begin{lemma}
\label{lemma.min cont cond}
Let $G$ be a finite, planar graph without leaves such that two faces share at most one edge and so that the set of faces of $G$ is itself minimal contractible. If $e_G-3f_G=-1$, 
then $G$ is a 2-cycle consisting of two vertices with two edges between them.
\end{lemma}

\begin{proof}
If a vertex in $G$ has degree two then either it is on the boundary or two faces of $G$ share more than one edge.  Since no two faces of $G$ share more than one edge, and moreover the set of faces of $G$ is minimal contractible with $e_G-3f_G=-1$, it follows from
Lemma~\ref{lemma.minimal contractible} that no vertex of $G$ has degree two unless $G$ is the 2-cycle consisting of two vertices with two edges between them.
Assume that $G$ is not the 2-cycle. Since $G$ is planar, Euler's formula asserts that
\[
	v_G - e_G +f_G =1,
\]
where $v_G$ is the number of vertices of $G$. Substituting $e_G-3f_G=-1$, we obtain
\[
	v_G = 2f_G.
\]
Furthermore, counting the number of edges by vertex and using that each vertex has at least three edges and each edge is counted by two vertices, we have
\[
	2e_G \geqslant 3v_G = 6f_G,
\]
which yields the contradiction $6f_G-2 \geqslant 6f_G$ when inserting $e_G=3f_G-1$.
\end{proof}

\begin{corollary}
\label{corollary.min contract}
Let $G$ be a finite, planar graph without leaves such that two faces share at most one edge and so that the set of faces of $G$ is itself minimal contractible.
Then $G$ is a single face (loop, 2-cycle or 3-cycle) or $e_G-3f_G=0$.
\end{corollary}

\begin{proof}
This follows directly from Lemmas~\ref{lemma.minimal contractible} and~\ref{lemma.min cont cond}.
\end{proof}

One case of particular interest is when the graph $G_{\Delta}$ is dual to a planar triangulation $\Delta$.  In this case, we can use Corollary~\ref{corollary.min contract} to simplify considerably and show that if $G_{\Delta}$ is contractible then $\Delta$ is the subdivision of a single triangle.

\begin{corollary}\label{corollary: contractible triangulation means subdivided triangle}
Suppose that $\Delta$ is a triangulation of a region in the plane for which the dual graph $G_{\Delta}$ is contractible.  Then $\Delta$ is a subdivided triangle, meaning that the boundary of $\Delta$ as a planar graph is a triangle.
\end{corollary}

\begin{proof}
We start by assuming that $G_{\Delta}$ is minimal contractible.  A triangle cannot share an edge with itself, nor can two triangles share more than one edge in a planar triangulation. Corollary~\ref{corollary.min contract} says $G_{\Delta}$ has three times as many edges as faces.  By construction of the dual graph, this means that in $\Delta$ the number of interior edges $e_{\mathsf{int}}$ and interior vertices $v_{\mathsf{int}}$ are related by
\[
	e_{\mathsf{int}} = 3v_{\mathsf{int}}.
\]
In addition, when $\Delta$ is considered as a planar graph, every face has three edges.  This counts interior edges twice and boundary edges once so the number of faces and edges in $\Delta$ satisfy:
\[
	3f = 2e_{\mathsf{int}} + e_b.
\]
Euler's formula for planar graphs holds for $\Delta$ as well, namely:
\[
	(v_{\mathsf{int}}+v_b) - (e_{\mathsf{int}}+e_b)+f=1,
\]
where $v_b$ refers to the number of boundary vertices.  Substituting for $v_{\mathsf{int}}$ and $f$ from the previous equations and then simplifying, we obtain
\[
	\frac{e_{\mathsf{int}}}{3} + v_{b} - e_{\mathsf{int}} - e_b + \frac{2e_{\mathsf{int}}}{3} + \frac{e_b}{3} = v_b - \frac{2e_b}{3} = 1.
\]
But the boundary of $\Delta$ is a cycle so $v_b= e_b$.  We conclude that $v_b = 3$ and hence the boundary of $\Delta$ is a triangle, as claimed.

Now suppose that $\Delta$ is a triangulation for which $G_{\Delta}$ is contractible.  Let $S$ be a minimal contractible subset of faces in $G_{\Delta}$ 
and consider the subtriangulation $\Delta_S$ of $\Delta$ dual to the subgraph $G_S$.  By our previous argument, the boundary of $\Delta_S$ is a 
triangle.  Erasing the edges and vertices interior to $\Delta_S$ does not change the triangulation $\Delta$ outside of $\Delta_S$.  In particular, the plane 
partition $\Delta'$ that we obtain is still a triangulation.  Inside $G_{\Delta}$, this operation erases all interior edges in $S$ and replaces all interior 
vertices in $S$ with one single new vertex $\widetilde{v}$ together with edges $u\widetilde{v}$ if and only if $uv$ was an edge in $G_{\Delta}$ for 
a vertex $v$ on a face in $S$ and a vertex $u$ disjoint from $S$.  In other words, the graph $G_{\Delta'}$ is exactly the same as the graph obtained 
from $G_{\Delta}$ by contracting the faces in $S$ to a vertex.  Hence $G_{\Delta'}$ is still contractible.  The triangulation $\Delta$ is finite so the claim 
holds by induction.
\end{proof}

The special case when $G_{\Delta}$ is dual to a triangulation satisfies an additional unusual property.

\begin{corollary} \label{corollary: special polynomials for triangulations}
If $\Delta$ is a plane triangulation and $G_{\Delta}$ its dual graph, then:
\begin{itemize}
    \item no two minimal contractible subgraphs of $G_{\Delta}$ share an edge, and
    \item if $G_S$ is a minimal contractible subgraph of $G_{\Delta}$ and $\Delta_S \subseteq \Delta$ is the subdivision of a single triangle dual to $G_S$ in $\Delta$ then the planar triangulation $\Delta'$ obtained from $\Delta$ by erasing the subdivision inside $\Delta_S$ has dual graph $G/G_S$.
\end{itemize} 
\end{corollary}

\begin{proof}
Suppose that $S_1$ and $S_2$ are two different minimal contractible subsets of faces in $G_{\Delta}$.  We will show that $S_1$ and $S_2$ do not 
share an edge. Corollary~\ref{corollary: contractible triangulation means subdivided triangle} says that all minimal contractible subgraphs in $G_{\Delta}$ 
arise as the dual graph to a subdivision of a single triangle in $\Delta$. Denote these by $\Delta_1$ and $\Delta_2$ respectively, meaning the faces in 
$S_i$ correspond to the interior vertices of $\Delta_i$. Denote the boundary of each $\Delta_i$ by $\partial \Delta_i$ and note that $\partial \Delta_i$ is a 
three-cycle.  

We will show that $\Delta_1$ and $\Delta_2$ do not share a face, which implies that $G_{S_1}$ and $G_{S_2}$ do not share an edge.  Suppose the 
vertices of the boundary triangle $\partial \Delta_1$ are $v_1, v_2, v_3$. Every triangulation in the plane is a simplicial decomposition so if 
$(\partial \Delta_1) \cap \Delta_2$ contains an edge $\{v_iv_j\}$ then it contains both endpoints $v_i$ and $v_j$.  Conversely, the triangle $\Delta_1$ 
is convex so if $(\partial \Delta_1) \cap \Delta_2$ contains two vertices $v_i$ and $v_j$ then it also contains the edge $v_iv_j$ between them.  This 
gives the possibilities in the first column of the following table:
\[\begin{array}{|l|l|}
\cline{1-2} (\partial \Delta_1) \cap \Delta_2 & \textup{ configuration of } \Delta_1 \textup{ and } \Delta_2 \\
\cline{1-2} \emptyset & \textup{ disjoint } \\
\{v_i\} & \textup{ intersect at a single vertex like a bowtie}\\
\{v_i, v_j, v_iv_j\} & \textup{ intersect along a single edge to form a square with a diagonal} \\ 
\partial \Delta_1 & \textup{ not allowed}\\
\hline \end{array}\]
The boundary $\partial \Delta_2$ is a simple closed curve so by the Jordan curve theorem defines an inside and an outside.  The boundary 
$\partial \Delta_1$ enters and exits $\Delta_2$ at the boundary $\partial \Delta_2$.  Since $G_{S_1}$ and $G_{S_2}$ are distinct minimal 
contractible subgraphs, neither of $\Delta_1, \Delta_2$ contains the other.  Comparing with the possible intersections $(\partial \Delta_1) \cap \Delta_2$ 
above means that $\Delta_1$ and $\Delta_2$ are in one of the configurations in the second column of the table above. In none of these cases do the 
dual graphs $G_{S_1}$ and $G_{S_2}$ share an edge (nor even a vertex).

Finally, suppose $G_S$ is a minimal contractible subgraph.  By Corollary~\ref{corollary: contractible triangulation means subdivided triangle}, we 
know $G_S$ is dual to a subdivided triangle $\Delta_S$ in $\Delta$. Let $\Delta'$ be the triangulation obtained from $\Delta$ by erasing the 
triangulation interior to $\Delta_S$.  The dual graph $G_{\Delta'}$ has one vertex $v'$ where $G_{\Delta}$ has a subgraph $G_{S}$ and three edges 
incident to $v'$ to the vertices corresponding to the three faces sharing an edge with the boundary of $\Delta_S$ with exactly the edge-labels as in 
$G_{\Delta}$.  In other words, the dual graph $G_{\Delta'}$ is exactly $G_{\Delta}/G_S$.
\end{proof}

Section~\ref{section: open questions} poses a question about how these results change on surfaces of higher genus.

%%%%%%%%%%%%%%%%%%%%%%%%%%%%%%%%%%%%%%%%%%%%%%%%%%%%%%%%%%%%%%%%%%%%%%%%%%%%%
\section{Edge-injective functions and generic edge-labelings: first properties}
%%%%%%%%%%%%%%%%%%%%%%%%%%%%%%%%%%%%%%%%%%%%%%%%%%%%%%%%%%%%%%%%%%%%%%%%%%%%%

In this section, we use both combinatorial and topological tools for preliminary calculations of $\rank M^{\mathsf{ext}}$.
We begin by defining generic edge labels for a graph $G$ in Section~\ref{section.generic},
which basically means that the rank of $M^{\mathsf{ext}}$ is maximal. In Section~\ref{section.edge injective}, we then analyze the determinant of the maximal 
square submatrix of $M^{\mathsf{ext}}$ in terms of edge-injective functions on the graph $G$, which are assignments of directions of the edges
in the dual graph $G^\star$ satisfying certain properties. We then give essential decomposition results showing how the notion of contractibility 
from Section~\ref{section.contractible} can be used with edge-injective functions to reduce our problem to computing the rank of $M^{\mathsf{ext}}$ for 
graphs that do not contain any contractible subsets of faces.

%%%%%%%%%%%%%%%%%%%%%%%%%%%%%%%%%%%%%%%%%%%%%%%%%%%%%%%%%%%%%%%
\subsection{Generic edge labels: definition and examples}
\label{section.generic}

We assume in this section that $G$ is a finite, directed, planar graph without leaves.
Let $e_1,e_2,\ldots, e_{e_G}$ be the edges in $G$ and $\ell_i = (x+a_iy)^2$ the label
for edge $e_i$ for $1\leqslant i \leqslant e_G$. The entries in $M^{\mathsf{ext}}$ are in $\{0,\pm 1,\pm a_i,\pm a_i^2 \mid 1\leqslant i \leqslant e_G\}$. 
Hence for any maximal square submatrix $N$ of $M^{\mathsf{ext}}$, the function $\det N$ is polynomial in the variables $a_i$ for 
$1\leqslant i \leqslant e_G$. We now define  \defn{generic edge labels}, which is slightly different from \emph{maximal rank}.

\begin{definition}
\label{definition.generic edge labels}
We call the edge labels $\ell_i = (x+a_iy)^2$ for $1\leqslant i \leqslant e_G$ in a labeled, directed graph $G$ \defn{generic} if there exists a 
$d \times d$ submatrix $N$ of $M^{\mathsf{ext}}$ with $d = \min\{e_G,3f_G\}$ such that
$\det N$ is not identically zero as a polynomial $p(a_1,\ldots,a_{e_G})$ in the variables $a_1, a_2, \ldots, a_{e_G}$.  The edge labels are in 
\defn{special position} if the values for the $a_i$'s are roots of the polynomial $p(a_1,\ldots,a_{e_G})$.
\end{definition}

Throughout this paper, we assume the $a_i$'s are \emph{not} in special position, that is, if the polynomial $p(a_1,\ldots,a_{e_G})$ is 
not the zero polynomial, then the $a_i$'s are not roots of the polynomial.

Note that unless the polynomial in the $a_i$'s given by $\det N$ is identically zero, the condition $\det N=0$ defines an algebraic 
hypersurface. In geometric terms, this means  the corresponding triangulation is in ``special position'' (which agrees with the above definition of
special position as the roots of the nonzero polynomial).

For generic edge labels not in special position we have 
\begin{equation}
\label{equation.generic rank}
	\rank M^{\mathsf{ext}} = \min\{e_G, 3f_G\}.
\end{equation}
In the next sections, we analyze graphs with generic edge labels.  First, we illustrate \emph{generic} and \emph{special position}.

\begin{example}
Suppose $G$ is a triangle with edges labeled $(x+a_1y)^2, (x+a_2y)^2, (x+a_3y)^2$ directed clockwise. In this case the matrix 
$M^{\mathsf{ext}}$ is square so that there is only one choice for $N$.  We have
\[
	\det N = \det M^{\mathsf{ext}} =  (a_1-a_2)(a_2-a_3)(a_1-a_3).
\]
Hence the edge labels are generic in this case since the polynomial is not identically zero for distinct $a_1, a_2, a_3$.
If any two of the $a_i$ coincide, then $\det N = 0$ and the edge-labeling is in special position.
\end{example}

\begin{example} 
\label{example: Morgan-Scott}
A more complicated example is the case of three 4-cycles, each pair of which shares an edge, and that share an interior vertex in common as shown
in Figure~\ref{figure.3squares}. This is the dual of the triangulation of Figure~\ref{figure.triangulation}. Assume that the edge labels are $\ell(e_i)=(x+a_iy)^2$ 
for $1\leqslant i \leqslant 9$.

\begin{figure}[h]
\begin{center}
\begin{tikzpicture}[xscale=1.5, yscale=0.6]
\draw[thick] (0,2) to (0,0);
\draw[thick] (0,0) to (1.5,-1.32);
\draw[thick] (0,0) to (-1.5,-1.32);
\draw[thick] (-1.5,0.68) to (-1.5,-1.32);
\draw[thick] (1.5,0.68) to (1.5,-1.32);
\draw[thick] (0,2) to (1.5,0.68);
\draw[thick] (0,2) to (-1.5,0.68);
\draw[thick] (0,-2.64) to (1.5,-1.32);
\draw[thick] (0,-2.64) to (-1.5,-1.32);
\node at (-0.85,1.7) {$e_1$};
\node at (-1.7,-0.2){$e_2$};
\node at (-0.2,0.7){$e_3$};
\node at (-0.65,-0.95){$e_4$};
\node at (-0.2,-2){$e_5$};
\node at (1.1,-2.1){$e_6$};
\node at (1,-0.5){$e_7$};
\node at (0.6,1.9){$e_8$};
\node at (1.7,0){$e_9$};
\node at (0,)[anchor=south east]{};
\node at (0,2){$\bullet$};
\node at (0,0){$\bullet$};
\node at (1.5,-1.32){$\bullet$};
\node at (-1.5,-1.32){$\bullet$};
\node at (0,-2.64){$\bullet$};
\node at  (-1.5,0.68){$\bullet$};
\node at (1.5,0.68){$\bullet$};
\end{tikzpicture}
\end{center}
\caption{Three 4-cycles in Example~\ref{example: Morgan-Scott} and dual to the
triangulation in Figure~\ref{figure.triangulation}.
\label{figure.3squares}}
\end{figure}

In this example, $M^{\mathsf{ext}}$ is a square matrix and we have (independent of any directions on edges)
\[
	\det N = \det M^{\mathsf{ext}}=(a_1-a_2)(a_5-a_6)(a_8-a_9)P(a_1,\ldots,a_9),
\]
where $P=(a_2-a_3)(a_3-a_1)(a_4-a_5)(a_6-a_4)(a_7-a_8)(a_9-a_7)-(a_2-a_4)(a_4-a_1)(a_6-a_7)(a_7-a_5)(a_3-a_8)(a_9-a_3)$ is a polynomial in 
variables $a_1,\ldots,a_9$. Here is one possible set of conditions under which this determinant is non-zero and hence the edge labels are 
not in special position:
\begin{itemize}
    \item $a_1\neq a_2,\ a_5\neq a_6,\ a_8\neq a_9$;
    \item $(a_3,a_4)\notin\{a_1,a_2,a_8,a_9\}\times\{a_1,a_2, a_5,a_6\} $;
    \item $(a_3,a_7)\notin\{a_1,a_2,a_8,a_9\}\times\{a_5,a_6,a_8,a_9\} $;
    \item $(a_4,a_7)\notin\{a_1,a_2,a_5,a_6\}\times\{a_5,a_6, a_8,a_9\} $.
\end{itemize}
\end{example}

\begin{example}
\label{non-generic graph}
An example of a non-generic graph is given in Figure \ref{figure.non-generic}.
\begin{figure}[h]
\begin{center}
\begin{tikzpicture}[xscale=2, yscale=0.5]
\draw[thick] (0,-2) to (2,-2);
\draw[thick] (0,2) to (2,2);
\draw[thick] (0,0) to (0,2);
\draw[thick] (0,0) to (0,-2);
\draw[thick] (2,0) to (2,-2);
\draw[thick] (2,0) to (2,2);
\draw[thick] plot [smooth cycle] coordinates {(0,0) (1,0.5) (2,0) (1,-0.5)};
\node at (1,2.35) {$e_9$};
\node at (1,0.85){$e_1$};
\node at (1,-0.85){$e_2$};
\node at (-0.2,-1){$e_3$};
\node at (-0.2,1){$e_7$};
\node at (2.2,-1){$e_4$};
\node at (2.2,1){$e_8$};
\node at (0.5,-2.4){$e_5$};
\node at (1,-2){$\bullet$};
\node at (1.5,-2.4){$e_6$};
\node at (0,-2){$\bullet$};
\node at (2,-2){$\bullet$};
\node at (0,2){$\bullet$};
\node at (0,-2){$\bullet$};
\node at (0,0){$\bullet$};
\node at (2,0){$\bullet$};
\node at (2,2){$\bullet$};
\end{tikzpicture}
\end{center}
\caption{Non-generic graph
discussed in Example~\ref{non-generic graph}.
\label{figure.non-generic}}
\end{figure}
In this case, the determinant of $M^{\mathsf{ext}}=N$ is zero because the three rows corresponding to the face bounded by $e_1, e_2$ have only 
two nonzero columns. Thus the graph has no generic edge labels with our definition of the term. However, note that the subgraph $G'$ consisting 
just of the edges $e_1$ and $e_2$ has a generic edge labeling, since $\min\{e_{G'},3f_{G'}\}$ 
is the rank of the corresponding cycle basis matrix in that case. 

However, it is still possible to find edge labels of the graph in Figure~\ref{figure.non-generic} for which $M^{\mathsf{ext}}$ is maximal rank.  For instance, $\rank M^{\mathsf{ext}}=8$ for edge labelings $\ell$ with distinct edge labels, and thus $\dim \mathsf{Spl}_2(G,\ell; v_0) = 1.$ 

Despite the fact that this graph has no generic edge labels, the algorithm we give in Section~\ref{section.algorithm} computes the dimension
of the graph's splines.
\end{example}

%%%%%%%%%%%%%%%%%%%%%%%%%%%%%%%%%%%%%%%%%%%%%%%%%%%%%%%%%%%%%%%
\subsubsection{The case $e_G \leqslant 3f_G$}
\label{section.trivial spline}

We generalize the example of Figure~\ref{figure.non-generic} as follows.  

\begin{lemma}
Suppose $G$ is a graph with $e_G \leqslant 3f_G$ and $\ell$ is a generic edge-labeling, as defined in Definition~\ref{definition.generic edge labels}.  Then
\[ \dim \mathsf{Spl}_2(G,\ell; v_0) = 0.\]
\end{lemma}

\begin{proof}
Since $e_G \leqslant 3f_G$ we have $\mathsf{rank} M^{\mathsf{ext}} \leqslant e_G$. In particular, for edge labels that are generic as in 
Definition~\ref{definition.generic edge labels}, we have $\rank M^{\mathsf{ext}}=e_G$.  The result follows from Theorem~\ref{theorem.spline rank}.
\end{proof}

%%%%%%%%%%%%%%%%%%%%%%%%%%%%%%%%%%%%%%%%%%%%%%%%%%%%%%%%%%%%%%%
\subsection{Determinant conditions when $e_G \geqslant 3f_G$}
\label{section.edge injective}

In this section we assume that $G$ is a finite, directed, planar graph without leaves. For convenience, we set $f:= f_G$.  We also assume that 
$e_G \geqslant 3f_G$ since we addressed the trivial case $e_G \leqslant 3f_G$ in Section~\ref{section.trivial spline}.

To determine $\rank M^{\mathsf{ext}}$, we compute the determinant of square submatrices of $M^{\mathsf{ext}}$ obtained by  deleting $e_G-3f_G$ columns. 
Let $N=(n_{i,j})$ be any $3f \times 3f$ square matrix. Then
\begin{equation}
\label{equation.determine}
	\det N = \sum_{\sigma \in S_{3f}} (-1)^\sigma n_{1,\sigma(1)} n_{2,\sigma(2)} \cdots n_{3f,\sigma(3f)},
\end{equation}
where $S_{3f}$ is the group of permutations on $3f$ elements.
We reinterpret this sum in terms of edges and faces.  The summand $n_{1,\sigma(1)} n_{2,\sigma(2)} \cdots n_{3f,\sigma(3f)}$ is identically 
zero unless there is a {\em nonzero} entry $n_{i,\sigma(i)}$ in each row $i$.
Each column corresponds to an edge and each face appears as exactly three rows with columns of forms $\pm 1, \pm a_i, \pm a_i^2$.  The permutation 
$\sigma$ chooses a distinct edge $\sigma(i)$ for each row $i$. So $\sigma$ 
assigns to each face a triple of edges, none of which are associated to any other face.

We will show that given any nonzero term in~\eqref{equation.determine}, there in fact exists a function
\[
	\varphi \colon \mathcal{F} \to \mathcal{E}_3,
\]
where $\mathcal{E}_3$ is the collection of {\em unordered} triples of edges such that
\begin{itemize}
    \item for each face $F$ all edges in $\varphi(F)$ are on the boundary of $F$ and
    \item $\varphi(F) \cap \varphi(G) = \emptyset$ for $F,G \in \mathcal{F}$ for $F\neq G$.
\end{itemize}
We call such a map \defn{edge-injective}.  Any edge-injective function $\varphi$ defines a square matrix $N_{\varphi}$ by restricting $M^{\mathsf{ext}}$ to 
the columns indexed by the image of $\varphi$. We use the conventions that
\begin{itemize}
\item columns are ordered in $N_{\varphi}$ according to their order within $M^{\mathsf{ext}}$ and
    \item rows associated to each face are adjacent to each other and ordered by 
increasing exponent, namely $\pm 1, \pm a_i, \pm a_i^2$.
\end{itemize}  

The following observation is key to our analysis of the matrices $N_{\varphi}$.

\begin{lemma}
\label{lemma.edge injective}
Fix an edge-injective function $\varphi \colon \mathcal{F} \to \mathcal{E}_3$ and let $N_{\varphi}$ be the associated square submatrix of $M^{\mathsf{ext}}$.  
Let $F$ be a face, $i$ be any row in $N_{\varphi}$ corresponding to $F$, and $j$ be the column corresponding to any $e \in \varphi(F)$. Then the entry 
$n_{i,j}$ is nonzero.
\end{lemma}

\begin{proof} 
Since $j$ is the column associated to an edge $e$ on face $F$, it has entries $\pm 1, \pm a_j, \pm a_j^2$ in the three rows associated to $F$.  
Thus $n_{i,j}$ is one of the entries $\pm 1, \pm a_j, \pm a_j^2$ and so is nonzero.
\end{proof}

We now show that the determinant of an arbitrary square $3f \times 3f$ submatrix $N$ of $M$ actually decomposes into a sum over the edge-injective 
functions $\varphi$ with $N = N_{\varphi}$.  In addition, each term in this sum is a product of determinants of smaller submatrices.

More precisely, we have the following.

\begin{proposition} 
\label{proposition.factoring det into edge-injective}
Suppose that $N$ is any square $3k \times 3k$ submatrix of $M^{\mathsf{ext}}$ obtained from restricting to a subset of columns and a subset of faces.
For each face $F$ all three rows consisting of values $\pm 1, \pm a_j, \pm a_j^2$ are selected. Then
\[
	\det N = \sum_{\substack{\varphi \colon \mathcal{F} \to \mathcal{E}_3\\ \text{edge-injective}\\\text{with }N_{\varphi}=N}}  
	\prod_{F \in \mathcal{F}} \det N_{F,\varphi(F)},
\]
where $N_{F,\varphi(F)}$ is the submatrix of $N_{\varphi}$ with the 3 rows corresponding to $F$ and columns indexed by the edges in $\varphi(F)$.
\end{proposition}

\begin{proof}
Conversely to Lemma~\ref{lemma.edge injective}, for each nonzero term in~\eqref{equation.determine} and for each face $F \in \mathcal{F}$
corresponding to rows $i_1, i_2, i_3$ in $N_{\varphi}$, let $e_{i_1}, e_{i_2}, e_{i_3}$ be the edges corresponding to the 
columns $\sigma(i_1) ,\sigma(i_2), \sigma(i_3)$. Define $\varphi(F) = \{e_{i_1}, e_{i_2}, e_{i_3}\}$.

Consider the permutations in $S_{3\ell}$ that permute rows associated to each face but not rows associated to different faces.  Note that this is a subgroup 
$G \subseteq S_{3\ell}$ isomorphic to the product $\prod_{F \in \mathcal{F}} S_3$ of smaller symmetric groups.  

Furthermore, Lemma~\ref{lemma.edge injective} states that if $\pi \in G$ then the term $n_{1,\sigma(1)}\cdots n_{\ell,\sigma(3\ell)}$ is nonzero if and only 
if $n_{1, \pi \sigma(1)} \cdots n_{\ell, \pi \sigma(3\ell)}$ is also nonzero.  Moreover $(-1)^{\pi \sigma} = (-1)^\pi (-1)^\sigma$.  

Hence we may rewrite~\eqref{equation.determine} by distributing as follows:
\[
	\det N = \sum_{\substack{\varphi \colon \mathcal{F} \to \mathcal{E}_3\\ \text{edge-injective}\\ \text{with }N=N_{\varphi}}}  
	\prod_{F \in \mathcal{F}} \det N_{F,\varphi(F)},
\]
where $N_{F,\varphi(F)}$ is the submatrix of $N_{\varphi}$ with the $3$ rows corresponding to $F$ and columns indexed by the edges 
in $\varphi(F)$. 
\end{proof}

Let $e_1, e_2, e_3$ be the edges in $\varphi(F)$ listed in the order they appear in $M^{\mathsf{ext}}$ and suppose their labels are 
$\ell_i = (x+ a_i y)^2$ respectively for $i=1,2,3$. Then $\det N_{F,\varphi(F)}$ is the Vandermonde determinant associated to 
\[
	N_{F,\varphi(F)} = \begin{pmatrix}
	\varepsilon_1  & \varepsilon_2  &\varepsilon_3 \\
	 \varepsilon_1 a_1 & \varepsilon_2 a_2 & \varepsilon_3 a_3\\
	 \varepsilon_1 a_1^2 & \varepsilon_2 a_2^2 & \varepsilon_3 a_3^2
	 \end{pmatrix}
\]
for some choices $\varepsilon_i \in \{\pm 1\}$.  In other words,
\begin{equation}
\label{equation.Vandermonde}
	\det N_{F,\varphi(F)} = \pm (a_1 - a_2)(a_2 - a_3)(a_3 - a_1).
\end{equation}

\begin{corollary} \label{corollary.unique edge-injective function}
Suppose that $M^{\textsf{ext}}$ has a $3f_G \times 3f_G$ square submatrix $N$ obtained by erasing certain columns and that there exists a unique 
edge-injective function $\varphi \colon \mathcal{F} \to \mathcal{E}_3$ with $N_{\varphi} = N$.  Then there is a generic edge labeling for which
$\det N \neq 0$ or equivalently $\rank M^{\textsf{ext}} = 3f_G$.
In particular, if all edge-labels in $\varphi(F)$ are distinct for all faces $F$ in $G$ and there is a single edge-injective function,
then $M^{\textsf{ext}}$ is full-rank.
\end{corollary}

\begin{proof}
Under these hypotheses, Proposition~\ref{proposition.factoring det into edge-injective} expresses the determinant of $N$ as the product
\[
	\prod_{F \in \mathcal{F}} \det N_{F,\varphi(F)}.
\]
Moreover, each $\det_{F,\varphi(F)}$ is a Vandermonde determinant and thus is nonzero if and only if the variables $a_i$ are distinct.
\end{proof}

We will later show that the existence of an edge-injective function $\varphi$ is enough to guarantee a generic edge labeling.  Our edge-injective functions have a similar flavor (in the dual sense) to Whiteley's 3-fans in his unpublished work~\cite[p. 16]{Wh.unpublished}.
See also~\cite{WW.1983,WW.1987}.

\subsection{Edge-injective functions and contractions} \label{section: edge-injective and contractions, first results}

We now combine these results with the idea of contractions and contractibility from Section~\ref{section: minimal contractibility}. Our first corollary factors the determinant formula of Proposition~\ref{proposition.factoring det into edge-injective} in terms of subgraphs and contracted subgraphs.

\begin{corollary}
\label{corollary.factor}
Let $G$ be a finite, directed, connected, planar graph without leaves such that $e_G \geqslant 3 f_G$.
Suppose that $G$ has a subgraph $G'$ consisting of a union of $f'$ faces in $G$ together with the vertices and $e'$ edges bounding those faces, and 
suppose that $e' \leqslant 3f'$.  Suppose that $N$ is any $3f_G \times 3f_G$ submatrix of $M^{\mathsf{ext}}$ obtained by removing some columns 
from $M^{\mathsf{ext}}$.  Then 
\begin{enumerate}
\item $\det N = 0$ if $e' < 3f'$, or
\item $\det N = \det N_{G'} \det N_{G\backslash G'}$ if $e'=3f'$, where $N_{G'}$ is the matrix $M_{G'}^{\mathsf{ext}}$ and $N_{G\backslash G'}$ is the matrix 
obtained from $M^{\mathsf{ext}}$ by restricting to faces in $G$ that are not contained in $G'$ and edges in $G$ that do not bound a face in $G'$. 
\end{enumerate}
Moreover the matrix $N_{G\backslash G'}$ in Part (2) is the matrix associated to the graph  $G\backslash G'$ obtained from $G$ by contracting the edges in $G'$.
\end{corollary}

\begin{proof}
Every edge-injective function $\varphi: \mathcal{F} \to \mathcal{E}_3$ must send faces in $G'$ 
to edges in $G'$ so it restricts to an edge-injective function on $G'$.  Conversely since $G'$ has $f'$ faces and at most $3f'$ edges, every edge-injective 
function on $G'$ uses up all those edges.  In particular, no edge-injective function on $G$ has an edge of $G'$ in the image of a face outside of $G'$.

In other words, the set of edge-injective functions $\mathcal{F} \to \mathcal{E}_3$ factors as
\[
	\left( \varphi: \mathcal{F} \to \mathcal{E}_3 \right) \hspace{1em} \longleftrightarrow \hspace{1em}
	\left( \varphi: \mathcal{F}_{G'} \to \mathcal{E}_{G'} \right) \times 
	\left( \varphi: \mathcal{F}_{G \backslash G'} \to \mathcal{E}_{G \backslash G'} \right).
\]
In particular, if $e'<3f'$, there is no edge-injective function $\varphi: \mathcal{F}_{G'} \to \mathcal{E}_{G'}$ and hence no edge-injective function
$\varphi: \mathcal{F} \to \mathcal{E}_3$. This implies by Corollary~\ref{proposition.factoring det into edge-injective} that $\det N =0$.
If $e'=3f'$, by Corollary~\ref{proposition.factoring det into edge-injective} and definition of the graph $G \backslash G'$ the claim follows. 
\end{proof}

\begin{corollary}
\label{corollary.rank decomp}
Let $G$ be a finite, directed, connected, planar graph without leaves and with $e_G \geqslant 3 f_G$.
Suppose that $G$ has a subgraph $G'$ with generic edge labels consisting of a union of $f'$ faces in $G$ together with the vertices and $e'$ 
edges bounding those faces, and suppose that $e' \leqslant 3f'$. Let $M^{\mathsf{ext}}_{G'}$ be the $3f'\times e'$ rectangular submatrix corresponding 
to $G'$ and $M_{G \backslash G'}^{\mathsf{ext}}$ the extended cycle basis matrix corresponding to the contracted graph $G \backslash G'$. 
Then 
\[
	\rank M^{\mathsf{ext}} = \rank M^{\mathsf{ext}}_{G'} + \rank M^{\mathsf{ext}}_{G\backslash G'}.
\]
\end{corollary}

\begin{proof}
By construction of $G'$, the extended cycle matrix is block diagonal: 
\[
	M^\mathsf{ext}_G = \left( \begin{array}{cc} M^\mathsf{ext}_{G'} & 0 \\ * & M^\mathsf{ext}_{G\backslash G'} \end{array} \right).
\]
This proves that the rank is additive.
\end{proof}

\begin{corollary}
\label{corollary.dim decomp}
Under the same assumptions as Corollary~\ref{corollary.rank decomp} and assuming that the subgraph $G'$ has generic edge labels, we have
\[
	\dim \mathsf{Spl}_2(G,\ell; v_0) = \dim \mathsf{Spl}_2(G \backslash G',\ell; v_0).
\]
\end{corollary}

\begin{proof}
For generic edge labels on $G'$, we have $\rank M_{G'}^{\mathsf{ext}} = e'$ and $\dim \mathsf{Spl}_2(G',\ell; v_0)=0$ by Section~\ref{section.trivial spline}.
By Corollary~\ref{corollary.rank  decomp} and Theorem~\ref{theorem.spline rank} the result follows.
\end{proof}

\begin{example}
Consider the planar graph $G$ on the left, below.
\begin{center}
\begin{tikzpicture}[scale=0.5]
\draw (0,2) to (2,4);
\draw (0,2) to (2,0);
\draw (2,4) to (4,4);
\draw (2,0) to (4,0);
\draw (4,0) to (4,4); 
\draw (4,0) to[out=20, in=5] (4,4);
\node at (3,4.5) {$e_1$};
\node at (0.8,3.5){$e_2$};
\node at (0.8,0.5){$e_3$};
\node at (3,-0.7){$e_4$};
\node at (3.5,2){$e_5$};
\node at (5.6,2){$e_6$};
\node at (0,)[anchor=south east]{};
\node at (0,2){$\bullet$};
\node at (2,4){$\bullet$};
\node at (2,0){$\bullet$};
\node at (4,0){$\bullet$};
\node at (4,4){$\bullet$};
\end{tikzpicture}
\hspace{1in}
\raisebox{0.6in}{$M^{\mathsf{ext}}=
	\begin{pmatrix}
	1&1& 1& 1 & 1 & 0 \\
	a_1&a_2& a_3& a_4 & a_5 & 0 \\
	a_1^2&a_2^2& a_3^2& a_4^2 & a_5^2 & 0\\
	0&0&0&0&-1&-1\\
	0&0&0&0&-a_5&-a_6\\
	0&0&0&0&-a_5^2&-a_6^2\\
	\end{pmatrix}$}
\end{center}
This graph has two faces $F_1$ and $F_2$ with face cycles $e_1 e_2 e_3 e_4 e_5$ and $e_5 e_6$, respectively. Directing the edges clockwise
for face $F_1$ and anticlockwise for face $F_2$ gives the extended cycle basis matrix on the right.  The subgraphs $G'$ and $G \backslash G'$ are respectively
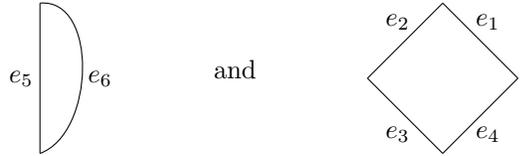
\begin{center}
\begin{tikzpicture}[scale=0.5]
\draw (4,0) to (4,4); 
\draw (4,0) to[out=20, in=5] (4,4);
\node at (3.5,2){$e_5$};
\node at (5.6,2){$e_6$};
\node at (0,)[anchor=south east]{};
\node at (4,0){$\bullet$};
\node at (4,4){$\bullet$};
\end{tikzpicture}
\hspace{1cm} \raisebox{1cm}{and} \hspace{1cm}
\begin{tikzpicture}[scale=0.5]
\draw (0,2) to (2,4);
\draw (0,2) to (2,0);
\draw (2,4) to (4,2);
\draw (2,0) to (4,2);
\node at (3.2,3.5) {$e_1$};
\node at (0.8,3.5){$e_2$};
\node at (0.8,0.5){$e_3$};
\node at (3.2,0.5){$e_4$};
\node at (0,)[anchor=south east]{};
\node at (0,2){$\bullet$};
\node at (2,4){$\bullet$};
\node at (2,0){$\bullet$};
\node at (4,2){$\bullet$};
\end{tikzpicture}
\raisebox{1cm}{.}
\end{center}
By case (1) of Corollary~\ref{corollary.factor} we have $\det M^{\mathsf{ext}} = \det N=0$.
By Corollary~\ref{corollary.dim decomp} we have $\dim \mathsf{Spl}_2(G,\ell; v_0) = \dim \mathsf{Spl}_2(G \backslash G',\ell;v_0)=1$ for generic edge 
labels on $G'$.
\end{example}

\begin{example}
\label{example.splittable}
Consider the planar graph $G$ on the left, below.
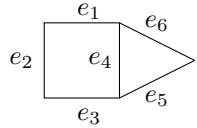
\begin{center}
\begin{tikzpicture}[scale=0.5]
\draw (0,0) rectangle (2,2);
\draw (2,2) to (4,1);
\draw (2,0) to (4,1); 
\node at (-0.6,1) {$e_2$};
\node at (1.2,2.3){$e_1$};
\node at (1.2,-0.5){$e_3$};
\node at (1.5,1){$e_4$};
\node at (3,2){$e_6$};
\node at (3,0){$e_5$};
\node at (0,)[anchor=south east]{};
\node at (0,0){$\bullet$};
\node at (2,2){$\bullet$};
\node at (0,2){$\bullet$};
\node at (2,0){$\bullet$};
\node at (4,1){$\bullet$};
\end{tikzpicture}
\hspace{1in}
\raisebox{0.4in}{$M^{\mathsf{ext}}=
	\begin{pmatrix}
	1&1& 1& 1 & 0 & 0 \\
	a_1&a_2& a_3& a_4 & 0 & 0 \\
	a_1^2&a_2^2& a_3^2& a_4^2 & 0 & 0\\
	0&0&0&-1&-1&-1\\
	0&0&0&-a_4&-a_5&-a_6\\
	0&0&0&-a_4^2&-a_5^2&-a_6^2\\
	\end{pmatrix}$}
\end{center}
This graph has two faces $F_1$ and $F_2$ with face cycles $e_1 e_2 e_3 e_4$ and $e_4 e_5 e_6$, respectively. Directing the edges clockwise
for face $F_1$ and anticlockwise for face $F_2$, the extended cycle basis matrix is on the right, above.  The subgraph $G'$ and $G \backslash G'$ are respectively
\begin{center}
\begin{tikzpicture}[scale=0.5]
\draw (2,0) to (2,2);
\draw (2,2) to (4,1);
\draw (2,0) to (4,1); 
\node at (1.5,1){$e_4$};
\node at (3,2){$e_6$};
\node at (3,0){$e_5$};
\node at (0,)[anchor=south east]{};
\node at (2,0){$\bullet$};
\node at (2,2){$\bullet$};
\node at (4,1){$\bullet$};
\end{tikzpicture}
\hspace{1cm} \raisebox{0.5cm}{and} \hspace{1cm}
\begin{tikzpicture}[scale=0.5]
\draw (2,0) to (2,2);
\draw (2,2) to (4,1);
\draw (2,0) to (4,1); 
\node at (1.5,1){$e_2$};
\node at (3,2){$e_1$};
\node at (3,0){$e_3$};
\node at (0,)[anchor=south east]{};
\node at (2,0){$\bullet$};
\node at (2,2){$\bullet$};
\node at (4,1){$\bullet$};
\end{tikzpicture} \raisebox{0.7cm}{.}
\end{center}
Note that $e_{G'} = 3 f_{G'}$, so that case (2) of Corollary~\ref{corollary.factor} applies. Indeed, we have $\det M^{\mathsf{ext}} = \det N = \det N_{G'}
\det N_{G\backslash G'}$, where
\[
	N_{G'} = \begin{pmatrix}
	-1&-1&-1\\
	-a_4&-a_5&-a_6\\
	-a_4^2&-a_5^2&-a_6^2
	\end{pmatrix}
	\qquad \text{and} \qquad
	N_{G\backslash G'} = \begin{pmatrix}
	1&1&1\\
	a_1&a_2&a_3\\
	a_1^2&a_2^2&a_3^2
	\end{pmatrix}.
\]
Corollary~\ref{corollary.dim decomp} states $\dim \mathsf{Spl}_2(G,\ell; v_0) = \dim \mathsf{Spl}_2(G \backslash G',\ell; v_0)=0$ for generic edge 
labels on $G'$.
\end{example}

We now add the vocabulary of contractibility to give an explicit polynomial that characterizes the locus of edge labels in special position, namely the edge labels for which $M^{\mathsf{ext}}$ has smaller-than-expected rank.

\begin{theorem} \label{theorem: special position formula}
Suppose that $(G_{S_1}, G_{S_2}, \ldots, G_{S_k}), (G_1', G_2', \ldots, G_{k+1}')$ are two sequences of graphs obtained from $G=G_1'$ successively as follows.  Each graph $G_{S_i}$ is obtained from a minimal contractible set of faces $S_i$ in $G_i'$ and the graph $G_{i+1}' = G_i' \backslash G_{S_i}$. Assume that $G_{k+1}'$ has no proper subset of contractible faces.  Further assume that within each $G_i'$ no two minimal contractible subgraphs share an edge.

For each $i = 1, \ldots, k$, the collection of edge labels $\{a_1, a_2, \ldots\}$ that are in special position for $G_{S_i}$ forms an algebraic variety $\mathcal{V}_i$.  The analogous claim holds for $G_{k+1}'$ and variety $\mathcal{V}_{k+1}$ as well.  

The union $\bigcup_{i=1}^{k+1} \mathcal{V}_i$ is the algebraic variety that describes the locus of edge labels $\{a_1, a_2, \ldots \}$ in special position for $G$.  

Each $\mathcal{V}_i$ with $i \leqslant k$ is an algebraic hypersurface given by the vanishing of the following polynomial:
\begin{itemize}
    \item if $e_{G_{S_i}}-3f_{G_{S_i}}=-2$ the polynomial $a_j$ where $e_j$ is the sole edge in $G_{S_i}$
    \item if $e_{G_{S_i}}-3f_{G_{S_i}}=-1$ the polynomial $a_j-a_{j'}$ where $e_j, e_{j'}$ are the two edges in $G_{S_i}$ and
    \item if $e_{G_{S_i}}-3f_{G_{S_i}}=0$ the polynomial 
\[
	\sum_{\substack{\varphi \colon \mathcal{F} \to \mathcal{E}_3\\ \text{edge-injective}}}  
	\prod_{F \in \mathcal{F}} \det N_{F,\varphi(F)},
\]
where the sum is taken over all possible edge-injective functions on $G_{S_i}$.
\end{itemize}
For the graph $G_{k+1}'$ without any proper contractible subsets, let $N$ be any square submatrix of $M^{\mathsf{ext}}_{G_{k+1}'}$ 
that has $3f_{G_{k+1}'}$ rows and define the hypersurface $\mathcal{H}_N$ to be the vanishing of the polynomial 
\[
	\sum_{\substack{\varphi \colon \mathcal{F} \to \mathcal{E}_3\\ \text{edge-injective}}}  
	\prod_{F \in \mathcal{F}} \det N_{F,\varphi(F)},
\]
where the sum is taken over all possible edge-injective functions on the subset of edges corresponding to $N$.  Then the variety 
$\mathcal{V}_{k+1}$ corresponding to $G_{k+1}'$ is the intersection of the hypersurfaces $\cap_{N} \mathcal{H}_N$ over every possible square 
submatrix $N$ in $M^{\mathsf{ext}}_{G_{k+1}'}$.
\end{theorem}

\begin{remark}
    We can restate this theorem using the language of linear algebra. In the hypotheses, the condition that the collection of edge labels 
    $\{a_1, a_2, \ldots\}$ be in special position for $G_{S_i}$  is equivalent to  $M^{\mathsf{ext}}_{G_{S_i}}$ being not full rank.  Since 
    $e_{G_{S_i}} \leqslant 3 f_{G_{S_i}}$ this in turn means $\rank M^{\mathsf{ext}}_{G_{S_i}} < e_{G_{S_i}} $.

    In the conclusion, varying over every possible square submatrix $N$ of $M^{\mathsf{ext}}_{G_{k+1}'}$ is equivalent to making every possible 
    choice of $3f_{G_{k+1}'}$ edges in $G_{k+1}'$.  Either interpretation can be used to calculate the intersection of the hypersurfaces 
    $\cap_{N} \mathcal{H}_N$.
\end{remark}

\begin{proof}
The graphs $(G_{S_1}, G_{S_2}, \ldots, G_{S_k})$ and $(G_1', G_2', \ldots, G_{k+1}')$ describe an iterated version of the process in 
Corollary~\ref{corollary.rank decomp}, which decomposes the rank of the matrix $M^{\mathsf{ext}}$ into a sum of the ranks of $M^{\mathsf{ext}}_{G_i}$.  
So if any of those ranks are independently smaller than expected, so too is the rank of $M^{\mathsf{ext}}$.  Each rank condition characterizes an algebraic 
variety using, e.g., the Pl\"ucker embedding. Combining this with Proposition~\ref{proposition.factoring det into edge-injective} gives the explicit polynomial 
listed here whose zero locus describes the edge labels that are in special position for each $G_{S_i}$.  When the graph is minimal contractible, 
Corollary~\ref{corollary.min contract} classifies the matrices $M^{\mathsf{ext}}_{G_{S_i}}$ as one-column (for a single face that is a loop), two-column 
(for a single face with two edges), or square; we simplify the determinantal calculations in each case. Finally, if $G_{k+1}'$ has no proper contractible 
subset, then it is full rank if the determinantal condition holds {\em{for any}} square submatrix of 
$M^{\mathsf{ext}}_{G_{k+1}'}$.  Hence the edge labels $\{a_1, a_2, \ldots\}$ are in special position if and only if the determinant is zero {\em{for every}} 
possibly square submatrix, namely the edge labels lie in the intersection of the corresponding hypersurfaces.
\end{proof}

\begin{example} \label{example: need extra condition for rank additivity}
Figure~\ref{figure: contractibility requires assumption} shows why the theorem depends on the hypothesis that two minimal contractible subgraphs do not share an edge. On the left is a graph containing two different minimal contractible subgraphs that share an edge.  On the right is its extended cycle basis matrix, where the top triangle has been directed clockwise and the bottom triangle has been directed counterclockwise.  

Suppose that $a_1, a_2, a_3, a_4$ are all distinct and that $a_3=a_5$.  The rank of the matrix overall is $5$.  Suppose $G_1$ is the top triangle and $G\backslash G_1$ is the graph with two edges formed by contracting the top triangle.  Then the extended cycle basis matrix for $G_1$ is the top diagonal $3 \times 3$ block while that for $G \backslash G_1$ is the bottom diagonal $3 \times 2$ block.  Since $a_1, a_2, a_3$ are distinct, these two blocks have rank $3$ and $2$ respectively, as desired.  
\begin{figure}[h]
\begin{tikzpicture}
\draw (0,0) to (2,0);
\draw (1,-1) to (0,0);
\draw (1,-1) to (2,0);
\draw (1,1) to (0,0);
\draw (1,1) to (2,0);

\node at (1,.2){$e_3$};
\node at (.2,-.6){$e_4$};
\node at (1.8,-.6){$e_5$};
\node at (1.8,.6){$e_2$};
\node at (.2,.6){$e_1$};

\node at (0,0){$\bullet$};
\node at (2,0){$\bullet$};
\node at (1,-1){$\bullet$};
\node at (1,1){$\bullet$};
\node at (0,)[anchor=south east]{};
\end{tikzpicture} \hspace{1in} 
\raisebox{.4in}{$\left( \begin{array}{ccccc} 
1 & 1 & 1 & 0 & 0 \\
a_1 & a_2 & a_3 & 0 & 0 \\ 
a_1^2 & a_2^2 & a_3^2 & 0 & 0 \\
0 & 0 & -1 & -1 & -1 \\
0 & 0 & -a_3 & -a_4 & -a_5 \\
0 & 0 & -a_3^2 & -a_4^2 & -a_5^2
\end{array} \right)$}
\caption{A graph and its extended cycle basis matrix discussed in Example~\ref{example: need extra condition for rank additivity}.
This graph demonstrates the assumptions in Theorem~\ref{theorem: special position formula}.
\label{figure: contractibility requires assumption}}
\end{figure}

However, if instead we choose $G_2$ to be the bottom triangle, then the extended cycle basis matrix for $G \backslash G_2$ corresponds to the 
top $3 \times 2$ diagonal block while that for $G_2$ corresponds to the bottom $3 \times 3$ diagonal block.  With our assumptions on $a_i$, these 
ranks are $2$ and $2$ respectively, and $2 + 2 \neq 5$. 
\end{example}

\begin{example}
\label{example.successive contraction}
We analyze Figure~\ref{figure.non-generic} according to Theorem~\ref{theorem: special position formula}.  In this case, there is only one possible choice 
of minimal contractible subset at each step, with $S_1 = \{e_1, e_2\}$ and $S_2 = \{e_7, e_8, e_9\}$.  Figure~\ref{figure: repeated contracting minimal 
contractible subgraphs} has the original graph on the left with subgraph $G_{S_1}$ in red, the graph $G_2'$ in the middle with subgraph $G_{S_2}$ in red, 
and the graph $G_2'$ on the right with no proper contractible subgraph.
\begin{figure}[h]
\begin{center}
\begin{tikzpicture}[xscale=1.5,yscale=.5]
\draw[thick] (0,-2) to (2,-2);
\draw[thick] (0,2) to (2,2);
\draw[thick] (0,0) to (0,2);
\draw[thick] (0,0) to (0,-2);
\draw[thick] (2,0) to (2,-2);
\draw[thick] (2,0) to (2,2);
\draw[red, thick] plot [smooth cycle] coordinates {(0,0) (1,0.5) (2,0) (1,-0.5)};
\node at (1,2.4) {$e_9$};
\node[red] at (1,0.9){$e_1$};
\node[red] at (1,-0.9){$e_2$};
\node at (-0.2,-1){$e_3$};
\node at (-0.2,1){$e_7$};
\node at (2.2,-1){$e_4$};
\node at (2.2,1){$e_8$};
\node at (0.5,-2.4){$e_5$};
\node at (1,-2){$\bullet$};
\node at (1.5,-2.4){$e_6$};
\node at (0,-2){$\bullet$};
\node at (2,-2){$\bullet$};
\node at (0,2){$\bullet$};
\node at (0,-2){$\bullet$};
\node[red] at (0,0){$\bullet$};
\node[red] at (2,0){$\bullet$};
\node at (2,2){$\bullet$};
\node at (0,)[anchor=south east]{};
\end{tikzpicture} \hspace{0.5in}
\begin{tikzpicture}[xscale=1.5,yscale=.5]
\draw[thick] (0,-2) to (2,-2);
\draw[red, thick] (0,2) to (2,2);
\draw[red, thick] (1,0) to (0,2);
\draw[thick] (1,0) to (0,-2);
\draw[thick] (1,0) to (2,-2);
\draw[red, thick] (1,0) to (2,2);

\node[red] at (1,2.2) {$e_9$};
\node at (0.2,-1){$e_3$};
\node[red] at (0.2,1){$e_7$};
\node at (1.8,-1){$e_4$};
\node[red] at (1.8,1){$e_8$};
\node at (0.5,-2.4){$e_5$};
\node at (1,-2){$\bullet$};
\node at (1.5,-2.4){$e_6$};
\node at (0,-2){$\bullet$};
\node at (2,-2){$\bullet$};
\node[red] at (0,2){$\bullet$};
\node at (0,-2){$\bullet$};
\node[red] at (1,0){$\bullet$};
\node[red] at (2,2){$\bullet$};
\node at (0,)[anchor=south east]{};
\end{tikzpicture}
\hspace{0.5in}
\begin{tikzpicture}[xscale=1.5,yscale=.5]
\draw[thick] (0,-2) to (2,-2);
\draw[thick] (1,0) to (0,-2);
\draw[thick] (1,0) to (2,-2);

\node at (0.2,-1){$e_3$};
\node at (1.8,-1){$e_4$};
\node at (0.5,-2.4){$e_5$};
\node at (1,-2){$\bullet$};
\node at (1.5,-2.4){$e_6$};
\node at (0,-2){$\bullet$};
\node at (2,-2){$\bullet$};
\node at (0,-2){$\bullet$};
\node[red] at (1,0){$\bullet$};
\node at (0,)[anchor=south east]{};
\end{tikzpicture}
\end{center}
\caption{Successively contract minimal contractible subgraphs until what remains has no proper contractible subgraph
as discussed in Example~\ref{example.successive contraction}.
\label{figure: repeated contracting minimal contractible subgraphs}}
\end{figure}

Now we use Theorem~\ref{theorem: special position formula} to identify the special locus of edge labels.  The locus of special positions for the 
graph $G_{S_1}$ is the hypersurface consisting of solutions to $a_1-a_2=0$.  The locus of special positions for the graph $G_{S_2}$ is the 
hypersurface consisting of solutions to $(a_7-a_8)(a_7-a_9)(a_8-a_9)=0$.  Like $G_{S_1}$, this is simply the set of edge labels that are not all 
distinct.  However, the locus of special positions for $G_2'$ is the {\em{intersection}} of the hypersurfaces $\mathcal{H}_{i,j,k}$ given by 
$(a_i-a_j)(a_j-a_k)(a_i-a_k)$ for any three distinct $i, j, k \in \{3, 4, 5, 6\}$.  This has several components given by the different ways of 
partitioning $\{a_3, a_4, a_5, a_6\}$ into two nonempty subsets, each of which assumes the same value.  Of those components, the only one that 
can be realized as the dual to a planar triangulation are those with $a_4=a_5$ and $a_3=a_6$.
\end{example}

Finally, we note that Theorem~\ref{theorem: special position formula} applies to all graphs that are dual to plane triangulations.  This is an 
immediate consequence of Corollary~\ref{corollary: special polynomials for triangulations}.

\begin{corollary}
    Suppose that $\Delta$ is a plane triangulation and $G_{\Delta}$ is its dual graph.  The hypotheses of Theorem~\ref{theorem: special position formula} 
    apply to $G_{\Delta}$.  Furthermore, each minimal contractible graph $G_{S_i}$ is dual to a subdivided triangle and satisfies 
    $e_{G_{S_i}} = 3f_{G_{S_i}}$, and each quotient $G'_{i+1}$ is dual to the plane triangulation with the subdivision corresponding to $G_{S_i}$ erased. 
\end{corollary}

%%%%%%%%%%%%%%%%%%%%%%%%%%%%%%%%%%%%%%%%%%%%%%%%%%%%%%%%%%%%%%%
\section{Paths on faces and the set of all edge-injective functions in the minimal contractible case}
\label{section.existence and all edge-inj}

In this section we start showing the true power of minimal contractibility.  We assume as usual that $G$ is a finite, planar graph without leaves, and 
then add the hypothesis that $G$ contains no proper contractible subset of faces. 
We will show that under the hypothesis of minimal contractibility, there exists an edge-injective function $\varphi \colon \mathcal{F} \to \mathcal{E}_3$.  

%%%%%%%%%%%%%%%%%%%%%%%%%%%%%%%%%%%%%%%%%%%%%%%%%%%%%%%%%%%%%%%
\subsection{Existence of edge-injective functions in the minimal contractible case}
\label{section.existence}
To show that there exists an edge-injective function, we will build paths through the faces of a graph that allow us to successively modify a function 
$\varphi$ until it satisfies edge-injectivity.  We start with a definition.

\begin{definition}
Suppose that $F_0F_1F_2 \cdots F_k$ is a sequence of faces in $\mathcal{F}$ for which $F_i \cap F_{i+1}$ is an edge on both $F_i$ and $F_{i+1}$ 
for each $0\leqslant i<k$. Such a sequence is called a \defn{path of faces}.
\end{definition}

Note that if $G$ is planar, then a path of faces corresponds to an actual path in the dual graph.

\begin{theorem} 
\label{theorem: no proper contractible subset means maximal edge-injective function}
If $G$ contains no proper contractible subset of faces, then there is a function $\varphi$ from the set $\mathcal{F}$ of bounded faces of $G$ to the 
set $2^{E}$ of subsets of edges of $G$ so that:
\begin{enumerate}
\item  For each face $F$ the image $\varphi(F)$ contains at most three edges.
\item For each bounded face $F$ the image $\varphi(F)$ consists of edges on the boundary of $F$ with no edge used twice, in the sense that 
$\varphi(F)\cap \varphi(F')=\emptyset$ for any two faces $F$ and $F'$.
\item If there is a face $F$ with $|\varphi(F)|<3$ then every edge of $G$ is contained in $\varphi({F})$.
\end{enumerate}
\end{theorem}

\begin{remark}
\label{remark.restatement of existence}
Note that the conditions in Theorem~\ref{theorem: no proper contractible subset means maximal edge-injective function} can be rewritten as follows: 
Either $\varphi$ is edge-injective or every edge of $G$ is contained in $\varphi(F)$ for some face $F$.
\end{remark}

\begin{proof}
We will show that if $\varphi$ is a function satisfying conditions (1) and (2) but $|\varphi(F)|<3$ then  
\begin{itemize}
\item either $\varphi$ uses all edges of $G$, or
\item $\varphi$ can be modified to create a function $\varphi'$ satisfying conditions (1) and (2), such that the image of $\varphi'$ is strictly contained in
the image of $\varphi$ and $|\varphi'(F)|>|\varphi(F)|$.
\end{itemize}

First suppose $\varphi$ is a function satisfying conditions (1) and (2) and that there is a face $F$ with both $|\varphi(F)| < 3$ and an edge $e$ on the 
boundary of $F$ that is not yet in the image of $\varphi,$ namely 
\[
	e \not \in \bigcup_{F \in \mathcal{F}} \varphi(F).
\] 
Then we can construct another function $\varphi'$ satisfying conditions (1) and (2) whose image is {\it strictly larger} than that of $\varphi$ by the rule that
\[
	\varphi'(F') = \begin{cases}
	\varphi(F') & \textup{ if } F' \neq F,\\
	\varphi(F) \cup \{e\} & \textup{ if } F' = F. 
	\end{cases}
\]

We now generalize this argument to the case when the unused edge is not on the face without three edges in its image. 
Consider a path of faces $F_0 F_1 \ldots F_k$ such that the edge $F_i \cap F_{i+1}$ is in the image $\varphi(F_i)$ for each $0\leqslant i<k$.
Let us call this a \defn{$\varphi$-compatible path of faces}.

Assume that there is at least one edge unused by $\varphi$ meaning the image of $\varphi$ does not contain all edges of $G$.  Suppose there 
is a face $F$ for which $|\varphi(F)|<3$. Define the subset $S_F \subseteq \mathcal{F}$ containing $F$ together with all faces $F'$ that have a 
$\varphi$-compatible path of faces from $F'$ to $F$.  By definition $S_F$ contains $F$.  
Consider the graph $G_{S_F}$ associated to this subset of faces. By construction, it is connected.

\smallskip
\noindent \textbf{Claim:} If $e'$ is an edge on the boundary of a face in $S_F$ as well as a face in $S_F^c$, then $e'  \in \varphi(S_F)$.  

\smallskip
\noindent
Indeed, if there is a face $F' \not \in S_F$ with $e' \in \varphi(F')$ and $e' = F' \cap F''$ for some face $F'' \in S_F$, then the $\varphi$-compatible
path from $F''$ to $F$ can be extended to a $\varphi$-compatible path $F'  F'' \ldots F$ using the edge $e'$ as the first step 
in the path. Put differently, we have
\[
	\varphi(S_F) = \varphi(\mathcal{F}) \cap G_{S_F}.
\]

Next we show by contradiction that there must be an edge of $G_{S_F}$ that is {\it not} in the image of $\varphi$. Suppose by contradiction that all 
edges of $G_{S_F}$ are in the image of $\varphi$.  We assumed that not all edges of $G$ are in the image of $\varphi$ (else $\varphi$ is the function 
claimed by the theorem).  Our first conclusion is that $S_F$ is a proper subset of $\mathcal{F}$ since $S_F$ contains $F$ but 
by the above claim no face with $e$ on its boundary is in $S_F$ so $S_F^c$ contains the face(s) bounded by $e$.

If $G_{S_F}$ has no unused edges then we further conclude that $\varphi|_{G_{S_F}}$ satisfies conditions (1), (2), and (3).  However we also assumed 
$|\varphi(F)|<3$ for some face.  This means the image of $\varphi$ contains all edges of $G_{S_F}$ but does not contain 3 edges for each face, 
namely $e_{G_{S_F}}-3f_{G_{S_F}}<0$.  Thus $G_{S_F}$ is contractible, contradicting the assumption that $G$ is has no proper contractible subset of 
faces.

It follows that if $\varphi(\mathcal{F})$ does not contain all edges of $G$ and also if there exists a face $F$ with $|\varphi(F)|<3$, then $G_{S_F}$ 
has an edge $e$ that is unused by $\varphi$ meaning $e$ is not in the image of $\varphi$ or equivalently not in the image of $\varphi|_{G_{S_F}}$.  
Let $F_0$ be a face in $S_F$ that $e$ bounds and consider any $\varphi$-compatible path of faces $F_0F_1 \ldots F_{k-1}F_k$ from $F_0$ to $F_k=F$.  
At least one such path exists by definition of $S_F$.  

From the $\varphi$-compatible path of faces $F_0 F_1 \ldots F_k$ create a function $\varphi'$ on the faces of $G$ whose image is 
{\it strictly larger} than the image of $\varphi$ by the rule that
\[
	\varphi'(F') = \begin{cases}
	\varphi(F') & \textup{ if } F' \textup{ is not on the sequence } F_0F_1 \ldots F_k,\\
	\varphi(F_0) \cup \{e\} - \{F_0 \cap F_{1}\} & \textup{ if } F' = F_0,\\
	\varphi(F_i) \cup \{F_{i-1} \cap F_i\} - \{F_i \cap F_{i+1}\} & \textup{ if } F' = F_i \textup{ for } 0<i<k,\\
	\varphi(F_k) \cup \{F_{k-1} \cap F_k\}  & \textup{ if } F' = F_k = F.
	\end{cases}
\]
Note that $\varphi'$ still satisfies conditions (1) and (2) and has one more edge in its image. Moreover $|\varphi'(F)| > |\varphi(F)|$.  

If $\varphi'$ either uses all edges or $|\varphi'(\mathcal{F})|=3f_G$, then we have proven the claim. Otherwise we can repeat the previous argument 
using $\varphi'$ in place of $\varphi$. Thus if $G$ is a graph with no contractible subset of faces we can continue to expand the function $\varphi$ 
until either it is edge-injective or its image contains all edges in $G$.  This proves the claim.
\end{proof}

\begin{remark}
Note that for the proof of Theorem~\ref{theorem: no proper contractible subset means maximal edge-injective function} it was not essential
that edge-injective functions select three edges for each face. The results of this section go through with assuming that 
$|\varphi(F)|<m$ for any positive integer $m$ if accordingly a contractible subset $S$ of faces is defined as 
$e_S \leqslant m f_S$.  This could be useful when generalizing to non-generic edge-labelings.
\end{remark}

\begin{corollary}
\label{cor.existence}
If $G$ is a finite, planar graph without leaves that contains no proper contractible subset of faces, it has an edge-injective function whenever 
$e_G\geqslant 3f_G$.
\end{corollary}

\begin{proof}
By Theorem~\ref{theorem: no proper contractible subset means maximal edge-injective function} and Remark~\ref{remark.restatement of existence},
there exists an edge-injective function if $e_G\geqslant 3f_G$.
\end{proof}

%%%%%%%%%%%%%%%%%%%%%%%%%%%%%%%%%%%%%%%%%%%%%%%%%%%%%%%%%%%%%%%
\subsection{The set of all edge-injective functions in the minimal contractible case}
\label{section.all edge injective}

In the previous section, we proved that under the hypothesis of minimal contractibility, there exists \textit{an} edge-injective function.  Now, using the 
same hypothesis, we describe \textit{all} edge-injective functions on the graph.  Our key tool is a modified dual graph $G^\star$ to $G$ that tracks 
edges shared between faces but not boundary edges. This graph $G^\star$ allows us to keep track of paths on faces and deepen the 
analysis in Section~\ref{section.existence}. Our key result is Lemma~\ref{lemma.directed cycles}, which says essentially that the edge-injective functions 
on $G$ correspond to the choice of certain cycles in $G^\star$.  Theorem~\ref{connected} refines this to show that any edge-injective function on $G$ 
can be obtained from any other edge-injective function on $G$ by swapping certain edges in a succession of cycles in $G^\star$.

The case when the graph $G_{\Delta}$ is dual to a planar triangulation $\Delta$ is of particular interest.
Remark~\ref{remark: modified dual when $G$ comes from triangulation} describes this situation.

Fix a planar embedding of $G$ with the same conditions as in Section~\ref{section.existence}, so  we know there exists an edge-injective function 
$\varphi \colon \mathcal{F} \to \mathcal{E}_3$ if $e_G \geqslant 3f_G$. Without loss of 
generality we may assume that $e_G = 3f_G$; otherwise we may restrict to the edges in $G$ in the image of $\varphi$.

We define $G^\star$ to be the modification of the usual dual graph of $G$ for which:
\begin{itemize}
    \item every bounded face $F\in \mathcal{F}$ of $G$ is a vertex in $G^\star$,
    \item there is an edge between $F$ and $F'$ in $G^\star$ if $F$ and $F'$ share an edge in $G$,
    \item every edge belonging to only one bounded face $F$ in $G$ is a loop in $G^\star$.
\end{itemize} 

\begin{remark}
\label{remark: modified dual when $G$ comes from triangulation}
    If the graph $G$ is in fact the dual $G_{\Delta}$ to a planar triangulation $\Delta$ then $G_{\Delta}{\star}$ consists of the subgraph of $\Delta$ 
    induced by the interior vertices, together with loops at each vertex whenever $\Delta$ has an edge to the boundary.
\end{remark}
Let $\Phi$ be the set of edge-injective functions $\varphi$ on $G$ with image $\mathcal{E}_3$.
In $G^\star$, an edge-injective function $\varphi \in \Phi$ can be viewed as assigning a direction to each edge in $G^\star$: if an edge in $G$ is associated 
to face $F \in \mathcal{F}$ under $\varphi$, then the edge in 
$G^\star$ is directed towards $F$. Note that under this correspondence, every vertex $F$ in $G^\star$ has exactly three arrows pointing towards $F$.
From now on, we interchangeably view elements in $\Phi$ as edge-injective functions on $G$ or as an assignment of directed arrows on $G^\star$ such
that every vertex in $G^\star$ has three incoming edges. We denote an edge-injective function $\varphi$ on $G$ in the dual setting as $G^\star_\varphi$.
\begin{example}
\label{example.phi dual}
Consider the planar graph $G$ on the left  in Figure~\ref{figure: edge inj example and dual graph}, with edges labeled $e_1,\ldots,e_{15}$.
\begin{figure}
\begin{center}
\usetikzlibrary {arrows.meta,bending,positioning} 
\scalebox{0.75}{\begin{tikzpicture}
\draw[thick] (0,2) to (0,0);
\node at (0,0){$\bullet$};
\node at (0,2){$\bullet$};
\draw[thick] (0,0) to (1.5,-1.32);
\node at (1.5,-1.32){$\bullet$};
\draw[thick] (0,0) to (-1.5,-1.32);
\node at (-1.5,-1.32){$\bullet$};
\draw[thick] (-1.5,0.68) to (-1.5,-1.32);
\node at (-1.5,0.68){$\bullet$};
\draw[thick] (1.5,0.68) to (1.5,-1.32);
\node at (1.5,0.68){$\bullet$};
\draw[thick] (0,2) to (1.5,0.68);
\draw[thick] (0,2) to (-1.5,0.68);
\draw[thick] (0,-2.64) to (1.5,-1.32);
\node at (0,-2.64){$\bullet$};
\node at (0.75,-2){$\bullet$};
\draw[thick] (0,-2.64) to (-1.5,-1.32);
\node at (-0.85,1.7) {\color{darkred}{\Large $e_1$}};
\node at (-1.8,-0.1){\color{darkred}{\Large $e_2$}};
\node at (-0.3,0.7){\color{blue}{\Large $e_4$}};
\node at (-0.6,-0.85){\color{darkred}{\Large $e_3$}};
\node at (-1,-2.1){\color{darkgreen}{\Large $e_5$}};
\node at (1.3,-1.9){\color{orange}{\Large $e_7$}};
\node at (1,-0.5){\color{darkgreen}{\Large $e_8$}};
\node at (0.8,1.7){\color{blue}{\Large $e_9$}};
\node at (1.8,-0.3){\color{blue}{\Large $e_{10}$}};
\node at (0.6,-2.5){\color{darkgreen}{\Large $e_6$}};
\draw[thick] (0,-2.64) to (1.5,-3.96);
\node at (1.5,-3.96){$\bullet$};
\draw[thick] (1.5,-1.32) to (3,-2.64);
\node at (3,-2.64){$\bullet$};
\draw[thick] (3,-2.64) to (1.5,-3.96);
\draw[thick] (1.5,0.68) to (3,-0.64);
\node at (3,-0.64){$\bullet$};
\draw[thick] (3,-0.64) to (3,-2.64);
\node at (2.5,0.3){\Large $e_{11}$};
\node at (3.3,-1.5){\Large $e_{12}$};
\node at (2.5,-1.8){\Large $e_{13}$};
\node at (2.6,-3.5){\color{orange}{\Large $e_{14}$}};
\node at (0.5,-3.5){\color{orange}{\Large $e_{15}$}};
\node at (0,)[anchor=south east]{};
\end{tikzpicture}
}
\hspace{1in}
\raisebox{0.2in}{\begin{tikzpicture}[auto]
\node (F3) at (1,1){$F_3$};
\node (F1) at (-1,1){$F_1$};
\node (F2) at (0,0){$F_2$};
\node (F4) at (2,0){$F_4$};
\node (F5) at (1,-1){$F_5$};
\node (alpha) at (0,0.7){$\alpha$};
\node (beta) at (1,-0){$\beta$};
\node (gamma) at (0.45,-0.45){$\gamma$};
\path (F1) edge [thick,loop left] (F1);
\path (F1) edge [thick,loop above] (F1);
\path (F2) edge [thick,loop left] (F2);
\path (F3) edge [thick,loop above] (F3);
\path (F4) edge [thick,loop above] (F4);
\path (F4) edge [thick,loop below] (F4);
\path (F5) edge [thick,loop below] (F5);
\path (F5) edge [thick,loop right] (F5);
\draw[edge,thick] (F1) -- (F3);
\draw[edge,thick] (F2) -- (F1);
\draw[edge,thick] (F3) -- (F2);
\draw[edge,thick] (F4) -- (F3);
\draw[edge,thick] (F5) -- (F4);
\draw [edge, thick] (F2) edge [bend left=30] (F5);
\draw [edge, thick] (F5) edge [bend left=30] (F2);
\node at (0,)[anchor=south east]{};
\end{tikzpicture}
}
\end{center}
\caption{Edge-injective function $\varphi$ on $G$ and $G^\star_\varphi$ on the dual graph $G^\star$} \label{figure: edge inj example and dual graph}
\end{figure}
The colors of the edges in the graph on the left indicate an edge-injective function $\varphi$, where the red edges belong to the face 
$\color{darkred}{F_1}$ given by $e_1 e_2 e_3 e_4$, 
the green edges belong to the face $\color{darkgreen}{F_2}$ given by $e_3 e_5 e_6 e_7 e_8$, the blue edges belong to face $\color{blue}{F_3}$
given by $e_4 e_8 e_{10} e_9$, the black edges belong to the face $F_4$ given by $e_{10} e_{11} e_{12} e_{13}$, and finally
the orange edges belong to the face $\color{orange}{F_5}$ given by $e_7 e_6 e_{15} e_{14} e_{13}$.  Under the correspondence given above, this 
edge-injective function in the dual graph $G^\star$ is given by $G^\star_\varphi$ shown in the right of Figure~\ref{figure: edge inj example and dual graph}.
\end{example}

\begin{definition}\label{definition: edge inj on dual graph}
Define a graph $G_\Phi$ with vertex set $\Phi$ and an edge between $G^\star_\varphi$ and $G^\star_{\varphi'}$ in $\Phi$ if 
$G^\star_\varphi$ and $G^\star_{\varphi'}$ only differ by reversing the arrows on a directed face cycle in $G^\star$. 
\end{definition}

Similar ideas of reversing cycles have appeared in~\cite{WW.1987} and the unpublished~\cite{Wh.unpublished}.

\begin{example}
The graph $G_{\Phi}$ for the graph in Example~\ref{example.phi dual} is given in Figure~\ref{figure.GPhi}.
\end{example}

\begin{figure}[h]
\begin{center}
\scalebox{0.9}{
\begin{tikzpicture}[auto]
\node (D1) at (0,-5){\scalebox{0.8}{
\begin{tikzpicture}[auto]
\node (F3) at (1,1){$F_3$};
\node (F1) at (-1,1){$F_1$};
\node (F2) at (0,0){$F_2$};
\node (F4) at (2,0){$F_4$};
\node (F5) at (1,-1){$F_5$};
\node (alpha) at (0,0.7){\textcolor{red}{$\alpha$}};
\node (beta) at (1,-0){\textcolor{blue}{$\beta$}};
\node (gamma) at (0.45,-0.45){\textcolor{darkgreen}{$\gamma$}};
\path (F1) edge [thick,loop left] (F1);
\path (F1) edge [thick,loop above] (F1);
\path (F2) edge [thick,loop left] (F2);
\path (F3) edge [thick,loop above] (F3);
\path (F4) edge [thick,loop above] (F4);
\path (F4) edge [thick,loop below] (F4);
\path (F5) edge [thick,loop below] (F5);
\path (F5) edge [thick,loop right] (F5);
\draw[edge,thick] (F1) -- (F3);
\draw[edge,thick] (F2) -- (F1);
\draw[edge,thick] (F3) -- (F2);
\draw[edge,dotted] (F3) -- (F2);
\draw[edge,thick] (F4) -- (F3);
\draw[edge,thick] (F5) -- (F4);
\draw [edge, thick] (F2) edge [bend left=30] (F5);
\draw [edge, thick] (F2) edge [bend left=30] (F5);
\draw [edge, thick] (F5) edge [bend left=30] (F2);
\node at (0,)[anchor=south east]{};
\end{tikzpicture}}
};
\node (D2) at (6,-5){\scalebox{0.8}{
\begin{tikzpicture}[auto]
\node (F3) at (1,1){$F_3$};
\node (F1) at (-1,1){$F_1$};
\node (F2) at (0,0){$F_2$};
\node (F4) at (2,0){$F_4$};
\node (F5) at (1,-1){$F_5$};
\node (alpha) at (0,0.7){\textcolor{red}{$\alpha$}};
\node (beta) at (1,-0){\textcolor{blue}{$\beta$}};
\node (gamma) at (0.45,-0.45){\textcolor{darkgreen}{$\gamma$}};
\path (F1) edge [thick,loop left] (F1);
\path (F1) edge [thick,loop above] (F1);
\path (F2) edge [thick,loop left] (F2);
\path (F3) edge [thick,loop above] (F3);
\path (F4) edge [thick,loop above] (F4);
\path (F4) edge [thick,loop below] (F4);
\path (F5) edge [thick,loop below] (F5);
\path (F5) edge [thick,loop right] (F5);
\draw[edge,thick] (F1) -- (F3);
\draw[edge,thick] (F2) -- (F1);
\draw[edge,thick] (F3) -- (F2);
\draw[edge,thick] (F4) -- (F3);
\draw[edge,thick] (F5) -- (F4);
\draw [edge, thick] (F5) edge [bend right=30] (F2);
\draw [edge, thick] (F2) edge [bend right=30] (F5);
\node at (0,)[anchor=south east]{};
\end{tikzpicture}}
};
\node (D3) at (0,0){\scalebox{0.8}{
\begin{tikzpicture}[auto]
\node (F3) at (1,1){$F_3$};
\node (F1) at (-1,1){$F_1$};
\node (F2) at (0,0){$F_2$};
\node (F4) at (2,0){$F_4$};
\node (F5) at (1,-1){$F_5$};
\node (alpha) at (0,0.7){\textcolor{red}{$\alpha$}};
\node (beta) at (1,-0){\textcolor{blue}{$\beta$}};
\node (gamma) at (0.45,-0.45){\textcolor{darkgreen}{$\gamma$}};
\path (F1) edge [thick,loop left] (F1);
\path (F1) edge [thick,loop above] (F1);
\path (F2) edge [thick,loop left] (F2);
\path (F3) edge [thick,loop above] (F3);
\path (F4) edge [thick,loop above] (F4);
\path (F4) edge [thick,loop below] (F4);
\path (F5) edge [thick,loop below] (F5);
\path (F5) edge [thick,loop right] (F5);
\draw[edge,thick] (F3) -- (F1);
\draw[edge,thick] (F1) -- (F2);
\draw[edge,thick] (F2) -- (F3);
\draw[edge,thick] (F4) -- (F3);
\draw[edge,thick] (F5) -- (F4);
\draw [edge, thick] (F2) edge [bend left=30] (F5);
\draw [edge, thick] (F5) edge [bend left=30] (F2);
\node at (0,)[anchor=south east]{};
\end{tikzpicture}}
};
\node (D4) at (6,0){\scalebox{0.8}{
\begin{tikzpicture}[auto]
\node (F3) at (1,1){$F_3$};
\node (F1) at (-1,1){$F_1$};
\node (F2) at (0,0){$F_2$};
\node (F4) at (2,0){$F_4$};
\node (F5) at (1,-1){$F_5$};
\node (alpha) at (0,0.7){\textcolor{red}{$\alpha$}};
\node (beta) at (1,-0){\textcolor{blue}{$\beta$}};
\node (gamma) at (0.45,-0.45){\textcolor{darkgreen}{$\gamma$}};
\path (F1) edge [thick,loop left] (F1);
\path (F1) edge [thick,loop above] (F1);
\path (F2) edge [thick,loop left] (F2);
\path (F3) edge [thick,loop above] (F3);
\path (F4) edge [thick,loop above] (F4);
\path (F4) edge [thick,loop below] (F4);
\path (F5) edge [thick,loop below] (F5);
\path (F5) edge [thick,loop right] (F5);
\draw[edge,thick] (F3) -- (F1);
\draw[edge,thick] (F1) -- (F2);
\draw[edge,thick] (F2) -- (F3);
\draw[edge,thick] (F4) -- (F3);
\draw[edge,thick] (F5) -- (F4);
\draw [edge, thick] (F5) edge [bend right=30] (F2);
\draw [edge, thick] (F2) edge [bend right=30] (F5);
\node at (0,)[anchor=south east]{};
\end{tikzpicture}}
};
\node (D5) at (-6,-5){\scalebox{0.8}{
\begin{tikzpicture}[auto]
\node (F3) at (1,1){$F_3$};
\node (F1) at (-1,1){$F_1$};
\node (F2) at (0,0){$F_2$};
\node (F4) at (2,0){$F_4$};
\node (F5) at (1,-1){$F_5$};
\node (alpha) at (0,0.7){\textcolor{red}{$\alpha$}};
\node (beta) at (1,-0){\textcolor{blue}{$\beta$}};
\node (gamma) at (0.45,-0.45){\textcolor{darkgreen}{$\gamma$}};
\path (F1) edge [thick,loop left] (F1);
\path (F1) edge [thick,loop above] (F1);
\path (F2) edge [thick,loop left] (F2);
\path (F3) edge [thick,loop above] (F3);
\path (F4) edge [thick,loop above] (F4);
\path (F4) edge [thick,loop below] (F4);
\path (F5) edge [thick,loop below] (F5);
\path (F5) edge [thick,loop right] (F5);
\draw[edge,thick] (F1) -- (F3);
\draw[edge,thick] (F2) -- (F1);
\draw[edge,thick] (F2) -- (F3);
\draw[edge,thick] (F3) -- (F4);
\draw[edge,thick] (F4) -- (F5);
\draw [edge, thick] (F5) edge [bend right=30] (F2);
\draw [edge, thick] (F5) edge [bend left=30] (F2);
\node at (0,)[anchor=south east]{};
\end{tikzpicture}}
};
\draw[darkgreen] (D1) to (D2);
\draw[blue] (D1) to (D5);
\draw[red] (D1) to (D3);
\draw[darkgreen] (D3) to (D4);
\draw[red] (D4) to (D2);
\end{tikzpicture}}
\end{center}
\caption{The graph $G_{\Phi}$ for the graph $G$ in Example~\ref{example.phi dual}.
\label{figure.GPhi}}
\end{figure}

We now show that $G_\Phi$ is a connected graph. In other words, we can construct all edge-injective functions on $G$ from a given 
$\varphi \in \Phi$ by successively reversing directed face cycles in $G^\star_\varphi$.

\begin{lemma}
\label{lemma.directed cycles}
Let $\varphi, \varphi' \in \Phi$. Then $G^\star_\varphi$ is obtained from $G^\star_{\varphi'}$ by successively reversing directed cycles.
\end{lemma}

\begin{proof}
If $\varphi=\varphi'$, we are done. Else, pick a vertex $F\in \mathcal{F}$ in $G^\star$ for which the incoming arrows for $G^\star_{\varphi}$ and 
$G^\star_{\varphi'}$ do not agree. Pick an edge $e_1$ between $F$ and some other vertex $F_1$ in $G^\star$ which is incoming 
for $G^\star_\varphi$ but outgoing for $G^\star_{\varphi'}$. Since every vertex in 
$G^\star$ has exactly three incoming edges for each edge-injective function, vertex $F_1$ must have another edge $e_2$ between $F_1$ and another
vertex $F_2$ which is incoming for $G^\star_{\varphi}$ but outgoing for $G^\star_{\varphi'}$. Repeat this process to obtain a sequence 
$F e_1 F_1 e_2 F_2 \cdots$. Since the graph $G^\star$ is finite, there must be some $j<k$ such that $F_j=F_k$. Then 
$F_j e_{j+1} F_j e_{j+2} \cdots e_k F_k$ is a directed cycle in both $G^\star_{\varphi}$ and $G^\star_{\varphi'}$. Reversing the cycle in $G^\star_{\varphi'}$ 
gives an edge-injective function $G^\star_{\varphi''}$ with fewer arrows different from $G^\star_{\varphi}$ than $G^\star_{\varphi'}$. By induction, repeat 
until the two edge-injective functions agree.
\end{proof}

\begin{theorem}
\label{connected}
The graph $G_\Phi$ is connected.
\end{theorem}

\begin{proof}
By Lemma~\ref{lemma.directed cycles}, any two vertices $G^\star_{\varphi}$ and $G^\star_{\varphi'}$ in $G_\Phi$ are connected by a sequence of 
reversing directed cycles. If we can show that each directed cycle $C$ can be reversed by successively reversing face cycles, we are done.

We induct on the number of enclosed faces by the cycle $C$. If the cycle encloses only one face, it is a face cycle and the claim holds. Otherwise,
the cycle $C$ must contain more than one face and hence must contain at least one edge in $G^\star$. Suppose that there is a vertex $F_o$ on the cycle 
$C$ which is incident to an edge $e_o$ directed from $F_o$ to the interior of $C$. We claim that if there is a directed path $F_o F_1F_2\cdots F_i$ in
$G^\star_\varphi$ with all edges interior to $C$ from $F_o$ to another vertex $F_i$ on $C$, then there is a directed cycle $C'$ enclosing fewer faces 
than $C$.  Indeed, the directed cycle $C$ gives us a directed path $e_1 \cdots e_k$ from $F_i$ to $F_o$ on $C$ going either clockwise or counterclockwise.
Then $F_o F_1F_2\cdots F_i e_1 \cdots e_k$ is a directed cycle $C''$ in $G^\star_\varphi$ that is properly contained in $C$ and hence contains fewer faces.

Denote by $\mathcal{O}$ the set of vertices $F$ in $G^\star_\varphi$ inside or on $C$ that have a directed path of the form $e_oe_1'e_2'\cdots e_k'$ 
from $F_o$ to $F$ with all edges $e_i'$ interior to $C$.

Suppose $\mathcal{O} \cap C = \emptyset$. Then consider the subgraph $G_{\mathcal{O}}$ of $G$ that is the union of all faces in $\mathcal{O}$. 
It is nonempty since there is an edge $e_o$ from $F_o$ to the interior of $C$. Furthermore all of the boundary edges of $G_{\mathcal{O}}$ are associated
with faces in $G_{\mathcal{O}}$ under the edge-injective function $\varphi$ by definition of $\mathcal{O}$. This means that $G_{\mathcal{O}}$ is contractible 
because the number of edges in $G_{\mathcal{O}}$ is three times the number of its faces. This contradicts our assumption of minimum contractible.

Thus $\mathcal{O} \cap C$ is nonempty.  Hence we have at least one directed path from $F_o$ to a vertex $F_i$ on $C$ (possibly equal to $F_o$). 
Thus we obtain a directed cycle interior to $C$ which must contain fewer faces than $C$. Hence induction applies.
\end{proof} 

In the following example, we demonstrate the proof of Theorem~\ref{connected}.

\begin{example}
\label{example.Th4.27} 
Let $G$ be the graph as shown on the left in Figure~\ref{figure1.GPhi}. An edge-injective function $\varphi$ in form of a 
directed graph $G^\star_{\varphi}$ on $G^\star$ is shown on the right of Figure~\ref{figure1.GPhi}. 

Let $C$ be the directed cycle in $G^\star_{\varphi}$ containing faces $\alpha,\beta,\gamma$ and $\delta$ and choose $F_o=F_2$ since there is an edge directed 
from $F_2$ to the interior of $C$. In this case, the set $\mathcal{O}$ consists of the vertices $F_1$ and $F_3$, and $\mathcal{O} \cap C=\{F_3\}$. 
Thus we have a directed cycle $F_2,F_1,F_3,F_4,F_2$ which contains only two faces $\alpha$ and $\delta$. 

If we repeat this process using the directed cycle $C'$ given by $F_2, F_1, F_3, F_4, F_2$ we see that there is only one edge directed to the interior of $C'$. 
Thus $F_o=F_4$.  In this case, the set $\mathcal{O}$ consists of the vertex $F_1$ and so $\mathcal{O} \cap C' = \{F_1\}$.  We obtain the directed cycle $F_4, 
F_1, F_3, F_4$ which contains only the face $\alpha$.  This shows that the vertex in the graph $G_{\Phi}$ that corresponds to $G^\star_\varphi$ is incident 
to an edge corresponding to reversing the boundary of the face $\alpha$. 

\begin{figure}[h]
\begin{tikzpicture}[auto]
\node (G) at (-6,-5){\scalebox{0.65}{
\begin{tikzpicture}[auto]
\draw[thick] (0,0) rectangle (5,5);
\draw[thick] (1.5,1.5) rectangle (3.5,3.5);
\draw[thick] (0,5) to (1.5,3.5);
\draw[thick] (5,5) to (3.5,3.5);
\draw[thick] (0,0) to (1.5,1.5);
\draw[thick] (5,0) to (3.5,1.5);
\node at (2.5,2.5){$F_1$};
\node at (0.8,2.5){$F_4$};
\node at (2.5,0.8){$F_2$};
\node at (4.2,2.5){$F_5$};
\node at (2.5,4.2){$F_3$};
\node at (0,0){$\bullet$};
\node at (0,2.5){$\bullet$};
\node at (1.5,1.5){$\bullet$};
\node at (2.5,0){$\bullet$};
\node at (5,0){$\bullet$};
\node at (0,5){$\bullet$};
\node at (5,2.5){$\bullet$};
\node at (5,5){$\bullet$};
\node at (3.5,3.5){$\bullet$};
\node at (1.5,3.5){$\bullet$};
\node at (3.5,1.5){$\bullet$};
\node at (0,)[anchor=south east]{};
\end{tikzpicture}}
};
\node (D3) at (0,-5){\scalebox{0.8}{
\begin{tikzpicture}[auto]
\node (F1) at (0,0){$F_1$};
\node (F2) at (0,-2){$F_2$};
\node (F3) at (0,2){$F_3$};
\node (F4) at (-2,0){$F_4$};
\node (F5) at (2,0){$F_5$};
\node (alpha) at (-0.7,0.7){\textcolor{red}{$\alpha$}};
\node (beta) at (0.7,0.7){\textcolor{blue}{$\beta$}};
\node (gamma) at (0.7,-0.7){\textcolor{darkgreen}{$\gamma$}};
\node (delta) at (-0.7,-0.7){\textcolor{orange}{$\delta$}};
\path (F2) edge [thick,loop left] (F2);
\path (F2) edge [thick,loop right] (F2);
\path (F3) edge [thick,loop above] (F3);
\path (F4) edge [thick,loop above] (F4);
\path (F4) edge [thick,loop below] (F4);
\path (F5) edge [thick,loop below] (F5);
\path (F5) edge [thick,loop above] (F5);
\draw[edge,thick] (F2) -- (F1);
\draw[edge,thick] (F1) -- (F3);
\draw[edge,thick] (F4) -- (F1);
\draw[edge,thick] (F3) -- (F4);
\draw[edge,thick] (F5) -- (F3);
\draw[edge,thick] (F5) -- (F1);
\draw[edge,thick] (F4) -- (F2);
\draw[edge,thick] (F2) -- (F5);
\node at (0,)[anchor=south east]{};
\end{tikzpicture}}
};
\end{tikzpicture}
\caption{Graph $G$ (left) and an edge-injective function $G^\star_{\varphi}$ on the dual graph $G^\star$ (right).}
\label{figure1.GPhi}
\end{figure}
\end{example}

Let $\alpha_1,\ldots,\alpha_s$ be faces in $G^\star$. We denote by $C_{\alpha_1+\cdots+\alpha_s}$ the cycle containing the faces
$\alpha_1,\ldots,\alpha_s$. Note that $C_{\alpha_1+\cdots+\alpha_s}$ could be a union of disjoint cycles if some of the faces do not share
edges. With this definition, $G^\star_\varphi$ and $G^\star_{\varphi'}$ for any two edge-injective functions $\varphi, \varphi' \in \Phi$ with
$\varphi\neq \varphi'$ can be obtained from each other by reversing a cycle $C$ by Lemma~\ref{lemma.directed cycles}. In particular,
$C=C_{\alpha_1+\cdots+\alpha_s}$, where $\alpha_1,\ldots,\alpha_s$ are the faces enclosed by $C$.

\begin{example}
In Example \ref{example.phi dual}, there are four directed cycles in $G^\star_\varphi$, namely $C_{\alpha}$, $C_{\beta}$, $C_{\gamma}$, 
$C_{\alpha+\gamma}$. Note that we view $C_{\alpha+\gamma}$ as a cycle even though it is the disjoint union of $C_\alpha$ and $C_\gamma$.
\end{example}

This leads immediately to a beautiful enumerative result: the total number of edge-injective functions $|\Phi|$ is one more than the number of directed cycles in an edge-injective function (including cycles that do not bound faces).  Since $|\Phi|$ is determined by the graph, we conclude that  $G^\star_\varphi$ has the same number of directed cycles for all edge-injective functions, as follows.

\begin{lemma}
\label{lemma.counting}
Suppose $\Phi \neq \emptyset$. Then $|\Phi| = m+1$, where $m$ is the number of all directed cycles in $G^\star_\varphi$ for a given $\varphi \in \Phi$.
In particular, the number of directed cycles in $G^\star_\varphi$ is independent of the edge-injective function $\varphi$.
\end{lemma}

\begin{proof}
As explained above, $G^\star_\varphi$ and $G^\star_{\varphi'}$ for any two edge-injective functions $\varphi,\varphi'\in \Phi$ with $\varphi\neq \varphi'$
can be obtained from each other by reversing a cycle $C_{\alpha_1+\cdots+\alpha_s}$ for some faces $\alpha_1,\ldots,\alpha_s$ in $G^\star$
by Lemma~\ref{lemma.directed cycles}. Conversely, any directed cycle in $G^\star_\varphi$ gives rise to another edge-injective function
$\varphi'$ by reversing it. This proves the claim. 
\end{proof}

%%%%%%%%%%%%%%%%%%%%%%%%%%%%%%%%%%%%%%%%%%%%%%%%%%%%%%%%%%%
\section{The main result} \label{section.main result}

We now prove our main result: that a graph $G$ has a generic edge labeling if it admits at least one edge-injective function (and satisfies some 
standard hypotheses, outlined below).  The next subsection contains a technical graph-theoretic discussion needed for one key point.  We then prove 
our main result in several key steps: produce a three-coloring of graph edges with certain technical conditions; then use this three-coloring to show that 
if a graph has an edge-injective function, the graph also has a generic edge labeling; finally, we compute the rank of $M^\mathsf{ext}$.

%%%%%%%%%%%%%%%%%%%%%%%%%%%%%%%%%%%%%%%%%%%%%%%%%%%%%%%%%%%
\subsection{Combinatorial preliminaries: subgraphs enclosed by cycles}
\label{section.counting}

This subsection consists of graph-theoretic results that we need for Section~\ref{section.existence generic}, where we show that generic labelings exist.  
Here, we analyze some properties of subgraphs enclosed by a cycle. 
We assume in this subsection that $G$ is a finite, planar graph without leaves and without holes. Furthermore, we assume that an edge-injective 
function exists (e.g. but not limited to the minimal contractible case).

Fix a planar embedding of $G$. We will work in the dual setting, namely we analyze a planar embedding of $G^\star$ as defined in 
Section~\ref{section.all edge injective}. We begin with some definitions. 

\begin{itemize}
    \item Let $B$ be a non-self-intersecting $n$-cycle in the fixed planar embedding of $G^\star$.  
    \item Since $B$ is a non-self-intersecting cycle, it defines an interior in the plane by the Jordan curve theorem. Let $G_i^\star$ be the subgraph of $G^\star$ interior to $B$, namely consisting 
of all vertices and edges in the interior or on the boundary of $B$. Denote by $v_i$ (resp. $e_i$) the number of vertices (resp.
edges) in $G^\star_i$.
\item Denote by $m$ the number of edges to 
$B$ in $G^\star_i$ but not on $B$.
\item We assumed that at least one edge-injective function on $G^\star$ exists. Fix a choice of $\varphi\in \Phi$ amongst the edge-injective functions.
\end{itemize}

We will first count the number of edges $e_i$ in $G^\star_i$ by counting the number of edges into each vertex. Note that every vertex in 
$G^\star$ has exactly three incoming edges since $\varphi$ is an edge-injective function. Hence all vertices in the interior of $G^\star_i$ must have
exactly three incoming edges. Thus the number of edges to an interior vertex of $G^\star_i$ is $3(v_i-n)$. Combining this with the number of edges
$m+n$ to $B$, we have the counting formula
\begin{equation}
\label{equation.edges}
	e_i = 3(v_i-n) + m + n.
\end{equation}

Let $f_i$ be the number of faces in $G^\star_i$. Euler's formula for planar graphs asserts that $f_i-e_i+v_i=1$. Combining this 
with~\eqref{equation.edges}, we obtain
\begin{equation}
\label{equation.faces}
	f_i = e_i-v_i +1 = 2(v_i-n) + m+1.
\end{equation}

Let $k$ be the smallest positive integer such that a $k$-gon occurs in $G^\star_i$ as a face.
Every face in $G^\star_i$ has at least $k$ edges and every edge not on the boundary $B$ is on two faces. Hence 
\begin{equation}
\label{equation.bound faces}
	k f_i \leqslant 2e_i-n.
\end{equation}
Combining~\eqref{equation.edges}, \eqref{equation.faces}, \eqref{equation.bound faces}, we obtain
\[
	2(k-3)(v_i-n) + (k-2)m + k-n \leqslant 0.
\]
Note that if $G^\star_i$ does not contain any loops or multi-edges, we have $k \geqslant 3$.

\begin{corollary} 
\label{corollary: cycle bound}
Let $G_i^\star$ be a subgraph of $G^\star$ without loops or multi-edges and $n$ vertices on its boundary. Then
\[
	0 \leqslant m \leqslant n-3.
\]
\end{corollary}

\begin{corollary}
\label{corollary.double bound}
Let $G^\star$ be a finite, planar graph without multi-edges. Fix a planar embedding of $G^\star$. Suppose that there is an edge-injective function
$G^\star_\varphi$ on $G^\star$.

If $G^\star$ contains three vertices $w, v, u$ that form a triangle whose interior contains no loops, then $G^\star_\varphi$ has no 
directed paths from any one of $w, v, u$ to any other through the interior of the triangle.  

In particular, if $G^\star$ contains the subgraph below and $G^\star$ has no loops interior to the cycles $wvu_1w$ or $wvuw$, then only $u$ can lie on a 
directed path from $v$ to $w$ in $G^\star_\varphi$.
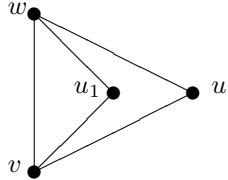
\[\begin{picture}(100,60)(0,-30)
\put(10,30){\circle*{5}}
\put(0,30){$w$}

\put(10,-30){\circle*{5}}
\put(0,-30){$v$}

\put(40,0){\circle*{5}}
\put(25,0){$u_1$}

\put(70,0){\circle*{5}}
\put(77,0){$u$}

\put(10,30){\line(0,-1){60}}
\put(10,30){\line(1,-1){30}}
\put(10,30){\line(2,-1){60}}

\put(10,-30){\line(1,1){30}}
\put(10,-30){\line(2,1){60}}
\end{picture}\]
\end{corollary}

\begin{proof}
Suppose that $B = wvuw$ is an (undirected) cycle in $G^\star$ and let $G^\star_i$ denote the graph consisting of all vertices and edges on $B$ or 
on the interior of $B$ in the given planar embedding of $G^\star$. The graph $G^\star_i$ has no loops or multi-edges by hypothesis and its boundary 
cycle $B$ has $n=3$ edges.  Corollary~\ref{corollary: cycle bound} implies that $G^\star_i$ has no edges into $B$ except for the edges on $B$. In other words, 
any edges between any of $u, v, w$ through a point on the interior of $B$ must be directed into the interior.  In particular, there cannot be a directed path 
from any of the vertices $u, v, w$ into another through the interior of $G^\star_i$.

Now assume that there is another vertex $u_1$ adjacent to both $v$ and $w$, without loss of generality in the configuration shown above. By the 
previous argument, we conclude the edges interior to $B$ must be directed $w \mapsto u_1$ and $v \mapsto u_1$, and that there is no directed 
path $u_1$ to any of $u, v, w$.  This proves the claim.
\end{proof}

We remark that the edge-injective function $\varphi$ guarantees that the vertex $u_1$ in $G^\star_\varphi$ must have three edges directed into it. Since $u_1$ is 
on the interior of the 3-cycle $B = wvuw$, there must be another edge $u' \mapsto u_1$ for a vertex $u'$ that is either interior to $wu_1vuw$ or 
on the boundary.  But multi-edges are not allowed so if $u'$ is not on the interior then $u'=u$.

%%%%%%%%%%%%%%%%%%%%%%%%%%%%%%%%%%%%%%%%%%%%%%%%%%%%%%%%%%%
\subsection{If an edge-injective function exists then a generic edge labeling exists}
\label{section.existence generic}

We now prove that a graph $G$ has a generic edge labeling if it admits at least one edge-injective function and is 
\begin{itemize}
    \item a finite, directed, planar graph without leaves (our standard condition),
    \item without holes (a condition added in the previous subsection),
    \item and each pair of faces in $G$ share at most one edge.   In the dual graph $G^\star$, the condition that
two faces share at most one edge is equivalent to the condition that $G^\star$ has no multi-edges.
\end{itemize}

{\bf We also assume} that $e_G \geqslant 3f_G$ and in fact, by ignoring edges if necessary, that $e_G = 3f_G$.  This may seem a restrictive assumption but Corollary~\ref{corollary.rank decomp} allows us to repeatedly decompose and contract, reducing the case either to a graph with $e_G \geqslant 3f_G$ and no contractible subset of faces or to the minimal contractible case.  Moreover, in the minimal contractible case, Lemmas~\ref{lemma.minimal contractible} and~\ref{lemma.min cont cond} guarantee that either $e_G=3f_G$ or $G$ is one of three possible graphs with a single face. So this final assumption is non-restrictive.

The itemized list of conditions together with $e_G \geqslant 3f_G$ are the assumptions of this section.

The structure of this section is as follows.  Proposition~\ref{proposition.generic} proves we can three-color in a particular way the edges of a graph satisfying the assumptions of this section.  The proof consists of straightforward but technical graph theory.  Theorem~\ref{theorem.generic} uses this result to show that if a graph satisfies the assumptions of this section and has at least one edge-injective function then it has a generic edge labeling.  Finally, we use Theorem~\ref{theorem.generic} to compute the rank of $M^\mathsf{ext}$ in both the minimal contractible case and the case when $G$ has no proper contractible subset of faces.

To begin, recall from Proposition~\ref{proposition.factoring det into edge-injective} that for $e_G = 3f_G$
\begin{equation}
\label{equation.det phi}
	\det M^{\mathsf{ext}} = \sum_{\substack{\varphi \colon \mathcal{F} \to \mathcal{E}_3\\ \text{edge-injective}}}
	\prod_{F \in \mathcal{F}} \det M^{\mathsf{ext}}_{F,\varphi(F)},
\end{equation}
where $M^{\mathsf{ext}}_{F,\varphi(F)}$ is the submatrix of $M^{\mathsf{ext}}$ with the 3 rows corresponding to face $F$ and columns indexed by
the edges in $\varphi(F)$.

\begin{proposition}
\label{proposition.generic}
Let $G$ be a graph with the assumptions of this section. Suppose that there is an edge-injective function $\varphi$ on $G$.
Then there exists a coloring of the edges of $G^\star_\varphi$ with colors $0,1,2$ such that:
\begin{enumerate}
\item for each vertex $F$ in $G^\star$ the three incoming edges into $F$ for $G^\star_\varphi$ have different colors $0,1,2$;
\item there are no monochromatic directed cycles (meaning, there are no directed cycles $C$ in $G^\star_\varphi$ for which the edges all have the same 
colors).
\end{enumerate}
\end{proposition}

\begin{proof}
We will show that we can assign labels in $\{0,1,2\}$ according to Conditions (1) and (2) by induction on the number of vertices in $G^\star$.
If $G^\star$ has only one vertex, then $G^\star_\varphi$ must be three loops at this vertex, which we can label $0,1,2$. This configuration
has no monochromatic directed cycles. If $G^\star$ has two vertices, $G_\varphi^\star$ is either disconnected and consists of two vertices with three loops
attached or there is one edge between the two vertices, one vertex has two loops and the other has three loops. It is easy to check that an
assignment with Condition (1) exists and since there are no directed cycles Condition (2) also holds.

\smallskip

We draw $G^\star$ using the same planar embedding as $G$, so that each vertex $v \in G^\star$ is in the interior of the corresponding face $F_v$ of 
$G$ and each edge $uv$ in $G^\star$ intersects the boundary $\overline{F_v} \cap \overline{F_u}$ in exactly one point.  Any closed curve $\gamma$ 
in $G$ has an \textit{exterior} defined as the connected component of  $\mathbb{R}^2 - \gamma$ containing the unbounded face.  (The boundary of 
$G$ is a simple closed curve so there is only one region exterior to the boundary of $G$.) Every edge in $G^\star$ is an embedded line segment in 
the plane so every cycle $C$ in $G^\star$ defines a closed curve in the plane.  We define the \textit{exterior} $C_o$ of $C$ as the subgraph with 
vertices and edges in $G^\star$ that are all exterior to $C$ or on $C$. The graph $G^\star$ is planar so if $u, v$ are two vertices in $C_o$ with an 
edge between them in $G^\star$, then in fact the edge is exterior to $C$.  In other words, the exterior $C_o$ is an induced subgraph of $G^\star$. 
Similarly, the \textit{interior} $C_i$ of $C$ consists of the vertices of $G^\star$ not in $C_o \cup C$ and the edges between those vertices in $G^\star$.   
The \textit{boundary} of $C$ is the intersection $C_o \cap C_i$, which equals $C$. 
    
\smallskip

First we reduce to the case when no (undirected) 3-cycle in $G^\star$ has any vertices or edges in the interior.  Indeed, assume there is a vertex $u$ 
with $wvuw$ an (undirected) 3-cycle in $G^\star$ and assume there is at least one vertex or edge on its interior.  By the definition of interior above, 
there are no loops on the interior of the cycle $wvuw$. Thus Corollary~\ref{corollary: cycle bound} shows that all 
edges between any of $w, v, u$ and the interior of the 3-cycle are pointing in.  Let $\widetilde{G}^\star$ be the graph with the interior of $wvuw$ removed. This 
removes at least one vertex, since any edge in the interior whose endpoints are both on the boundary would be a multi-edge, which by hypothesis 
cannot happen.  By induction, we may assign colors to the edges of the induced edge-injective function $H^\star_\varphi$ without any 
monochrome directed cycles.  Now consider the subgraph $\widetilde{H}^\star$ consisting of the interior of $wvuw$ together with a loop at 
$u'$ for any vertex $u'$ that has an edge directed into it from the boundary of the triangle in $G_\varphi^\star$.  We know that $\widetilde{H}^\star$ 
has fewer vertices than $G^\star$ since it does not have $w, v, u$ so we can also inductively assign colors to the edges in the induced edge-injective 
function $\widetilde{H}^\star_\varphi$ without monochrome directed cycles. Gluing the colorings for $H^\star_\varphi$ and 
$\widetilde{H}^\star_\varphi$ together produces a coloring of $G_\varphi^\star$ with no monochrome directed cycle.  Indeed, if there were a monochrome 
directed cycle, then the cycle must contain a point both from the interior of a triangle and from the boundary or exterior of a triangle.  But we showed that 
there is no path from the interior of a triangle out, which is a contradiction. 

\smallskip

\begin{figure}
\[
\begin{picture}(100,60)(-50,-30)
\put(10,30){\circle*{5}}
\put(0,35){$w$}

\put(10,-30){\circle*{5}}
\put(0,-35){$v$}

\put(40,0){\circle*{5}}
\put(45,0){$u_2$}

\put(-20,0){\circle*{5}}
\put(-35,0){$u_1$}

\put(10,30){\line(0,-1){60}}
\put(10,30){\line(1,-1){30}}
\put(-20,0){\line(1,1){30}}

\put(10,-30){\line(1,1){30}}
\put(10,-30){\line(-1,1){30}}
\end{picture}
\]
\caption{Kite configuration referred to in the proof of Proposition~\ref{proposition.generic}.
\label{figure.kite}}
\end{figure}  

Thus we may assume that every set of vertices $u, v, w$ in $G^\star$ with all possible edges between them has no vertices on its interior.  It further 
follows that there are no sets of vertices $w, v, u, u_1$ in the configuration shown in Corollary~\ref{corollary.double bound}. This means that if $v$ 
and $w$ have an edge between them in $G^\star$, then there are at most two vertices $u_1, u_2$ with the property that $u_i$ is adjacent to both 
$v$ and $w$ since neither $u_i$ can be in the interior of the 3-cycle formed by the other and $v, w$.  Moreover, if there are two vertices $u_1, u_2$ 
with this property then the 4-cycle $wu_1vu_2w$ contains the edge between $w$ and $v$ on its interior.  We call this the kite configuration, see
Figure~\ref{figure.kite}.

We claim that at least one of the vertices on the boundary of $G^\star$ has two loops. Indeed, let $B$ be the boundary of $G^\star$.  By the above arguments
and the choice of our planar embedding, the interior of $G^\star$ contains no loops.  Suppose that the boundary $B$ has $n$ edges.  The existence of the 
edge-injective function $\varphi$ demonstrates that the total in-degree for vertices on $B$ is $3n$. 
In a directed graph, the in-degree of a vertex is the number of edges that have this vertex as their endpoint.
Edges from the interior contribute at most $n-3$ of this 
total in-degree by Corollary~\ref{corollary: cycle bound}. The edges on $B$ contribute exactly $n$ of the in-degree.  Thus loops at boundary vertices account 
for the remaining incoming edges and there are at least 
\[
	3n - n - (n-3) = n+3.
\]
By the pigeonhole principle, this means at least one boundary vertex has two loops.  

Let $v$ be one of the boundary vertices of $G^\star$ with two loops. Let $w \mapsto v$ denote the unique vertex mapping to $v$ in $G^\star_\varphi$.
Consider the directed graph $\widetilde{G}^\star_\varphi$ obtained from $G^\star_\varphi$ by identifying vertices $w$ and $v$, contracting the edge 
$w \mapsto v$, turning into a loop any edge $v \mapsto u$ for which the vertex $u$ is adjacent to $w$, and omitting any loops at $v$.  More precisely, 
define $\widetilde{G}^\star_\varphi$ to have vertex set
\[
	\widetilde{V} = V(G^\star_\varphi) \backslash \{v\}
\]
and directed edge set 
\[
	\begin{array}{ll}
    	\widetilde{E} & = \{u \mapsto u': \textup{ if } u\mapsto u' \in E(G^\star_\varphi) \textup{ and } u, u' \neq v\} \\ 
	&\bigcup \{w \mapsto u: v \mapsto u \in E(G^\star_\varphi) \textup{ but neither } u \mapsto w \in E(G^\star_\varphi) \textup{ nor } w \mapsto u \in 
	E(G^\star_\varphi) \} \\ 
	&\bigcup \{u \mapsto u: \textup{ if } v \mapsto u \in E(G^\star_\varphi) \textup{ and either } w \mapsto u \in E(G^\star_\varphi) \textup{ or } u 
	\mapsto w \in E(G^\star_\varphi)\}. \end{array}
\]
Every vertex in $\widetilde{G}^\star_\varphi$ has exactly three incoming edges since we omitted all edges of $G^\star_\varphi$ into $v$ and created 
exactly one edge in $\widetilde{G}^\star_\varphi$ for all edges of $G^\star_\varphi$ into any vertex $u \neq v$.  

In addition, note that $\widetilde{G}^\star_\varphi$ has no multi-edges. Indeed, to create $\widetilde{E}$ we erased all edges into $v$, replaced some 
edges $v \mapsto u$ with edges $w \mapsto u$, and replaced the rest of the edges $v \mapsto u$ with loops $u \mapsto u$.  So suppose there are two 
non-loop edges between $u \neq u'$ in $G^\star_\varphi$. One of $u, u'$ must be $w$ else both edges are in $E(G^\star_\varphi)$, contradicting our 
assumption that $G^\star_\varphi$ has no multi-edges.  Moreover, one of the edges between $u$ and $w$ in $\widetilde{E}$ must correspond to an edge 
$v \mapsto u$ in $G^\star_\varphi$ else $G^\star_\varphi$ again has multi-edges.  But by the third part of the definition of $\widetilde{E}$ this means that
the edges incident to $u$ in $\widetilde{G}^\star_\varphi$ differ from those incident to $u$ in $G^\star_\varphi$ by exchanging $v \mapsto u$ for a loop 
$u \mapsto u$, which cannot create a multi-edge.

Since $\widetilde{G}^\star_\varphi$ has one fewer vertex than $G^\star_\varphi$ we conclude by induction that we can color its edges satisfying 
Conditions (1) and (2). If $u \mapsto u'$ is an edge in $G^\star_\varphi$ with $u' \neq v$, then assign it 
\begin{itemize}
    \item the same color as in $\widetilde{G}^\star_\varphi$ if $u \neq v$, 
    \item the same color as $w \mapsto u'$ if $u=v$ and there is no edge between $u'$ and $w$ in $G^\star_\varphi$, and
    \item the same color as $u' \mapsto u'$ if $u = v$ and there is an edge between $u'$ and $w$ in $G^\star_\varphi$.
\end{itemize}
This assigns a coloring to each edge in $G^\star_\varphi$ except those into $v$.  

Every directed cycle in $G^\star_\varphi$ either goes through $v$ or induces a directed cycle in $\widetilde{G}^\star_\varphi$. By induction, the graph 
$\widetilde{G}^\star_\varphi$ has no monochrome directed cycles.  So if we can choose colors for the three edges coming into $v$ that guarantees no 
directed cycle through $v$ is monochrome, then we have produced an edge-coloring of $G^\star_\varphi$ that satisfies Conditions (1) and (2).

In addition, assigning a color to the edge $w \mapsto v$ can only create a monochrome directed cycle in $G^\star_\varphi$ if there is a monochrome 
directed path from $v$ to $w$ in $\widetilde{G}^\star_\varphi$. Every directed path from $v$ to $w$ of length at least three in $G^\star_\varphi$ became 
a directed cycle through $w$ in $\widetilde{G}^\star_\varphi$ unless it began $v \mapsto u$ for some vertex $u$ that is adjacent to $w$.  By induction, 
the graph $\widetilde{G}^\star_\varphi$ has no monochrome directed cycles.  It follows that the only monochromatic paths from $v$ to $w$ in our 
edge-coloring of $G^\star_\varphi$ start with an edge $v \mapsto u$ for a vertex $u$ adjacent to $w$. 

We further showed that we can assume there are at most two vertices adjacent to both $w$ and $v$.  So every directed path between $v$ and $w$ in 
$G^\star_\varphi$ starts either with the edge $v \mapsto u_1$ or $v \mapsto u_2$ where $wu_1vu_2w$ forms a 4-cycle whose only interior is the edge 
$w \mapsto v$. In both cases, every directed monochrome path $v$ to $w$ starts either $v \mapsto u$ or $v \mapsto u'$.  Thus if we assign the edge 
$w \mapsto v$ a color different from those assigned to $v \mapsto u$ and $v \mapsto u'$, we have ensured there are no monochrome directed cycles.  
The two loops at $v$ can now be assigned the other two colors arbitrarily.
\end{proof}

\begin{example}
The edges of the following directed graph $G_\varphi^\star$ cannot be colored with three colors such that each vertex has three incoming edges of
different colors and such that there are no monochromatic cycles.
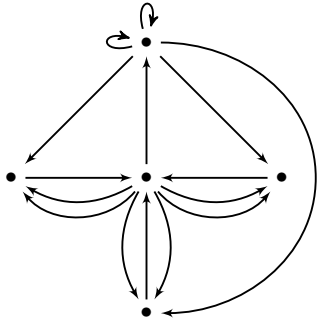
\begin{center}
\scalebox{0.9}{
\begin{tikzpicture}[auto]
\node (F1) at (0,0){$\bullet$};
\node (F2) at (0,-2){$\bullet$};
\node (F3) at (0,2){$\bullet$};
\node (F4) at (-2,0){$\bullet$};
\node (F5) at (2,0){$\bullet$};
\draw [edge, thick] (F1) edge [bend left=30] (F4);
\draw [edge, thick] (F1) edge [bend left=50] (F4);
\draw [edge, thick] (F1) edge [bend right=30] (F5);
\draw [edge, thick] (F1) edge [bend right=50] (F5);
\draw [edge, thick] (F1) edge [bend right=30] (F2);
\draw [edge, thick] (F1) edge [bend left=30] (F2);
\draw [->,edge, thick] (F3) to[out=0,in=90] (2.5,0) to[out=-90,in=0] (F2);
\path (F3) edge [thick,loop left] (F3);
\path (F3) edge [thick,loop above] (F3);
\draw[edge,thick] (F2) -- (F1);
\draw[edge,thick] (F1) -- (F3);
\draw[edge,thick] (F4) -- (F1);
\draw[edge,thick] (F3) -- (F4);
\draw[edge,thick] (F3) -- (F5);
\draw[edge,thick] (F5) -- (F1);
\node at (0,)[anchor=south east]{};
\end{tikzpicture}}
\end{center}
This graph has multiple edges between vertices and hence does not satisfy the condition of Proposition~\ref{proposition.generic} that no two faces
in the underlying graph $G$ share more than one edge. This shows that this assumption in Proposition~\ref{proposition.generic} is needed.
\end{example}

This brings us to the main theorem of this section.

\begin{theorem}
\label{theorem.generic}
Let $G$ be a graph with the assumptions of this section. Suppose that there is an edge-injective function $\varphi$ on $G$.
Then there exists a generic edge labeling.
\end{theorem}

\begin{proof}
To show that $G$ has a generic edge labeling, we want to show that $\det M^{\mathsf{ext}}$ is a nonzero polynomial in the variables $a_i$
coming from the edge labels. 
By assumption, there is at least one edge-injective function, so that the sum 
in~\eqref{equation.det phi} is nonempty. For a fixed edge-injective function $\varphi$, we will pick a particular monomial in the variables $a_i$
in $\prod_{F \in \mathcal{F}} \det M^{\mathsf{ext}}_{F,\varphi(F)}$, which cannot appear in the term for any other edge-injective function. 
We will do so by using that every edge-injective function can be obtained from another one by successively reversing directed cycles in $G^{\star}$ 
as shown in Lemma~\ref{lemma.directed cycles}. 

As shown before, each term $\det M^{\mathsf{ext}}_{F,\varphi(F)}$ is a Vandermonde determinant $\pm (a_1-a_2) (a_2-a_3) (a_3 - a_1)$, where 
$1,2,3$ are the indices for the edges in $\varphi(F)$. The monomials in this Vandermonde determinant are of the form $a_i^2 a_j^1 a_k^0$ for distinct 
indices $i,j,k\in \{1,2,3\}$. Viewing the edge-injective function $\varphi$ as an assignment of directions on the edges in $G^\star$, we can associate each 
monomial $a_i^2 a_j^1 a_k^0$ by assigning the exponents $2,1,0$ to the edges $i,j,k$, respectively. Hence the three incoming edges in 
$G^\star_\varphi$ for each vertex $F$ in $G^\star$ are assigned distinct numbers (or colors) $0,1,2$.

By Proposition~\ref{proposition.generic}, there exists a coloring of the edges of $G^\star_\varphi$ with colors $0,1,2$ such that:
\begin{enumerate}
\item for each vertex $F$ in $G^\star$ the three incoming edges into $F$ for $G^\star_\varphi$ have different colors $0,1,2$;
\item there are no monochromatic directed cycles (meaning, there are no directed cycles $C$ in $G^\star_\varphi$ for which the edges all have the same 
colors).
\end{enumerate}
Note that the second condition ensures that reversing cycle $C$ changes the monomial in the $a_i$. By Lemmas~\ref{lemma.directed cycles}
and~\ref{lemma.counting}, any two edge-injective functions $\varphi, \psi \in \Phi$ are related by reversing a directed cycle. Since there exists
a coloring without monochrome cycles, this term as a monomial in the $a_i$ cannot cancel. Hence a generic labeling exists. 
\end{proof}

We use the previous theorem to prove the following two results, which together are the heart of the paper.  The first computes the rank of $M^{\mathsf{ext}}$ when $G$ has no proper contractible subset of faces, and the second computes the rank in the minimal contractible case.

\begin{theorem}
\label{theorem.generic edge}
Let $G$ be a graph with the assumptions of this section without any proper contractible subset of faces and $e_G\geqslant 3f_G$.
Then there exists a generic edge labeling and $\rank(M^{\mathsf{ext}}) = 3f_G$.
\end{theorem}

\begin{proof}
By Corollary~\ref{cor.existence}, there exists an edge-injective function on $G$. Furthermore by Theorem~\ref{theorem.generic}, there exists a generic 
labeling $(G,\ell)$. Hence by~\eqref{equation.generic rank}, we have $\rank(M^{\mathsf{ext}}) = 3f_G$.
\end{proof}

\begin{theorem}
\label{theorem.generic edge minimal contractible}
Let $G$ be a graph with the assumptions of this section with a minimal contractible set of faces.
Then there exists a generic edge labeling and $\rank(M^{\mathsf{ext}}) = e_G$.
\end{theorem}

\begin{proof}
By Corollary~\ref{corollary.min contract}, since we are assuming that two faces share at most one edge, we have 
that $G$ either consists of a single face or $e_G-3f_G=0$. In the latter case the result follows from Theorem~\ref{theorem.generic edge}. For single faces, 
it is easy to check that generic edge labels exist.
\end{proof}

%%%%%%%%%%%%%%%%%%%%%%%%%%%%%%%%%%%%%%%%%%%%%%%%%%%%%%%%%%%
\section{Conclusion and open questions}
\label{section.algorithm}
%%%%%%%%%%%%%%%%%%%%%%%%%%%%%%%%%%%%%%%%%%%%%%%%%%%%%%%%%%%
\subsection{Algorithm to compute dimension of degree two splines}

The core challenge in computing the dimension of splines of degree two is that the quantity $e-3f$ may not be the same as the rank of 
$M^{\mathsf{ext}}$, or equivalently, the minimum number $e-|\varphi(\mathcal{F})|$ of edges leftover by an edge-injective function $\varphi$. 
The arguments used in Theorem~\ref{theorem: no proper contractible subset means maximal edge-injective function} actually outline a greedy 
algorithm that constructs a maximal edge-injective function, and thus compute the dimension of splines of degree two, 
thereby proving the upper-bound conjecture for $d=2$ for graphs with edge labels not in special position.

In conclusion, we now describe the algorithm derived from Theorem~\ref{theorem: no proper contractible subset means maximal edge-injective function} 
to compute the dimension of degree two splines.
Let $G$ be a finite, directed, planar graph without holes in which any two faces share at most one edge. In this section we allow $G$ to have leaves.  
Consider the following algorithm.

\begin{enumerate}
    \item Order the faces $F_1, F_2, \ldots$ in $G$ so that each $F_{i+1}$ shares an edge with at least one of $F_1, \ldots, F_{i-1}$.  \item Define $\varphi_1(F_1)$ to be any three edges on the boundary of $F_1$ or, if $F_1$ has only one or two edges on its boundary, all edges on the boundary of $F_1$.  
    \item Suppose $\varphi_i$ has been defined on $F_1, \ldots, F_i$.  Assign $\varphi_{i+1}(F_{i+1})$ to be up to three edges on the boundary of $F_{i+1}$ that have not already been assigned, namely are not in $\textup{Im}(\varphi_i)$.
    \item If the edge-injective function is saturated in the sense that
    \[|\varphi_{i+1}(F_{i+1})| = \min\{3, \# \textup{ boundary edges of } F_{i+1}\}\] 
    then increment the index $i$ and return to step (3).
    \item Otherwise, for each edge $e$ on the boundary of $F_{i+1}$ that is not in $\varphi_{i+1}(F_{i+1})$, use the following process from Theorem~\ref{theorem: no proper contractible subset means maximal edge-injective function}:
    \begin{enumerate}
        \item Search along $\varphi_{i+1}$-compatible paths of faces that end in $e$ for a face $F_j$ whose boundary edges have not all been assigned, namely are not all contained in $\textup{Im}(\varphi_{i+1})$.  
        \item If such a face $F_j$ is identified, then update $\varphi_{i+1}$ as specified in Theorem~\ref{theorem: no proper contractible subset means maximal edge-injective function} and go to step (4).
        \item If no such face $F_j$ is identified, then go to step (3).
    \end{enumerate}
\end{enumerate}

The time complexity of step (5a) can be reduced by storing the data of faces $\mathcal{F}_u$ whose bounding edges are not all used, together with, for each other face $F$, a $\varphi_i$-compatible path from a face in $\mathcal{F}_u$ to $F$ if one exists.  This data only needs to be modified whenever $\varphi_{i+1}$ is updated in step (5) and can be identified via standard methods for minimal-weight spanning trees.

If all edge-labels are distinct, this algorithm terminates with an edge-injective function with the property that $e - |\textup{Im}(\varphi_f)|$ is the rank of $M^{\mathsf{ext}}$.  Indeed, the map $\varphi_1$ is by construction saturated. Assume as the inductive hypothesis that $\varphi_i$ has been defined so that the quantity $e-|\textup{Im}(\varphi_i)|$ is as small as possible.  The algorithm extends $\varphi_{i+1}$ in steps (5a) and (5b) unless no face linked to $F_{i+1}$ by $\varphi_{i+1}$-compatible paths has an unassigned bounding edge.  In this case,  Theorem~\ref{theorem: no proper contractible subset means maximal edge-injective function} shows that $F_{i+1}$ is part of a contractible subgraph and reaching step (5c) confirms the edges of the subgraph are all contained in $\textup{Im}(\varphi_{i+1})$.  In other words, the quantity $e-|\textup{Im}(\varphi_{i+1})|$ is as small as possible if step (5c) is attained.  By induction, the algorithm terminates with an edge-injective function $\varphi_f$ whose image has as many edges as possible.

\begin{example}
We apply the above algorithm to compute the dimension of the non-generic labelled graph in Example~\ref{non-generic graph} for two different orderings of the faces. If $\varphi_i$ assigns edge $e$ to face $F$ we draw a small vector on edge $e$ pointing into $F$. Note that $e-3f=0$ for this graph. 

We first implement with optimal choices:
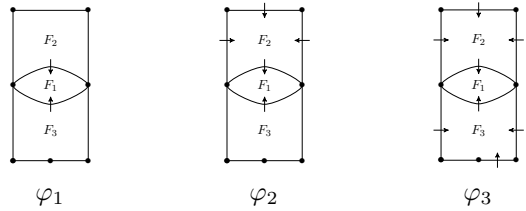
\[
\begin{array}{ccc}
\scalebox{0.5}{\begin{tikzpicture}
\draw (0,-2) to (2,-2);
\draw (0,2) to (2,2);
\draw (0,0) to (0,2);
\draw (0,0) to (0,-2);
\draw (2,0) to (2,-2);
\draw (2,0) to (2,2);
\draw plot [smooth cycle] coordinates {(0,0) (1,0.5) (2,0) (1,-0.5)};
\node at (1,-2){$\bullet$};
\node at (0,-2){$\bullet$};
\node at (2,-2){$\bullet$};
\node at (0,2){$\bullet$};
\node at (0,-2){$\bullet$};
\node at (0,0){$\bullet$};
\node at (2,0){$\bullet$};
\node at (2,2){$\bullet$};
\node at (1,1.2){$F_2$};
\node at (1,0){$F_1$};
\node at (1,-1.2){$F_3$};
\node at (0,)[anchor=south east]{};
\draw[<-] (1,-0.3) to (1, -0.7); %center b
\draw[<-] (1,0.3) to (1, 0.7); %center t
%\draw[->] (-0.2,-1.2) to (0.2,-1.2); %left b
%\draw[<-] (1.8,-1.2) to (2.2,-1.2); %right b
%\draw[->] (-0.2,1.2) to (0.2,1.2); %left t
%\draw[<-] (1.8,1.2) to (2.2,1.2); %right t
%\draw[<-] (1,1.8) to (1, 2.2); %top edge
%\draw[<-] (1.5,-1.8) to (1.5, -2.2); %bottom l
%\draw[<-] (0.5,-1.8) to (0.5, -2.2); %bottom r
\end{tikzpicture}} \hspace{0.5in} &
\scalebox{0.5}{\begin{tikzpicture}
\draw (0,-2) to (2,-2);
\draw (0,2) to (2,2);
\draw (0,0) to (0,2);
\draw (0,0) to (0,-2);
\draw (2,0) to (2,-2);
\draw (2,0) to (2,2);
\draw plot [smooth cycle] coordinates {(0,0) (1,0.5) (2,0) (1,-0.5)};
\node at (1,-2){$\bullet$};
\node at (0,-2){$\bullet$};
\node at (2,-2){$\bullet$};
\node at (0,2){$\bullet$};
\node at (0,-2){$\bullet$};
\node at (0,0){$\bullet$};
\node at (2,0){$\bullet$};
\node at (2,2){$\bullet$};
\node at (0,)[anchor=south east]{};
\node at (1,1.2){$F_2$};
\node at (1,0){$F_1$};
\node at (1,-1.2){$F_3$};
\draw[<-] (1,-0.3) to (1, -0.7); %center b
\draw[<-] (1,0.3) to (1, 0.7); %center t
%\draw[->] (-0.2,-1.2) to (0.2,-1.2); %left b
%\draw[<-] (1.8,-1.2) to (2.2,-1.2); %right b
\draw[->] (-0.2,1.2) to (0.2,1.2); %left t
\draw[<-] (1.8,1.2) to (2.2,1.2); %right t
\draw[<-] (1,1.8) to (1, 2.2); %top edge
%\draw[<-] (1.5,-1.8) to (1.5, -2.2); %bottom l
%\draw[<-] (0.5,-1.8) to (0.5, -2.2); %bottom r
\end{tikzpicture}} \hspace{0.5in} &
\scalebox{0.5}{\begin{tikzpicture}
\draw (0,-2) to (2,-2);
\draw (0,2) to (2,2);
\draw (0,0) to (0,2);
\draw (0,0) to (0,-2);
\draw (2,0) to (2,-2);
\draw (2,0) to (2,2);
\draw plot [smooth cycle] coordinates {(0,0) (1,0.5) (2,0) (1,-0.5)};
\node at (1,-2){$\bullet$};
\node at (0,-2){$\bullet$};
\node at (2,-2){$\bullet$};
\node at (0,2){$\bullet$};
\node at (0,-2){$\bullet$};
\node at (0,0){$\bullet$};
\node at (2,0){$\bullet$};
\node at (2,2){$\bullet$};
\node at (0,)[anchor=south east]{};
\node at (1,1.2){$F_2$};
\node at (1,0){$F_1$};
\node at (1,-1.2){$F_3$};
\draw[<-] (1,-0.3) to (1, -0.7); %center b
\draw[<-] (1,0.3) to (1, 0.7); %center t
\draw[->] (-0.2,-1.2) to (0.2,-1.2); %left b
\draw[<-] (1.8,-1.2) to (2.2,-1.2); %right b
\draw[->] (-0.2,1.2) to (0.2,1.2); %left t
\draw[<-] (1.8,1.2) to (2.2,1.2); %right t
\draw[<-] (1,1.8) to (1, 2.2); %top edge
\draw[<-] (1.5,-1.8) to (1.5, -2.2); %bottom l
%\draw[<-] (0.5,-1.8) to (0.5, -2.2); %bottom r
\end{tikzpicture}} \\
\varphi_1 \hspace{0.5in} & \varphi_2 \hspace{0.5in} & \varphi_3
\end{array}\]

We now implement with suboptimal choices, prompted in part by a less-efficient ordering of faces:
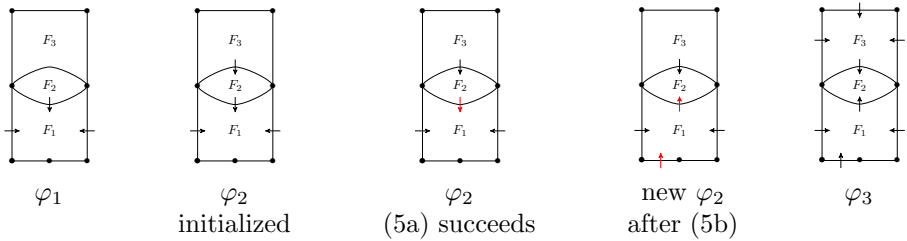
\[
\begin{array}{ccccc}
\scalebox{0.5}{\begin{tikzpicture}
\draw (0,-2) to (2,-2);
\draw (0,2) to (2,2);
\draw (0,0) to (0,2);
\draw (0,0) to (0,-2);
\draw (2,0) to (2,-2);
\draw (2,0) to (2,2);
\draw plot [smooth cycle] coordinates {(0,0) (1,0.5) (2,0) (1,-0.5)};
\node at (1,-2){$\bullet$};
\node at (0,-2){$\bullet$};
\node at (2,-2){$\bullet$};
\node at (0,2){$\bullet$};
\node at (0,-2){$\bullet$};
\node at (0,0){$\bullet$};
\node at (2,0){$\bullet$};
\node at (2,2){$\bullet$};
\node at (1,1.2){$F_3$};
\node at (1,0){$F_2$};
\node at (1,-1.2){$F_1$};
\node at (0,)[anchor=south east]{};
\draw[->] (1,-0.3) to (1, -0.7); %center b
%\draw[<-] (1,0.3) to (1, 0.7); %center t
\draw[->] (-0.2,-1.2) to (0.2,-1.2); %left b
\draw[<-] (1.8,-1.2) to (2.2,-1.2); %right b
%\draw[->] (-0.2,1.2) to (0.2,1.2); %left t
%\draw[<-] (1.8,1.2) to (2.2,1.2); %right t
%\draw[<-] (1,1.8) to (1, 2.2); %top edge
%\draw[<-] (1.5,-1.8) to (1.5, -2.2); %bottom l
%\draw[<-] (0.5,-1.8) to (0.5, -2.2); %bottom r
\end{tikzpicture}} \hspace{0.25in} &
\scalebox{0.5}{\begin{tikzpicture}
\draw (0,-2) to (2,-2);
\draw (0,2) to (2,2);
\draw (0,0) to (0,2);
\draw (0,0) to (0,-2);
\draw (2,0) to (2,-2);
\draw (2,0) to (2,2);
\draw plot [smooth cycle] coordinates {(0,0) (1,0.5) (2,0) (1,-0.5)};
\node at (1,-2){$\bullet$};
\node at (0,-2){$\bullet$};
\node at (2,-2){$\bullet$};
\node at (0,2){$\bullet$};
\node at (0,-2){$\bullet$};
\node at (0,0){$\bullet$};
\node at (2,0){$\bullet$};
\node at (2,2){$\bullet$};
\node at (0,)[anchor=south east]{};
\node at (1,1.2){$F_3$};
\node at (1,0){$F_2$};
\node at (1,-1.2){$F_1$};
\draw[->] (1,-0.3) to (1, -0.7); %center b
\draw[<-] (1,0.3) to (1, 0.7); %center t
\draw[->] (-0.2,-1.2) to (0.2,-1.2); %left b
\draw[<-] (1.8,-1.2) to (2.2,-1.2); %right b
%\draw[->] (-0.2,1.2) to (0.2,1.2); %left t
%\draw[<-] (1.8,1.2) to (2.2,1.2); %right t
%\draw[<-] (1,1.8) to (1, 2.2); %top edge
%\draw[<-] (1.5,-1.8) to (1.5, -2.2); %bottom l
%\draw[<-] (0.5,-1.8) to (0.5, -2.2); %bottom r
\end{tikzpicture}} \hspace{0.25in} &
\scalebox{0.5}{\begin{tikzpicture}
\draw (0,-2) to (2,-2);
\draw (0,2) to (2,2);
\draw (0,0) to (0,2);
\draw (0,0) to (0,-2);
\draw (2,0) to (2,-2);
\draw (2,0) to (2,2);
\draw plot [smooth cycle] coordinates {(0,0) (1,0.5) (2,0) (1,-0.5)};
\node at (1,-2){$\bullet$};
\node at (0,-2){$\bullet$};
\node at (2,-2){$\bullet$};
\node at (0,2){$\bullet$};
\node at (0,-2){$\bullet$};
\node at (0,0){$\bullet$};
\node at (2,0){$\bullet$};
\node at (2,2){$\bullet$};
\node at (0,)[anchor=south east]{};
\node at (1,1.2){$F_3$};
\node at (1,0){$F_2$};
\node at (1,-1.2){$F_1$};
\draw[->, red] (1,-0.3) to (1, -0.7); %center b
\draw[<-] (1,0.3) to (1, 0.7); %center t
\draw[->] (-0.2,-1.2) to (0.2,-1.2); %left b
\draw[<-] (1.8,-1.2) to (2.2,-1.2); %right b
%\draw[->] (-0.2,1.2) to (0.2,1.2); %left t
%\draw[<-] (1.8,1.2) to (2.2,1.2); %right t
%\draw[<-] (1,1.8) to (1, 2.2); %top edge
%\draw[<-] (1.5,-1.8) to (1.5, -2.2); %bottom l
%\draw[<-] (0.5,-1.8) to (0.5, -2.2); %bottom r
\end{tikzpicture}} \hspace{0.25in} &
\scalebox{0.5}{\begin{tikzpicture}
\draw (0,-2) to (2,-2);
\draw (0,2) to (2,2);
\draw (0,0) to (0,2);
\draw (0,0) to (0,-2);
\draw (2,0) to (2,-2);
\draw (2,0) to (2,2);
\draw plot [smooth cycle] coordinates {(0,0) (1,0.5) (2,0) (1,-0.5)};
\node at (1,-2){$\bullet$};
\node at (0,-2){$\bullet$};
\node at (2,-2){$\bullet$};
\node at (0,2){$\bullet$};
\node at (0,-2){$\bullet$};
\node at (0,0){$\bullet$};
\node at (2,0){$\bullet$};
\node at (2,2){$\bullet$};
\node at (0,)[anchor=south east]{};
\node at (1,1.2){$F_3$};
\node at (1,0){$F_2$};
\node at (1,-1.2){$F_1$};
\draw[<-, red] (1,-0.3) to (1, -0.7); %center b
\draw[<-] (1,0.3) to (1, 0.7); %center t
\draw[->] (-0.2,-1.2) to (0.2,-1.2); %left b
\draw[<-] (1.8,-1.2) to (2.2,-1.2); %right b
%\draw[->] (-0.2,1.2) to (0.2,1.2); %left t
%\draw[<-] (1.8,1.2) to (2.2,1.2); %right t
%\draw[<-] (1,1.8) to (1, 2.2); %top edge
%\draw[<-] (1.5,-1.8) to (1.5, -2.2); %bottom l
\draw[<-,red] (0.5,-1.8) to (0.5, -2.2); %bottom r
\end{tikzpicture}} \hspace{0.25in} &
\scalebox{0.5}{\begin{tikzpicture}
\draw (0,-2) to (2,-2);
\draw (0,2) to (2,2);
\draw (0,0) to (0,2);
\draw (0,0) to (0,-2);
\draw (2,0) to (2,-2);
\draw (2,0) to (2,2);
\draw plot [smooth cycle] coordinates {(0,0) (1,0.5) (2,0) (1,-0.5)};
\node at (1,-2){$\bullet$};
\node at (0,-2){$\bullet$};
\node at (2,-2){$\bullet$};
\node at (0,2){$\bullet$};
\node at (0,-2){$\bullet$};
\node at (0,0){$\bullet$};
\node at (2,0){$\bullet$};
\node at (2,2){$\bullet$};
\node at (0,)[anchor=south east]{};
\node at (1,1.2){$F_3$};
\node at (1,0){$F_2$};
\node at (1,-1.2){$F_1$};
\draw[<-] (1,-0.3) to (1, -0.7); %center b
\draw[<-] (1,0.3) to (1, 0.7); %center t
\draw[->] (-0.2,-1.2) to (0.2,-1.2); %left b
\draw[<-] (1.8,-1.2) to (2.2,-1.2); %right b
\draw[->] (-0.2,1.2) to (0.2,1.2); %left t
\draw[<-] (1.8,1.2) to (2.2,1.2); %right t
\draw[<-] (1,1.8) to (1, 2.2); %top edge
%\draw[<-] (1.5,-1.8) to (1.5, -2.2); %bottom l
\draw[<-] (0.5,-1.8) to (0.5, -2.2); %bottom r
\end{tikzpicture}} \\
\varphi_1 \hspace{0.25in} & \varphi_2 \hspace{0.25in} & \varphi_2 \hspace{0.25in} & \textup{ new }\varphi_2 \hspace{0.25in} & \varphi_3 \\
& \textup{ initialized } \hspace{0.25in} & \textup{ (5a) succeeds } \hspace{0.25in} & \textup{ after (5b)} \hspace{0.25in} & 
\end{array}\]
In this second case, no matter how $\varphi_3$ is initialized, there is no $\varphi_3$-compatible path of faces from the top face to the bottom face.  This means that if the top face had fewer boundary edges, the algorithm would reach step (5c) at this point, reflecting the (essentially equivalent) facts that 
\begin{itemize}
    \item $\{F_2, F_3\}$ form a contractible subset of faces,
    \item edges cannot be reassigned from outside $\{F_2, F_3\}$ to give $\varphi(\{F_2, F_3\})$ more, and
    \item the cardinality $|\textup{Im}(\varphi_3)|$ is as large as possible. 
\end{itemize}
So $e-|\textup{Im}(\varphi_3)|=1$ as desired.
\end{example}

Alternatively, to compute the dimension of the spline $\mathsf{Spl}_2(G,\ell; v_0)$, find a minimal contractible subset of faces (see 
Definition~\ref{definition.contractible}). If a minimal contractible subset of faces exists, contract it. By Theorem~\ref{theorem.generic edge minimal contractible} 
a generic edge labeling for this minimal contractible subset of faces exists.  Hence by Corollary~\ref{corollary.dim decomp}, the dimension of the space of 
splines does not change under this operation. Furthermore, the contracted graph still has the properties that it is finite, planar, without holes, and no two 
faces share more than one edge. Repeat this process until what remains is a graph $G$ with no proper contractible subset of faces. By 
Theorems~\ref{theorem.generic edge} and~\ref{theorem.spline rank}
\[
	\dim \mathsf{Spl}_2(G,\ell; v_0) = e_G - 3f_G
\]
which is also the dimension of the space of splines for the original graph.	

In conclusion, the two algorithms in this section compute the dimensions of splines of degree two when the edge labels are not in special position.

%%%%%%%%%%%%%%%%%%%%%%%%%%%%%%%%%%%%%%%%%%%%%%%%%%%%%
\subsection{Open questions} 
\label{section: open questions}

We end with some directions for future research.

Section~\ref{section.homogeneous} described a method to homogenize splines that is different from the usual homogenization process of algebraic geometry, which embeds objects in a larger projective space.  By contrast, our homogenization translates all lines simultaneously to the origin, essentially decreasing dimension.  This raises a natural question.

\begin{question}\label{question: homogenization means what?}
Does the homogenization process of Section~\ref{section.homogeneous} have a geometric interpretation, either for general graphs or for graphs dual to planar triangulations?
\end{question}

Corollary~\ref{corollary: contractible triangulation means subdivided triangle} showed that if a triangulation $\Delta$ in the plane is dual to contractible graph, then $\Delta$ is a subdivided triangle. Note that triangulations on surfaces other than the plane behave differently and cases where $e-3f<0$ can occur naturally.  For instance, drawing a cycle of length three along the equator partitions the (two-dimensional) sphere into two faces with three-cycles along their boundary.  The dual to this triangulation is the multigraph consisting of two vertices with three edges between them, which has two minimal contractible faces, each with two boundary edges.  This gives rise to the following.

\begin{question} \label{question: whiteley comment}
How many of the results of this paper, including the descriptions of contractibility from Section~\ref{section.contractible}, can be extended to splines on surfaces of higher genus?
\end{question}

%%%%%%%%%%%%%%%%%%%%%%%%%%%%%%%%%%%%%%%%%%%%%%%%%%%%%
\bibliography{main}{}

@book {LaiSchumaker.2007,
    AUTHOR = {Lai, Ming-Jun and Schumaker, Larry L.},
     TITLE = {Spline functions on triangulations},
    SERIES = {Encyclopedia of Mathematics and its Applications},
    VOLUME = {110},
 PUBLISHER = {Cambridge University Press, Cambridge},
      YEAR = {2007},
     PAGES = {xvi+592},
      ISBN = {978-0-521-87592-9; 0-521-87592-7},
   MRCLASS = {41-02 (41A15 41A63 65D05 65D07)},
  MRNUMBER = {2355272},
MRREVIEWER = {Peter Alfeld},
       DOI = {10.1017/CBO9780511721588},
       URL = {https://doi.org/10.1017/CBO9780511721588},
}

@article{billera1991dimension,
  title={A dimension series for multivariate splines},
  author={Billera, Louis J and Rose, Lauren L},
  journal={Discrete \& Computational Geometry},
  volume={6},
  pages={107--128},
  year={1991},
  publisher={Springer}
}

@article {GTV.2016,
    AUTHOR = {Gilbert, Simcha and Tymoczko, Julianna and Viel, Shira},
     TITLE = {Generalized splines on arbitrary graphs},
   JOURNAL = {Pacific J. Math.},
  FJOURNAL = {Pacific Journal of Mathematics},
    VOLUME = {281},
      YEAR = {2016},
    NUMBER = {2},
     PAGES = {333--364},
      ISSN = {0030-8730},
   MRCLASS = {05C78 (05C25 05E15 55N25)},
  MRNUMBER = {3463040},
MRREVIEWER = {Christopher Park Mooney},
       DOI = {10.2140/pjm.2016.281.333},
       URL = {https://doi.org/10.2140/pjm.2016.281.333},
}

@article {Billera.1991,
    AUTHOR = {Billera, L. J. and Haas, R.},
     TITLE = {The dimension and bases of divergence-free splines; a
              homological approach},
   JOURNAL = {Approx. Theory Appl.},
  FJOURNAL = {Approximation Theory and its Applications},
    VOLUME = {7},
      YEAR = {1991},
    NUMBER = {1},
     PAGES = {91--99},
      ISSN = {1000-9221},
   MRCLASS = {41A15 (13D25 41A63 65D07)},
  MRNUMBER = {1117311},
MRREVIEWER = {Rong Qing Jia},
}

@article {GKM.1998,
    AUTHOR = {Goresky, Mark and Kottwitz, Robert and MacPherson, Robert},
     TITLE = {Equivariant cohomology, {K}oszul duality, and the localization
              theorem},
   JOURNAL = {Invent. Math.},
  FJOURNAL = {Inventiones Mathematicae},
    VOLUME = {131},
      YEAR = {1998},
    NUMBER = {1},
     PAGES = {25--83},
      ISSN = {0020-9910},
   MRCLASS = {55N91 (14F25 14F32 16E99 18G10 55N33)},
  MRNUMBER = {1489894},
MRREVIEWER = {Roy Joshua},
       DOI = {10.1007/s002220050197},
       URL = {https://doi.org/10.1007/s002220050197},
}

@article {SSY.2020,
    AUTHOR = {Schenck, Hal and Stillman, Mike and Yuan, Beihui},
     TITLE = {A new bound for smooth spline spaces},
   JOURNAL = {J. Comb. Algebra},
  FJOURNAL = {Journal of Combinatorial Algebra},
    VOLUME = {4},
      YEAR = {2020},
    NUMBER = {4},
     PAGES = {359--367},
      ISSN = {2415-6302},
   MRCLASS = {41A15 (13D40)},
  MRNUMBER = {4186484},
MRREVIEWER = {Tatyana Sorokina},
       DOI = {10.4171/jca/43},
       URL = {https://doi.org/10.4171/jca/43},
}

@article {AS.1990,
    AUTHOR = {Alfeld, Peter and Schumaker, Larry L.},
     TITLE = {On the dimension of bivariate spline spaces of smoothness
              {$r$} and degree {$d=3r+1$}},
   JOURNAL = {Numer. Math.},
  FJOURNAL = {Numerische Mathematik},
    VOLUME = {57},
      YEAR = {1990},
    NUMBER = {6-7},
     PAGES = {651--661},
      ISSN = {0029-599X},
   MRCLASS = {65D07},
  MRNUMBER = {1062372},
       DOI = {10.1007/BF01386434},
       URL = {https://doi.org/10.1007/BF01386434},
}

@article {Hong.1991,
    AUTHOR = {Hong, Dong},
     TITLE = {Spaces of bivariate spline functions over triangulation},
   JOURNAL = {Approx. Theory Appl.},
  FJOURNAL = {Approximation Theory and its Applications},
    VOLUME = {7},
      YEAR = {1991},
    NUMBER = {1},
     PAGES = {56--75},
      ISSN = {1000-9221},
   MRCLASS = {65D07 (41A15 41A63)},
  MRNUMBER = {1117308},
MRREVIEWER = {Peter Alfeld},
}

@incollection {Strang.1974,
    AUTHOR = {Strang, Gilbert},
     TITLE = {The dimension of piecewise polynomial spaces, and one-sided
              approximation},
 BOOKTITLE = {Conference on the {N}umerical {S}olution of {D}ifferential
              {E}quations ({U}niv. {D}undee, {D}undee, 1973)},
    SERIES = {Lecture Notes in Math., Vol. 363},
     PAGES = {144--152},
 PUBLISHER = {Springer, Berlin},
      YEAR = {1974},
   MRCLASS = {41A30 (65N35)},
  MRNUMBER = {0430621},
MRREVIEWER = {Donald Kershaw},
}

@article {Strang.1973,
    AUTHOR = {Strang, Gilbert},
     TITLE = {Piecewise polynomials and the finite element method},
   JOURNAL = {Bull. Amer. Math. Soc.},
  FJOURNAL = {Bulletin of the American Mathematical Society},
    VOLUME = {79},
      YEAR = {1973},
     PAGES = {1128--1137},
      ISSN = {0002-9904},
   MRCLASS = {65N30},
  MRNUMBER = {327060},
MRREVIEWER = {J.\ R.\ Kuttler},
       DOI = {10.1090/S0002-9904-1973-13351-8},
       URL = {https://doi.org/10.1090/S0002-9904-1973-13351-8},
}

@article {Billera.1988,
    AUTHOR = {Billera, Louis J.},
     TITLE = {Homology of smooth splines: generic triangulations and a
              conjecture of {S}trang},
   JOURNAL = {Trans. Amer. Math. Soc.},
  FJOURNAL = {Transactions of the American Mathematical Society},
    VOLUME = {310},
      YEAR = {1988},
    NUMBER = {1},
     PAGES = {325--340},
      ISSN = {0002-9947},
   MRCLASS = {41A15 (65D07)},
  MRNUMBER = {965757},
MRREVIEWER = {L. L. Schumaker},
       DOI = {10.2307/2001125},
       URL = {https://doi.org/10.2307/2001125},
}

@article {MacLane.1937,
    AUTHOR = {MacLane, Saunders},
     TITLE = {A combinatorial condition for planar graphs},
   JOURNAL = {Fundamenta Mathematicae},
    VOLUME = {28},
      YEAR = {1937},
     PAGES = {22--32},
       DOI = {10.4064/fm-28-1-22-32},
}

@article {Diestel.2012,
    AUTHOR = {Diestel, Reinhard and Spr\"{u}ssel, Philipp},
     TITLE = {Locally finite graphs with ends: a topological approach,
              {III}. {F}undamental group and homology},
   JOURNAL = {Discrete Math.},
  FJOURNAL = {Discrete Mathematics},
    VOLUME = {312},
      YEAR = {2012},
    NUMBER = {1},
     PAGES = {21--29},
      ISSN = {0012-365X},
   MRCLASS = {05C83 (05-02 05C10 05C25)},
  MRNUMBER = {2852503},
       DOI = {10.1016/j.disc.2011.02.007},
       URL = {https://doi.org/10.1016/j.disc.2011.02.007},
}

@article {Diestel.2010,
    AUTHOR = {Diestel, Reinhard},
     TITLE = {Locally finite graphs with ends: a topological approach, {II}.
              {A}pplications},
   JOURNAL = {Discrete Math.},
  FJOURNAL = {Discrete Mathematics},
    VOLUME = {310},
      YEAR = {2010},
    NUMBER = {20},
     PAGES = {2750--2765},
      ISSN = {0012-365X},
   MRCLASS = {05C85 (05C15)},
  MRNUMBER = {2672223},
       DOI = {10.1016/j.disc.2010.05.027},
       URL = {https://doi.org/10.1016/j.disc.2010.05.027},
}

@article {Diestel.2011,
    AUTHOR = {Diestel, Reinhard},
     TITLE = {Locally finite graphs with ends: a topological approach, {I}.
              {B}asic theory},
   JOURNAL = {Discrete Math.},
  FJOURNAL = {Discrete Mathematics},
    VOLUME = {311},
      YEAR = {2011},
    NUMBER = {15},
     PAGES = {1423--1447},
      ISSN = {0012-365X},
   MRCLASS = {05C63 (05C10)},
  MRNUMBER = {2800969},
MRREVIEWER = {Agelos Georgakopoulos},
       DOI = {10.1016/j.disc.2010.05.023},
       URL = {https://doi.org/10.1016/j.disc.2010.05.023},
}

@article {ASW.1993,
    AUTHOR = {Alfeld, Peter and Schumaker, Larry L. and Whiteley, Walter},
     TITLE = {The generic dimension of the space of {$C^1$} splines of
              degree {$d\geq 8$} on tetrahedral decompositions},
   JOURNAL = {SIAM J. Numer. Anal.},
  FJOURNAL = {SIAM Journal on Numerical Analysis},
    VOLUME = {30},
      YEAR = {1993},
    NUMBER = {3},
     PAGES = {889--920},
      ISSN = {0036-1429},
   MRCLASS = {65D07 (41A15 41A63)},
  MRNUMBER = {1220660},
MRREVIEWER = {Georgi\ Grozev},
       DOI = {10.1137/0730047},
       URL = {https://doi.org/10.1137/0730047},
}

@incollection {Wh.1991,
    AUTHOR = {Whiteley, Walter},
     TITLE = {The combinatorics of bivariate splines},
 BOOKTITLE = {Applied geometry and discrete mathematics},
    SERIES = {DIMACS Ser. Discrete Math. Theoret. Comput. Sci.},
    VOLUME = {4},
     PAGES = {587--608},
 PUBLISHER = {Amer. Math. Soc., Providence, RI},
      YEAR = {1991},
      ISBN = {0-8218-6593-5},
   MRCLASS = {41A15 (52B55)},
  MRNUMBER = {1116378},
MRREVIEWER = {Wolfgang\ Boehm},
       DOI = {10.1090/dimacs/004/44},
       URL = {https://doi.org/10.1090/dimacs/004/44},
}

@incollection {Wh.1991a,
    AUTHOR = {Whiteley, Walter},
     TITLE = {A matrix for splines},
 BOOKTITLE = {Progress in approximation theory},
     PAGES = {821--828},
 PUBLISHER = {Academic Press, Boston, MA},
      YEAR = {1991},
      ISBN = {0-12-516750-4},
   MRCLASS = {65D07 (41A15)},
  MRNUMBER = {1114817},
}

@article {MS.1975,
    AUTHOR = {Morgan, John and Scott, Ridgway},
     TITLE = {A nodal basis for {$C^{1}$} piecewise polynomials of degree
              {$n\geq 5$}},
   JOURNAL = {Math. Comput.},
  FJOURNAL = {Mathematics and Computers in Simulation},
    VOLUME = {29},
      YEAR = {1975},
     PAGES = {736--740},
      ISSN = {0378-4754},
   MRCLASS = {65D15},
  MRNUMBER = {0375740},
MRREVIEWER = {Charles K. Chui},
}

@inproceedings {Schumaker.1979,
    AUTHOR = {Schumaker, Larry L.},
     TITLE = {On the dimension of spaces of piecewise polynomials in two
              variables},
 BOOKTITLE = {Multivariate approximation theory ({P}roc. {C}onf., {M}ath.
              {R}es. {I}nst., {O}berwolfach, 1979)},
    SERIES = {Internat. Ser. Numer. Math.},
    VOLUME = {51},
     PAGES = {396--412},
 PUBLISHER = {Birkh\"{a}user, Basel-Boston, Mass.},
      YEAR = {1979},
   MRCLASS = {41A15 (41A63)},
  MRNUMBER = {560683},
MRREVIEWER = {P. Lipow},
}

@incollection {Schumaker.1984,
    AUTHOR = {Schumaker, Larry L.},
     TITLE = {Bounds on the dimension of spaces of multivariate piecewise
              polynomials},
      NOTE = {Surfaces (Stanford, Calif., 1982)},
   JOURNAL = {Rocky Mountain J. Math.},
  FJOURNAL = {The Rocky Mountain Journal of Mathematics},
    VOLUME = {14},
      YEAR = {1984},
    NUMBER = {1},
     PAGES = {251--264},
      ISSN = {0035-7596,1945-3795},
   MRCLASS = {41A99 (26C99)},
  MRNUMBER = {736177},
MRREVIEWER = {J.\ Albrycht},
       DOI = {10.1216/RMJ-1984-14-1-251},
       URL = {https://doi.org/10.1216/RMJ-1984-14-1-251},
}

@article {AMT.2021,
    AUTHOR = {Anderson, Portia and Matherne, Jacob P. and Tymoczko,
              Julianna},
     TITLE = {Generalized splines on graphs with two labels and polynomial
              splines on cycles},
   JOURNAL = {Electron. J. Combin.},
  FJOURNAL = {Electronic Journal of Combinatorics},
    VOLUME = {31},
      YEAR = {2024},
    NUMBER = {1},
     PAGES = {Paper No. 1.29, 52},
   MRCLASS = {05C25 (05C78 05E16 41A15)},
  MRNUMBER = {4695546},
       DOI = {10.37236/11155},
       URL = {https://doi.org/10.37236/11155},
}

@article {CW.1983,
    AUTHOR = {Chui, Charles K. and Wang, Ren Hong},
     TITLE = {On smooth multivariate spline functions},
   JOURNAL = {Math. Comp.},
  FJOURNAL = {Mathematics of Computation},
    VOLUME = {41},
      YEAR = {1983},
    NUMBER = {163},
     PAGES = {131--142},
      ISSN = {0025-5718},
   MRCLASS = {41A15 (41A63)},
  MRNUMBER = {701629},
MRREVIEWER = {Klaus H\"{o}llig},
       DOI = {10.2307/2007771},
       URL = {https://doi.org/10.2307/2007771},
}

@article {Wang.1975,
    AUTHOR = {Wang, Ren Hong},
     TITLE = {Structure of multivariate splines, and interpolation},
   JOURNAL = {Acta Math. Sinica},
  FJOURNAL = {Acta Mathematica Sinica. Shuxue Xuebao},
    VOLUME = {18},
      YEAR = {1975},
    NUMBER = {2},
     PAGES = {91--106},
      ISSN = {0583-1431},
   MRCLASS = {41A15},
  MRNUMBER = {454458},
MRREVIEWER = {Charles K. Chui},
}

@article {WW.1987,
    AUTHOR = {White, Neil and Whiteley, Walter},
     TITLE = {The algebraic geometry of motions of bar-and-body frameworks},
   JOURNAL = {SIAM J. Algebraic Discrete Methods},
  FJOURNAL = {Society for Industrial and Applied Mathematics. Journal on
              Algebraic and Discrete Methods},
    VOLUME = {8},
      YEAR = {1987},
    NUMBER = {1},
     PAGES = {1--32},
      ISSN = {0196-5212},
   MRCLASS = {52A37 (51K99 70B15 73K99)},
  MRNUMBER = {872054},
MRREVIEWER = {Robert Connelly},
       DOI = {10.1137/0608001},
       URL = {https://doi.org/10.1137/0608001},
}

@article {WW.1983,
    AUTHOR = {White, Neil L. and Whiteley, Walter},
     TITLE = {The algebraic geometry of stresses in frameworks},
   JOURNAL = {SIAM J. Algebraic Discrete Methods},
  FJOURNAL = {Society for Industrial and Applied Mathematics. Journal on
              Algebraic and Discrete Methods},
    VOLUME = {4},
      YEAR = {1983},
    NUMBER = {4},
     PAGES = {481--511},
      ISSN = {0196-5212},
   MRCLASS = {52A25 (51K05)},
  MRNUMBER = {721619},
       DOI = {10.1137/0604049},
       URL = {https://doi.org/10.1137/0604049},
}

@unpublished{Wh.unpublished,
    AUTHOR = {Whiteley, Walter},
     TITLE = {The geometry of bivariate ${C}^1_2$-splines},
     YEAR = {1990},
}
\bibliographystyle{amsalpha}

\end{document}